\documentclass[11pt,reqno]{amsart}
\usepackage{mathrsfs}
\usepackage{url}
\usepackage{mathtools}
\usepackage{latexsym,epsfig,amssymb,amsmath,amsthm,color,url,bm}
\usepackage[inline,shortlabels]{enumitem}
\usepackage{hyperref}
\usepackage[foot]{amsaddr}
\usepackage{amsmath,amsbsy}
\usepackage{mwe}
\RequirePackage[numbers]{natbib}
\usepackage{mathptmx}
\usepackage[text={16cm,24cm}]{geometry}

\allowdisplaybreaks 
 \numberwithin{equation}{section}
\newtheorem{theorem}{Theorem}[section]

\newtheorem{lemma}[theorem]{Lemma}
\newtheorem{corollary}[theorem]{Corollary}

\newcommand\Item[1][]{%
  \ifx\relax#1\relax  \item \else \item[#1] \fi
  \abovedisplayskip=0pt\abovedisplayshortskip=0pt~\vspace*{-\baselineskip}}

\theoremstyle{definition}

\theoremstyle{definition}
\newtheorem{remark}[theorem]{Remark}

\DeclareMathOperator{\Prob}{\mathbf{P}}
\DeclareMathOperator{\E}{\mathbf{E}}

\DeclareMathOperator{\NW}{NW}
\DeclareMathOperator{\NL}{NL}
\DeclareMathOperator{\ND}{ND}
\DeclareMathOperator{\nw}{nw}
\DeclareMathOperator{\nl}{n\ell}

\DeclareMathOperator{\MW}{MW}
\DeclareMathOperator{\ML}{ML}
\DeclareMathOperator{\MD}{MD}
\DeclareMathOperator{\mw}{mw}
\DeclareMathOperator{\ml}{m\ell}

\DeclareMathOperator{\nd}{nd}
\DeclareMathOperator{\md}{md}

\title[Several-generation-jump combinatorial games]{Combinatorial games on Galton-Watson trees involving several-generation-jump moves}
\date{}
\author{Dhruv Bhasin, Moumanti Podder}
\address{Moumanti Podder, Indian Institute of Science Education and Research (IISER) Pune, Dr.\ Homi Bhabha Road, Pashan, Pune 411008, Maharashtra, India.}
\address{Dhruv Bhasin, Indian Institute of Science Education and Research (IISER) Pune, Dr.\ Homi Bhabha Road, Pashan, Pune 411008, Maharashtra, India.}
\email{moumanti@iiserpune.ac.in}
\email{bhasin.dhruv@students.iiserpune.ac.in}

\begin{document}
\bibliographystyle{plainnat}

\begin{abstract}
We study the \emph{$k$-jump normal} and \emph{$k$-jump mis\`{e}re} games on rooted Galton-Watson trees, expressing the probabilities of various possible outcomes of these games as specific fixed points of functions that depend on $k$ and the offspring distribution. We discuss phase transition results pertaining to draw probabilities when the offspring distribution is Poisson$(\lambda)$. We compare the probabilities of various outcomes of the $2$-jump normal game with those of the $2$-jump mis\`{e}re game, and a similar comparison is drawn between the $2$-jump normal game and the $1$-jump normal game, under the Poisson regime. We describe the rate of decay of the probability that the first player loses the $2$-jump normal game as $\lambda \rightarrow \infty$. We also discuss a sufficient condition for the average duration of the $k$-jump normal game to be finite.
\end{abstract}

\subjclass[2020]{}

\keywords{two-player combinatorial games; normal games; mis\`{e}re games; rooted Galton-Watson trees; fixed points; Poisson offspring; generalized finite state tree automata}

\maketitle

\section{Introduction}\label{sec:intro}
The simplest yet intriguing versions of the normal and mis\`{e}re games on rooted random trees were studied in \cite{holroyd_martin}. Each game involves two players (henceforth addressed as P1 and P2) and a token, and requires visualizing a given rooted tree as a directed graph in which an edge between a parent vertex $u$ and its child $v$ is assumed to be directed from $u$ to $v$. In each game, the players take turns to move the token along these directed edges. In a normal game, the first player to get stuck at a leaf vertex (i.e.\ unable to move the token any further) loses, whereas in a mis\`{e}re game, this same player wins. The two games share a fair amount of similarities in their analysis, but the objective of each player in a normal game is precisely the opposite of that in a mis\'{e}re game. While the authors of \cite{holroyd_martin} provide an incredibly thorough analysis of these two games when played on rooted Galton-Watson (henceforth abbreviated as GW) trees, they also pose several open questions. Our paper delves deeper into the fascinating world of these two-player combinatorial games and asks: \emph{what if we do not restrict the players to only single-generation moves?} In other words, instead of allowing each player, when it is her turn, to move the token from its current position $u$ to a child of $u$, we now permit her to move the token from $u$ to any descendant of $u$ that is at a distance at most $k$ away from $u$, where $k \geqslant 1$ is a pre-assigned positive integer. We call the corresponding versions of the normal and mis\'{e}re games the \emph{$k$-jump normal} and \emph{$k$-jump mis\`{e}re} games respectively. 


\subsection{Introduction to the games}\label{subsec:intro} We begin with a description of the rooted GW trees on which our games are played. A Galton-Watson branching process (henceforth denoted $\mathcal{T}_{\chi}$), introduced in \cite{galton_watson} and independently studied in \cite{bienayme} as a model to investigate the extinction of ancestral family names, begins with a root $\phi$ giving birth to a random number $X_{0}$ of children where $X_{0}$ follows the \emph{offspring distribution} $\chi$ (a probability distribution supported on the set $\mathbb{N}_{0}$ of non-negative integers). If $X_{0} = 0$, we stop the process, whereas if $X_{0} = m$ for some $m \in \mathbb{N}$, the children of $\phi$ are named $v_{1}, \ldots, v_{m}$, and $v_{i}$ gives birth to $X_{i}$ children with $X_{1}, \ldots, X_{m}$ i.i.d.\ $\chi$. Thus the process continues, and the resulting tree is infinite with positive probability iff the expectation of $\chi$ exceeds $1$. We refer the reader to \cite{athreya_vidyashankar}, \cite{athreya_jagers} and \cite{athreya_ney} for further reading on GW trees. 

We now come to a formal description of the games studied in this paper. Given any realization $T$ of $\mathcal{T}_{\chi}$, any vertex $u$ in $T$, and $i \in \mathbb{N}$, let $\Gamma_{i}(u)$ denote the set of all descendants $v$ of $u$ (excluding $u$ itself) such that the distance between $u$ and $v$ is at most $i$. The vertex at which the token is placed at the beginning of a game is known as the \emph{initial vertex}. The players P1 and P2 take turns to make \emph{moves} (with P1 moving in the first round), where a move constitutes relocating the token from its current position, which is some vertex $u$ in $T$, to a vertex $v \in \Gamma_{k}(u)$, where $k \in \mathbb{N}$ is fixed \emph{a priori}. The player who is unable to make a move loses the $k$-jump normal game. Hence, in this game, each of P1 and P2 strives to relocate the token, obeying the rules of the game, from its current position to a leaf vertex of $T$, thereby making her opponent lose in the next round. On the other hand, the player who is unable to make a move wins the $k$-jump mis\`{e}re game. Therefore, in this game, each of P1 and P2 strives to force her opponent to relocate the token to a leaf vertex of $T$, thereby ensuring that she herself wins the game in the next round.

It is important to note here that a realization $T$ of the random tree $\mathcal{T}_{\chi}$ is first generated and then revealed \emph{in its entirety} to \emph{both} P1 and P2, \emph{before} the game begins. These games are thus \emph{complete information} games. We also assume that P1 and P2 are both intelligent agents who play \emph{optimally}, i.e.\ when a game is destined to end in a decision, the player who wins tries to end the game in as few rounds as possible, while the player who loses tries to prolong the game as much as possible. 

\subsection{Motivations for studying these games}\label{subsec:motive+aspects}The primary motivation for studying these games stems from our interest in examining how allowing each player more room to maneuver in each round ends up affecting the probability of each possible outcome (these outcomes being a win for P1, a loss for P1, and a draw for both players). It is also imperative that we view the $k$-jump versions of the games as broad generalizations of the versions studied in \cite{holroyd_martin}. 

It turns out that it is rather hard to draw a direct, analytical comparison between the ($1$-jump) normal and mis\`{e}re games studied in \cite{holroyd_martin} on one hand and the corresponding $k$-jump versions (for $k \geqslant 2$) on the other, even though our intuitions may suggest otherwise. Almost all such questions remain open and unexplored. For the commonly studied regime where the offspring distribution $\chi$ of the GW tree is Poisson$(\lambda$), we analytically compare the $1$-jump normal game with the $2$-jump normal game for sufficiently large values of $\lambda$ in Theorem~\ref{main:compare_1_vs_2}. Further comparisons can be drawn visually by plotting the curves corresponding to $\nl_{1}$, $\nl_{2}$ and $\nl_{3}$, the curves corresponding to $\nw_{1}$, $\nw_{2}$ and $\nw_{3}$, and the curves corresponding to $\nd_{1}$, $\nd_{2}$ and $\nd_{3}$, as functions of $\lambda$, when $\chi$ is Poisson$(\lambda)$ (see Figures~\ref{nl_comparison}, \ref{nw_comparison} and \ref{nd_comparison}). Here $\nl_{k}$, $\nw_{k}$ and $\nd_{k}$ respectively denote the probabilities of P1's loss, P1's win, and a draw in the $k$-jump normal game (see \S\ref{subsec:notations} for detailed definitions).
\begin{figure}[h!]
  \centering
    \includegraphics[width=0.5\textwidth]{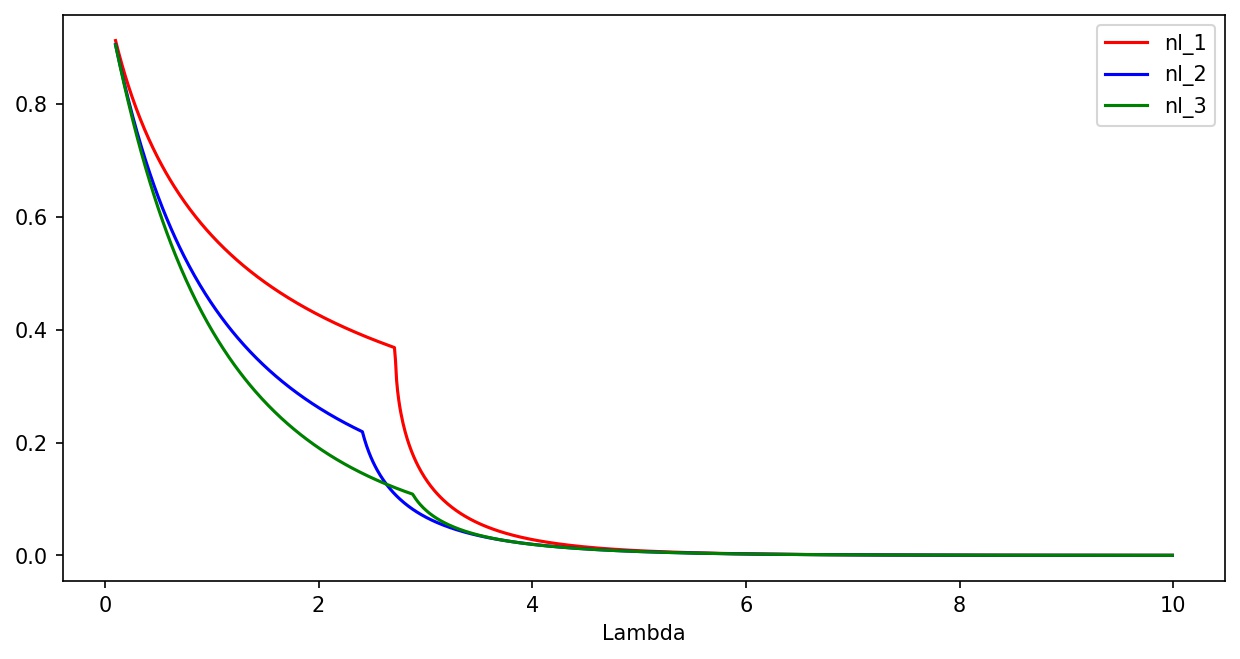}
\caption{Comparing probabilities $\nl_{k}$ of P1 losing as functions of $\lambda$, for $k = 1, 2, 3$}
  \label{nl_comparison}
\end{figure}
\begin{figure}
    \centering
    \begin{minipage}{0.5\textwidth}
        \centering
        \includegraphics[width=1.0\textwidth]{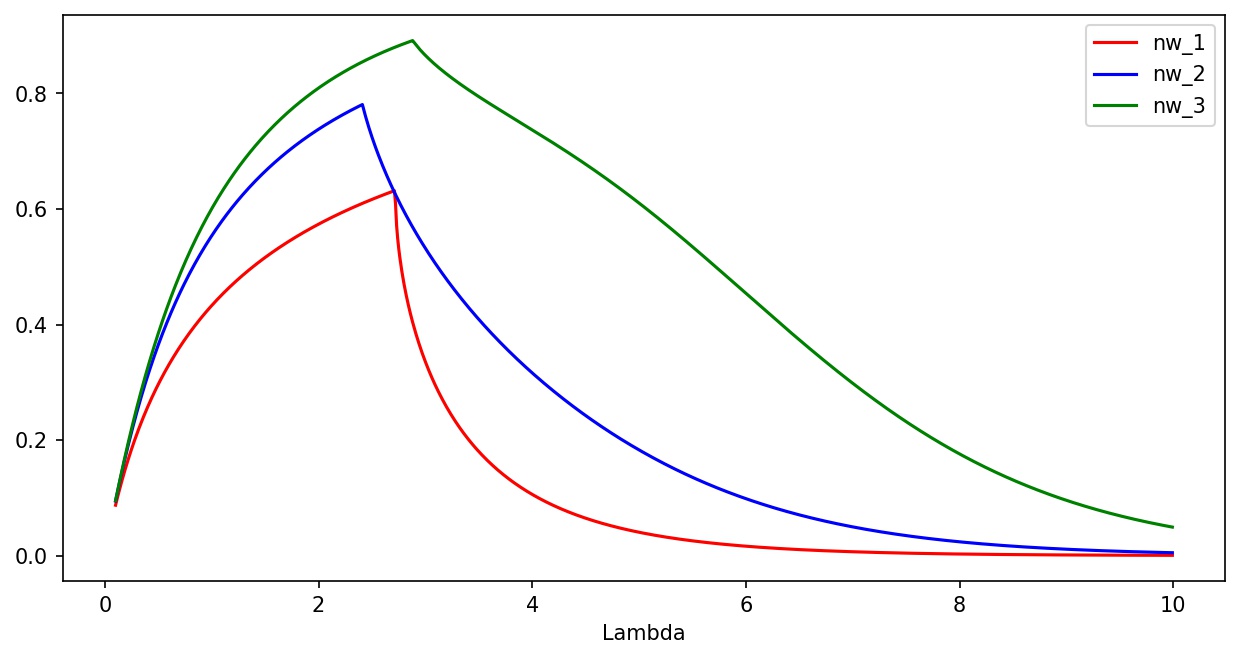} 
        \caption{Comparing probabilities $\nw_{k}$ of P1 winning as functions of $\lambda$, for $k = 1, 2, 3$}
\label{nw_comparison}
    \end{minipage}\hfill
    \begin{minipage}{0.5\textwidth}
        \centering
        \includegraphics[width=1.0\textwidth]{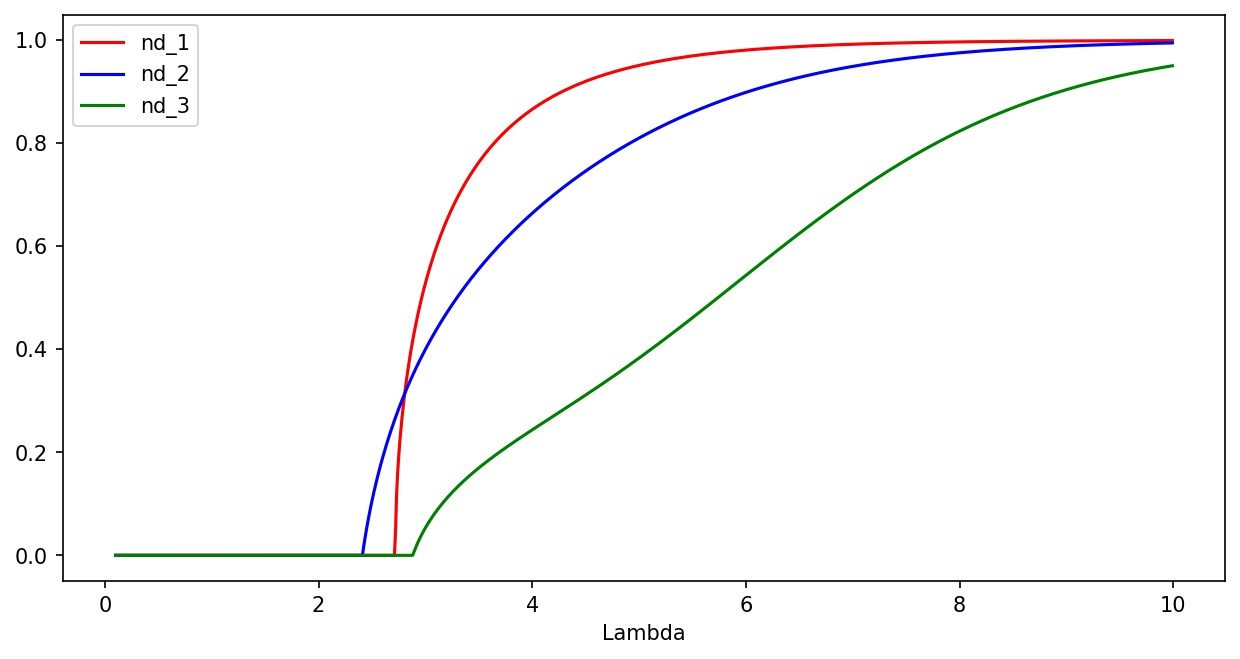} 
        \caption{Comparing probabilities $\nd_{k}$ of draw as functions of $\lambda$, for $k = 1, 2, 3$}
\label{nd_comparison}
    \end{minipage}
\end{figure}

While direct comparisons seem difficult to deduce analytically for $k \geqslant 3$, it is illuminating to explore the many characteristics of the functions $H_{k}$ (see Theorem~\ref{thm:main_1_normal}) and $J_{k}$ (see Theorem~\ref{thm:main_1_misere}) whose minimum positive fixed points equal the probabilities $\nl_{k}$ and $\ml_{k}$ of P1 losing the $k$-jump normal game and the $k$-jump mis\`{e}re game respectively. It is instructive to examine how these functions behave for various values of $k \in \mathbb{N}$. Analyzing these functions is the key to understanding the probabilities of the various possible outcomes of the $k$-jump games, and comparing and contrasting $H_{k}$ (respectively $J_{k}$) for different values of $k$ is instrumental in comparing and contrasting the games themselves. 

For a quick appraisal of how $H_{k}$ (respectively $J_{k}$) behaves as we vary $k$, and the pivotal roles they play in determining various characteristics of the probabilities $\nl_{k}$, $\nw_{k}$, $\nd_{k}$ (respectively $\ml_{k}$, $\mw_{k}$, $\md_{k}$), we mention here some of our findings that have been described in detail later on in the paper. In \S\ref{subsec:thm:main_1_normal_proof_part_2}, we show that $H_{k}$ (and likewise, $J_{k}$, as mentioned in \S\ref{sec:thm_1_misere_proof}) is increasing on $[0,c_{k-1}] \subset [0,1]$, where $\{[0,c_{k}]\}_{k}$ forms a sequence of steadily shrinking intervals (see Lemma~\ref{lem:main_thm_1_1}). The proof is far more involved for higher values of $k$ than when we consider $k=1$ (in fact, for $k = 1$, we have $H_{1}$ and $J_{1}$ defined and increasing on the entire interval $[0,1]$). In \S\ref{sec:Poisson_k=2}, it takes considerable work to show that when the offspring distribution $\chi$ is Poisson$(\lambda)$ with $\lambda \geqslant 2$, the function $H_{2}$ is strictly convex on the interval $[0,c_{2}]$. Plotting $H_{2}$ for various values of $\lambda$ (see, for example, Figure~\ref{H_{2}_between_c_{2}_c_{1}_plot}, where $\lambda = 5$ has been considered) seems to suggest that $H_{2}$ is, in fact, \emph{not} convex on the interval $(c_{2},c_{1}]$. When $\chi$ is Poisson$(\lambda)$, the second assertion of Theorem~\ref{main:normal_draw_probab_limit_Poisson} sheds light on the decay rate of $\nl_{k}$ for \emph{all} values of $k$ as $\lambda \rightarrow \infty$, whereas Theorem~\ref{lem:decay_rate_nl_{2,lambda}} provides an even stronger result on the decay rate of $\nl_{k}$ when $k=2$. It is through a careful analysis of the derivative of $H_{k}$ at the point $c_{k}$ that we obtain Theorems~\ref{main:normal_draw_probab_limit_Poisson} and \ref{main:Poisson_k=2_normal_phase_transition}, both of which shed light on the \emph{phase transition} phenomenon pertaining to the draw probability $\nd_{k}$ when $\chi$ is Poisson$(\lambda)$, i.e.\ how the value of $\nd_{k}$ evolves from $0$ to strictly positive as $\lambda$ is increased gradually. The magnitude of $H'_{k}(c_{k})$ also plays a role, in Theorem~\ref{main:avg_dur}, in determining if the expected duration of a $k$-jump normal game is finite. Finally, comparing $H_{2}$ with $J_{2}$ enables us to compare the probabilities of the various outcomes of the $2$-jump normal game with those of the $2$-jump mis\`{e}re game in Theorem~\ref{main:misere_normal_comparison}, while comparing $H_{2}$ with $H_{1}$ allows a similar comparison between the $1$-jump normal game and the $2$-jump normal game in Theorem~\ref{main:compare_1_vs_2}.
\begin{figure}[h!]
  \centering
    \includegraphics[width=0.5\textwidth]{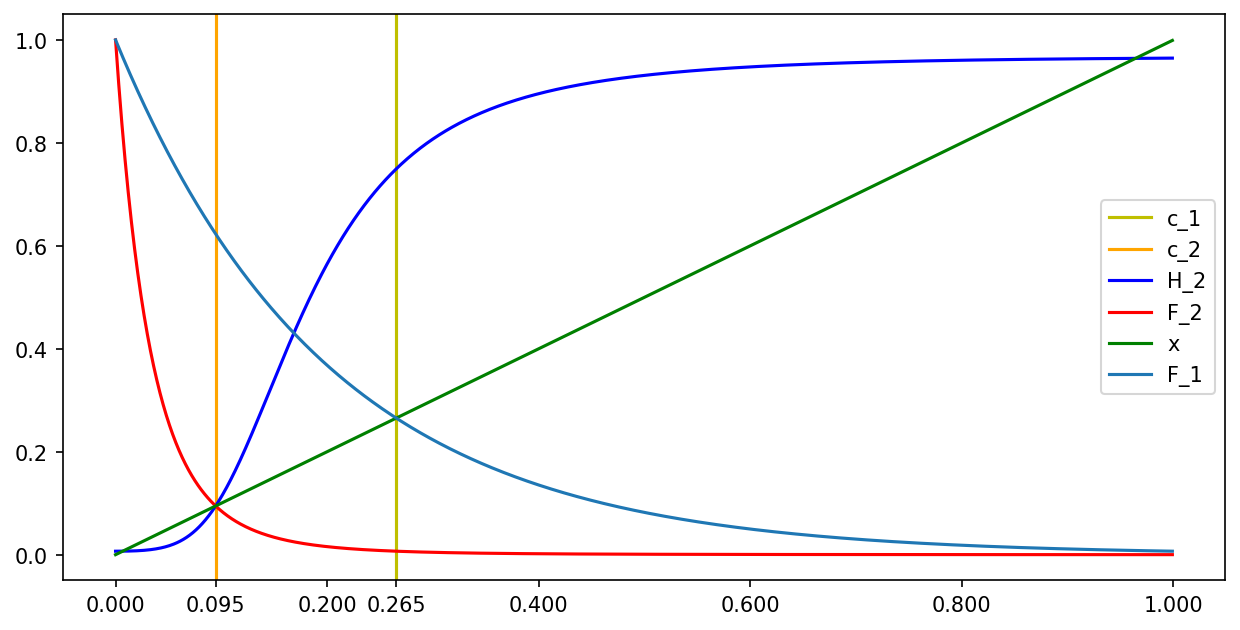}
\caption{The function $H_{2}$ is \emph{not} convex between $c_{2} \approx 0.095$ and $c_{1} \approx 0.265$, when $\chi$ is Poisson$(5)$}
  \label{H_{2}_between_c_{2}_c_{1}_plot}
\end{figure}

A second, and perhaps equally compelling, motivation arises from viewing these games as tools for understanding a generalized notion of \emph{finite state tree automata} (henceforth abbreviated as FSTA). As such, the simplest FSTA is a state machine that comprises a \emph{finite} set $\Sigma = \{1, 2, \ldots, r\}$ of \emph{states} or \emph{colours}, and a \emph{rule} $f_{1}$, which is a function $:\mathbb{N}_{0}^{r} \rightarrow \Sigma$, such that if a vertex $v$ in a rooted tree has $n_{i}$ children that are in state $i$ for each $i \in \Sigma$, then the state of $v$ is given by $f_{1}(n_{1}, n_{2}, \ldots, n_{r})$. Given a fixed $k \in \mathbb{N}$, a \emph{generalized depth-$k$ FSTA} (henceforth abbreviated as a $k$-GFSTA), with rule $f_{k}$, extends the above definition as follows: given a rooted tree $T$, a vertex $v$ of $T$ and an assignment $\sigma: \Gamma_{k}(v) \rightarrow \Sigma$ of states to the vertices of $\Gamma_{k}(v)$ (recall from above that $\Gamma_{k}(v)$ is the set of all descendants of $v$, other than $v$ itself, that are at distance at most $k$ away from $v$), the state of $v$, as dictated by this $k$-GFSTA, is given by $f_{k}\big((\sigma(w): w \in \Gamma_{k}(v))\big)$. For instance, a rule $f_{k}:\mathbb{N}_{0}^{r} \rightarrow \Sigma$ may be considered such that if there are $n_{i}$ vertices in $\Gamma_{k}(v)$ that are in state $i$ for each $i \in \Sigma$, then the state of $v$ is given by $f_{k}(n_{1}, n_{2}, \ldots, n_{r})$. Denoting by $\NW_{k}$ the set of vertices $v$ such that P1 wins the $k$-jump normal game that starts at $v$, by $\NL_{k}$ the set of vertices $v$ such that P1 loses the $k$-jump normal game that starts at $v$, and by $\ND_{k}$ the set of vertices $v$ such that the $k$-jump normal game that starts at $v$ ends in a draw (see also \S\ref{subsec:notations} for these definitions), we obtain a $k$-GFSTA with the state space $\Sigma = \{\NW_{k}, \NL_{k}, \ND_{k}\}$ and the rule $f_{k}$ defined by \eqref{normal_main_recursion_1} and \eqref{normal_main_recursion_2}. Yet another $k$-GFSTA is obtained from the recurrence relations \eqref{misere_main_recursion_1} and \eqref{misere_main_recursion_2} arising from the $k$-jump mis\`{e}re game. It is evident that studying these games may pave the way for a deeper understanding of $k$-GFSTAs for large values of $k$, the associated \emph{recursive distributional equations} and their \emph{fixed points}.

A fixed point of an FSTA is a probability distribution $\nu$ on the state space $\Sigma$ such that if the children of the root $\phi$ of a GW tree $\mathcal{T}_{\chi}$ are assigned i.i.d.\ states from $\Sigma$, each following the common law $\nu$, then the \emph{induced} (random) state (via the rule $f_{1}$) at $\phi$ also follows the law $\nu$. Let $\mathcal{T}$ be the set of \emph{all} possible rooted trees, and let, for any vertex $v$ in a rooted tree $T$, $T(v)$ denote the subtree that comprises $v$ and all its descendants. A map $\iota: \mathcal{T} \rightarrow \Sigma$ is called an \emph{interpretation} of the FSTA if assigning the state or colour $\iota(T(v))$ to each vertex $v$ of \emph{any} arbitrary rooted tree $T$ gives us a colouring of the \emph{entire} tree $T$ that is consistent with the rule $f_{1}$ of the FSTA. We call a fixed point $\nu$ of the FSTA \emph{interpretable} if there exists an interpretation $\iota$ of the FSTA with $\iota(\mathcal{T}_{\chi})$ following the law $\nu$. Necessary and sufficient conditions for fixed points of a certain class of FSTAs to be interpretable were addressed in \cite{podder_4}. We mention here that the notion of interpretability ties in closely with the concept of \emph{endogeny} (see, for instance, \S 2.4, and in particular, Definition 7, of \cite{aldous2005survey}, as well as \cite{mach2018new}, \cite{mach2020recursive}, \cite{martin2020minimax}, \cite{rath2021frozen}, \cite{rath2022phase}).

Let us understand how one may extend the above-mentioned notion of fixed points to the case of the $2$-GFSTA obtained from the $2$-jump normal game (with its rule $f_{2}$ given by \eqref{normal_main_recursion_1} and \eqref{normal_main_recursion_2} for $k=2$). Let $\mathcal{C}_{0,1}$ denote the set of all vertices $v$ such that $v$ has at least one child whose state is $\NL_{2}$, $\mathcal{C}_{0,2}$ the set of all vertices $v$ such that $v$ has no child in state $\NL_{2}$ but at least one grandchild in state $\NL_{2}$, and $\mathcal{C}_{1,2}$ the set of all vertices $v$ such that every child of $v$ is in state $\NW_{2}$ and at least one grandchild of $v$ is in state $\NL_{2}$. It is immediate from these definitions that $\mathcal{C}_{1,2} \subset \mathcal{C}_{0,2}$. We mention here that these subsets are analogous to those defined in \eqref{C_{i,j,n}_defn} (with the superscripts $n$, $n+1$ etc.\ removed). 

We now describe the recurrence relations (these have been elaborated upon in \S\ref{sec:thm_1_normal_proof} for the general case of \emph{any} $k \in \mathbb{N}$) that tie the above-mentioned subsets together. A vertex $v$ is in $\NW_{2}$ if and only if it is either in $\mathcal{C}_{0,1}$ or $\mathcal{C}_{0,2}$, which is equivalent to saying that $v$ either has a child that is in $\NL_{2}$ or a child that is in $\mathcal{C}_{0,1}$. A vertex $v$ is in $\NL_{2}$ if and only if either $v$ is childless or every child of $v$ is in $\mathcal{C}_{1,2}$. In all other situations, $v$ is in $\ND_{2}$. We further note that $v$ is in $\mathcal{C}_{0,1}$ if and only if at least one child of $v$ belongs to $\NL_{2}$. It is in $\mathcal{C}_{1,2}$ if and only if all its children are in $\NW_{2}$ and at least one of its children is in $\mathcal{C}_{0,1}$, which is equivalent to saying that all children of $v$ are in $\mathcal{C}_{0,1} \cup \mathcal{C}_{0,2}$ and at least one of them is in $\mathcal{C}_{0,1}$. Finally, $v$ is in $\mathcal{C}_{0,2}$ if none of its children is in $\NL_{2}$ but at least one of them is in $\mathcal{C}_{0,1}$.

We assign i.i.d.\ states from the state space $\Sigma = \{\mathcal{C}_{0,1}, \mathcal{C}_{1,2}, \mathcal{C}_{0,2} \setminus \mathcal{C}_{1,2}, \NL_{2}, \ND_{2}\}$ to the children of the root $\phi$ of the GW tree $\mathcal{T}_{\chi}$ according to the common law $\nu$, where we set $p_{0,1} = \nu[\mathcal{C}_{0,1}]$, $p_{1,2} = \nu[\mathcal{C}_{1,2}]$, $p_{0,2} = \nu[\mathcal{C}_{0,2}]$ (so that $\nu[\mathcal{C}_{0,2} \setminus \mathcal{C}_{1,2}] = p_{0,2} - p_{1,2}$), $\nl_{2} = \nu[\NL_{2}]$ and $\nd_{2} = \nu[\ND_{2}]$. For $\nu$ to be a fixed point of this $2$-GFSTA, the random state induced at $\phi$ must follow the law $\nu$ as well. This requires that the following equations, derived from the above-mentioned recurrence relations, hold:
\begin{align}
&\nl_{2} = \sum_{m=0}^{\infty}p_{1,2}^{m}\chi(m) = G(p_{1,2}),\label{eq_1}\\
&p_{0,1} = \sum_{m=1}^{\infty}\left[1 - \left(1-\nl_{2}\right)^{m}\right]\chi(m) = 1 - G\left(1-\nl_{2}\right),\label{eq_2}\\
&p_{1,2} = \sum_{m=1}^{\infty}\left[\left(p_{0,1}+p_{0,2}\right)^{m} - p_{0,2}^{m}\right]\chi(m) = G\left(p_{0,1}+p_{0,2}\right) - G\left(p_{0,2}\right),\label{eq_3}\\
&p_{0,2} = \sum_{m=1}^{\infty}\left[\left(1-\nl_{2}\right)^{m} - \left(1-\nl_{2}-p_{0,1}\right)^{m}\right]\chi(m) = G\left(1-\nl_{2}\right) - G\left(1-\nl_{2}-p_{0,1}\right),\label{eq_4}
\end{align}
where $\chi(m)$ denotes the probability of $\phi$ having $m$ children, for $m \in \mathbb{N}_{0}$, and $G$ indicates the probability generating function corresponding to $\chi$. The above equations are used to solve for the fixed point $\nu$. Note that we do not need a separate equation for $\nd_{2}$ since $\nd_{2} = 1 - \nl_{2} - p_{0,1} - p_{0,2}$.

It is obvious that the sets $\mathcal{C}_{0,1}$, $\mathcal{C}_{1,2}$, $\mathcal{C}_{0,2} \setminus \mathcal{C}_{1,2}$, $\NL_{2}$ and $\ND_{2}$ that arise out of the $2$-jump normal game itself provide an interpretation $\iota$ of the above $2$-GFSTA with $\iota(\mathcal{T}_{\chi})$ following the law $\nu$. But there are questions pertaining to interpretability that remain open when the system of equations constituting \eqref{eq_1}, \eqref{eq_2}, \eqref{eq_3} and \eqref{eq_4} does not yield a unique solution for $\nu$. For instance, Theorem~\ref{thm:main_1_normal} asserts that $\nl_{2}$ is the minimum positive fixed point of $H_{2}$, whereas Corollary~\ref{cor:c_{k}_fixed_point} tells us that $c_{2}$ is also a fixed point of $H_{2}$. From  Theorem~\ref{main:Poisson_k=2_normal_phase_transition}, we see that when $\chi$ is Poisson$(\lambda)$, we have $\nl_{2} < c_{2}$ for all $\lambda > \lambda_{c}$ for $\lambda_{c} \approx 2.41$. So now we ask: is the probability distribution obtained by replacing $\nl_{2}$ by $c_{2}$ (and computing the corresponding values of $p_{0,1}$, $p_{0,2}$, $p_{1,2}$ and $\nd_{2}$ from \eqref{eq_1}, \eqref{eq_2}, \eqref{eq_3} and \eqref{eq_4}) interpretable in the sense described earlier (see the definition for interpretability of fixed points of FSTAs in the previous page)? For higher values of $k$, $H_{k}$ may have several fixed points other than these two, and understanding the interpretability of the corresponding probability distributions is also of interest to us.

A third motivation for investigating these games arises from our speculation that these games may serve as precursors to more complicated versions where, for example, each round involves choosing one of P1 and P2 uniformly randomly and then allowing her to make a move (where a move involves relocating the token from its current vertex to a child of that vertex), but ensuring that no player is chosen for more than $k$ consecutive rounds, where $k$ is a pre-fixed positive integer.

\vspace{-0.1in}\subsection{Notations and some definitions}\label{subsec:notations}
Given a rooted tree $T$, we denote by $V(T)$ its vertex set. As previously mentioned, for $u \in V(T)$, we define $\Gamma_{i}(u)$, for $i \in \mathbb{N}$, as the set of descendants $v$ of $u$ (not including $u$ itself) with $\rho(u,v) \leqslant i$, where $\rho$ denotes the graph metric on $T$. As mentioned above in \S\ref{subsec:motive+aspects}, we shall denote by $G$ the probability generating function (pgf) of the offspring distribution $\chi$ of the GW tree $\mathcal{T}_{\chi}$, i.e.\ $G(x) = \sum_{i=0}^{\infty} x^{i} \chi(i)$ for any $x \in [0,1]$. All offspring distributions $\chi$ considered in this paper satisfy $0 < \chi(0) < 1$.

Given $k \in \mathbb{N}$, we define $\NL_{k}$ (or simply $\NL$ when the value of $k$ is clear from the context) to be the set of all vertices $v \in V(\mathcal{T}_{\chi})$ such that if $v$ is the initial vertex, then P1, who plays the first round, loses the $k$-jump normal game. Likewise, let $\NW_{k}$ (or simply $\NW$) denote the set of all $v \in V(\mathcal{T}_{\chi})$ such that if $v$ is the initial vertex, then P1, playing the first round, wins the $k$-jump normal game. Let $\ND_{k}$ (or simply $\ND$) denote the set of all $v \in V(\mathcal{T}_{\chi})$ such that if $v$ is the initial vertex, the $k$-jump normal game ends in a draw. 

It is important to consider refinements of the above subsets of vertices in order to understand the probabilities of the game's outcomes better. For every $n \in \mathbb{N}$, we define $\NL_{k}^{(n)}$ (or $\NL^{(n)}$ when the value of $k$ is unambiguous from the context) to be the subset of $\NL_{k}$ comprising vertices $v$ such that if $v$ is the initial vertex, the $k$-jump normal game lasts for less than $n$ rounds. Likewise, $\NW_{k}^{(n)}$ (or simply $\NW^{(n)}$) is the subset of $\NW_{k}$ comprising vertices $v$ such that the $k$-jump normal game starting at $v$ ends in less than $n$ rounds. We define $\ND^{(n)} = \ND_{k}^{(n)} = V(\mathcal{T}_{\chi}) \setminus \left[\NL_{k}^{(n)} \cup \NW_{k}^{(n)}\right]$, or in other words, $v \in \ND_{k}^{(n)}$ iff the $k$-jump normal game starting at $v$ lasts for at least $n$ rounds. We set $\NL_{k}^{(0)} = \NW_{k}^{(0)} = \emptyset$. By definition, we have $\NL_{k}^{(n)} \subset \NL_{k}^{(n+1)} \text{ and } \NW_{k}^{(n)} \subset \NW_{k}^{(n+1)}$ for all $n \in \mathbb{N}_{0}$.

We define $\nl_{k}$ to be the probability of the event that the root $\phi$ of $\mathcal{T}_{\chi}$ belongs to $\NL_{k}$, whereas for each $n \in \mathbb{N}$, we define $\nl_{k}^{(n)}$ to be the probability of the event that $\phi$ belongs to $\NL_{k}^{(n)}$. Likewise, we define $\nw_{k}$, $\nw_{k}^{(n)}$, $\nd_{k}$ and $\nd_{k}^{(n)}$. From above, we have $\nl_{k}^{(0)} = \nw_{k}^{(0)} = 0$. Once again, the subscript $k$ is dropped whenever its value is clear from the context. From the last line of the previous paragraph, we get
\begin{align}\label{increasing_sequence_nl_nw}
\nl_{k}^{(n)} \leqslant \nl_{k}^{(n+1)} \text{ and } \nw_{k}^{(n)} \leqslant \nw_{k}^{(n+1)} \text{ for all } n \in \mathbb{N}_{0}.
\end{align} 
The corresponding subsets for the $k$-jump mis\'{e}re games are denoted by $\ML_{k}$, $\MW_{k}$, $\MD_{k}$, $\ML_{k}^{(n)}$, $\MW_{k}^{(n)}$ and $\MD_{k}^{(n)}$, and the corresponding probabilities by $\ml_{k}$, $\mw_{k}$, $\md_{k}$, $\ml_{k}^{(n)}$, $\mw_{k}^{(n)}$ and $\md_{k}^{(n)}$ (as above, the subscript is removed when the value of $k$ is clear from the context).

\subsection{Main results}\label{sec:main_results}
We begin by introducing a couple of sequences of functions that are defined recursively. The first is $\{F_{i}\}_{i \in \mathbb{N}_{0}}$, where $F_{0}(x) = 1$ for $x \in [0,1]$, and $F_{i}: [0,c_{i-1}] \rightarrow [0,1]$, for all $i \in \mathbb{N}$, is defined as 
\begin{align}\label{F_{i}_defn}
F_{i}(x) = G(F_{i-1}(x) - x), \ x \in [0,c_{i-1}],
\end{align}
where recall that $G$ is the pgf of $\chi$, and $c_{i}$ is the unique (as shown in Lemma~\ref{lem:main_thm_1_1}) fixed point of $F_{i}$. The second sequence of functions $\{g_{i}\}_{i \in \mathbb{N}}$ is defined as follows. The function $g_{1}: \mathbb{R}^{2} \rightarrow \mathbb{R}$ is defined as $g_{1}(x,y) = x-y$, and having defined $g_{i-1}$ for any $i \geqslant 2$, we define $g_{i}: \mathcal{D}_{i} \rightarrow \mathbb{R}$ as
\begin{align}\label{g_{i}_recursive_defn}
g_{i}(x_{0},x_{1},\ldots,x_{i}) = G(g_{i-1}(x_{0},x_{2},\ldots,x_{i})) - G(g_{i-1}(x_{1},x_{2},\ldots,x_{i})),\ (x_{0},x_{1},\ldots,x_{i}) \in \mathcal{D}_{i},
\end{align}
where the sets $\mathcal{D}_{i}$ are also recursively defined, as follows:
\begin{align}\label{D_{i}_domain_defn}
\mathcal{D}_{i} = \{(x_{0},x_{1},\ldots,x_{i}): g_{i-1}(x_{0},x_{2},\ldots,x_{i}) \in [0,1], g_{i-1}(x_{1},x_{2},\ldots,x_{i}) \in [0,1]\}.
\end{align}
As an example, $\mathcal{D}_{2} = \{(x_{0},x_{1},x_{2}): x_{2} \leqslant x_{0} \leqslant x_{2}+1, x_{2} \leqslant x_{1} \leqslant x_{2}+1\}$. The motivation behind defining $\mathcal{D}_{i}$ this way is to simply ensure that the arguments $g_{i-1}(x_{0},x_{2},\ldots,x_{i})$ and $g_{i-1}(x_{1},x_{2},\ldots,x_{i})$ of the function $G$ in \eqref{g_{i}_recursive_defn} belong to the domain $[0,1]$ on which $G$ is defined.

\begin{theorem}\label{thm:main_1_normal}
Consider the $k$-jump normal game for $k \in \mathbb{N}$. Define the function $H_{k}: [0,c_{k-1}] \rightarrow [0,1]$ as
\begin{align}\label{H_{k}_defn}
H_{k}(x) = G(g_{k}(F_{0}(x),F_{1}(x),\ldots,F_{k}(x))).
\end{align}
Then $\nl_{k}$ is the minimum positive fixed point of $H_{k}$. Moreover, $\nw_{k} = 1 - F_{k}(\nl_{k})$. 
\end{theorem}

Some discussions are in order regarding the functions defined above, as they form an integral part of the results in this paper. First, we compare the findings of Theorem~\ref{H_{k}_defn} with those of Theorem 1 (i) of \cite{holroyd_martin}. Let us define the function $R(x) = 1 - G\big(1-G(x)\big)$ for $x \in [0,1]$. According to the notation used in this paper, Theorem 1 (i) of \cite{holroyd_martin} states that $1-\nl_{1}$ equals the maximum fixed point and $\nw_{1}$ the minimum fixed point of $R(x)$ in $[0,1]$. From our definitions of $g_{1}$ and $F_{0}$, \eqref{F_{i}_defn} and \eqref{H_{k}_defn}, we see that $H_{1}(x) = G\big(1 - G(1-x)\big)$. For any $y \in [0,1]$, we observe that $1-y$ is a fixed point of $R(x)$ if and only if
\begin{equation}
1-y = 1 - G\big(1 - G(1-y)\big) \Longleftrightarrow y = G\big(1 - G(1-y)\big) = H_{1}(y),\nonumber
\end{equation}
which is equivalent to $y$ being a fixed point of $H_{1}(x)$. This observation immediately reveals that $\nl_{1}$ is the minimum fixed point of $H_{1}(x)$ in $[0,1]$ if and only if $1-\nl_{1}$ is the maximum fixed point of $R(x)$ in $[0,1]$, i.e.\ our conclusion about $\nl_{1}$ from Theorem~\ref{thm:main_1_normal} matches with what Theorem 1 (i) of \cite{holroyd_martin} yields. Moreover, for every $y \in [0,1]$ that is a fixed point of $H_{1}(x)$, we observe that
\begin{align}
1 - F_{1}(y) = 1 - G(1 - y) = 1 - G\big(1 - H_{1}(y)\big) = R\big(1 - G(1-y)\big) = R(1 - F_{1}(y)),\nonumber
\end{align}
showing that $1-F_{1}(y)$ is a fixed point of $R(x)$ in $[0,1]$. Conversely, under the assumption that $G$ is a strictly increasing function on $[0,1]$ (which is true whenever $\chi(0) < 1$), we observe that $1-F_{1}(y)$ is a fixed point of $R(x)$ in $[0,1]$, for $y \in [0,1]$, if and only if
\begin{align}
& R\big(1 - F_{1}(y)\big) = 1-F_{1}(y) \Longleftrightarrow G\big(1 - G\big(1-F_{1}(y)\big)\big) = F_{1}(y)\nonumber\\
&\Longleftrightarrow G\big(1 - G\big(1-G(1-y)\big)\big) = G(1-y) \Longleftrightarrow 1 - G\big(1-G(1-y)\big) = 1-y \Longleftrightarrow H_{1}(y) = y,\nonumber
\end{align}
thus showing us that $y$ is a fixed point of $H_{1}(x)$ in $[0,1]$. The last two observations tell us that $y$ is a fixed point of $H_{1}(x)$ in $[0,1]$ if and only if $1-F_{1}(y)$ is a fixed point of $R(x)$ in $[0,1]$. Since $F_{1}$ is strictly decreasing on $[0,1]$ and $\nl_{1}$ has already been shown above to be the minimum fixed point of $H_{1}(x)$ in $[0,1]$, this establishes that $1-F_{1}(\nl_{1})$ must be the minimum fixed point of $R(x)$ in $[0,1]$, thus showing that our conclusion about $\nw_{1}$ from Theorem~\ref{thm:main_1_normal} matches with what Theorem 1 (i) of \cite{holroyd_martin} yields. 

We now try to give the reader an idea as to how the functions in \eqref{F_{i}_defn}, \eqref{g_{i}_recursive_defn} and \eqref{H_{k}_defn} come to be defined, without going into the actual technical details (which have been laid out fully in the proof of Theorem~\ref{thm:main_1_normal} in \S\ref{sec:thm_1_normal_proof}), and to this end, we focus on the case of $k=2$. This case has already been discussed, to some extent, in \S\ref{subsec:motive+aspects}. In \S\ref{subsec:motive+aspects}, along with $\NW_{2}$, $\NL_{2}$ and $\ND_{2}$, we defined the subsets $\mathcal{C}_{0,1}$, $\mathcal{C}_{0,2}$ and $\mathcal{C}_{1,2}$, described the recurrence relations that tie these subsets to one another, and derived Equations \eqref{eq_1}, \eqref{eq_2}, \eqref{eq_3} and \eqref{eq_4} that relate the probabilities $\nl_{2}$, $\nd_{2}$, $p_{0,1}$, $p_{0,2}$ and $p_{1,2}$ with each other (where $p_{i,j}$ is the probability that the root of $\mathcal{T}_{\chi}$ belongs to $\mathcal{C}_{i,j}$, for $0 \leqslant i < j \leqslant 2$). 

From the definition of $F_{0}$ and \eqref{F_{i}_defn}, we have $F_{1}(\nl_{2}) = G(1-\nl_{2})$ and $F_{2}(\nl_{2}) = G\big(G(1-\nl_{2})-\nl_{2}\big)$. Using the expression for $F_{1}$ and the definition of $g_{1}$, it becomes immediate from \eqref{eq_2} that $p_{0,1} = g_{1}\big(F_{0}(\nl_{2}), F_{1}(\nl_{2})\big)$. Using the expressions for $F_{1}$ and $F_{2}$ and Equations \eqref{eq_2} and \eqref{eq_4}, we have
\begin{align}
p_{0,2} &= G(1-\nl_{2}) - G(1-p_{0,1}-\nl_{2}) = F_{1}(\nl_{2}) - G(G(1-\nl_{2})-\nl_{2}) \nonumber\\&= F_{1}(\nl_{2}) - F_{2}(\nl_{2}) = g_{1}\big(F_{1}(\nl_{2}),F_{2}(\nl_{2})\big).\nonumber
\end{align}
From \eqref{eq_2}, \eqref{eq_3}, \eqref{eq_4} and \eqref{g_{i}_recursive_defn}, we see that 
\begin{align}
p_{1,2} &= G\big(1-F_{2}(\nl_{2})\big) - G\big(g_{1}\big(F_{1}(\nl_{2}),F_{2}(\nl_{2})\big)\big)\nonumber\\
&= G\big(g_{1}\big(F_{0}(\nl_{2}),F_{2}(\nl_{2})\big)\big) - G\big(g_{1}\big(F_{1}(\nl_{2}),F_{2}(\nl_{2})\big)\big) = g_{2}\big(F_{0}(\nl_{2}),F_{1}(\nl_{2}),F_{2}(\nl_{2})\big).\nonumber
\end{align}
\sloppy These ideas extend to the general case of arbitrary $k \in \mathbb{N}$, with the identity $p_{i,j,n} = g_{i+1}\big(F_{j-i-1}(\nl_{k}), F_{j-i}(\nl_{k}), F_{k-i+1}(\nl_{k}), \ldots, F_{k}(\nl_{k})\big)$ being true for all $0 \leqslant i < j \leqslant k$ (see Lemma~\ref{lem:main_thm_1_3} and its proof for a better understanding of this fact). We now use the expression for $p_{1,2}$ derived above, \eqref{eq_1} and \eqref{H_{k}_defn} to conclude that $\nl_{2}$ is indeed a fixed point of $H_{2}(x)$ in $[0,1]$. Although the proof of Theorem~\ref{thm:main_1_normal} is rather involved when arbitrary values of $k$ are considered, we hope that the above exposition, for $k=2$, helps to shed some light on how our argument proceeds, how the recursive definitions of the functions in \eqref{F_{i}_defn}, \eqref{g_{i}_recursive_defn} and \eqref{H_{k}_defn} arise etc.

In order to state, for the case of $k$-jump mis\`{e}re games, the result that is analogous to Theorem~\ref{thm:main_1_normal}, we introduce yet another sequence of functions $\{\gamma_{i}\}_{i \in \mathbb{N}}$ that bears significant resemblance to \eqref{g_{i}_recursive_defn}. Setting $\gamma_{1} \equiv g_{1}$, for each $i \geqslant 2$ we define $\gamma_{i}: \mathcal{D}'_{i} \rightarrow \mathbb{R}$ recursively as
\begin{align}\label{gamma_{i}_recursive_defn}
\gamma_{i}(x_{0},x_{1}, \ldots, x_{i}) = G\left(\chi(0) + \gamma_{i-1}(x_{0},x_{2}, \ldots, x_{i})\right) - G\left(\chi(0) + \gamma_{i-1}(x_{1},x_{2}, \ldots, x_{i})\right),
\end{align}
\sloppy where the sets $\mathcal{D}'_{i}$ are recursively defined as
$\mathcal{D}'_{i} = \{(x_{0},x_{1},\ldots,x_{i}): \chi(0) + \gamma_{i-1}(x_{0},x_{2}, \ldots, x_{i}) \in [0,1], \chi(0) + \gamma_{i-1}(x_{1},x_{2}, \ldots, x_{i}) \in [0,1]\}$.
\begin{theorem}\label{thm:main_1_misere}
Consider the $k$-jump mis\`{e}re game, $k \in \mathbb{N}$. Define the function $J_{k}: [0,c_{k-1}] \rightarrow [0,1]$ as
\begin{equation}\label{mathcal{J}_{k}_defn}
J_{k}(x) = G\left(\chi(0) + \gamma_{k}(F_{0}(x), F_{1}(x), \ldots, F_{k}(x))\right) - \chi(0).
\end{equation}
Then $\ml_{k}$ is the minimum positive fixed point of $J_{k}$. Moreover, $\mw_{k} = 1 - F_{k}(\ml_{k}) + \chi(0)$.
\end{theorem}
We mention here, for the convenience of the reader, that much of the proof of Theorem~\ref{thm:main_1_misere} unfolds the same way as that of Theorem~\ref{thm:main_1_normal}, and the motivations behind the recursive defintions of the functions in \eqref{gamma_{i}_recursive_defn} and \eqref{mathcal{J}_{k}_defn} are very similar to those behind the recursive definitions of the functions in \eqref{g_{i}_recursive_defn} and \eqref{H_{k}_defn} respectively.

Theorem~\ref{main:bounds_on_nl_{k}_ml_{k}} provides bounds on $\nl_{k}$ and $\ml_{k}$, and necessary and sufficient conditions for the draw probabilities $\nd_{k}$ and $\md_{k}$ to be positive. Recall that $c_{k}$ is the (unique, by Lemma~\ref{lem:main_thm_1_1}) fixed point of $F_{k}$.
\begin{theorem}\label{main:bounds_on_nl_{k}_ml_{k}}
For every $k \in \mathbb{N}$, we have $\chi(0) < \nl_{k} \leqslant c_{k}$ and $\ml_{k} \leqslant \hat{c}_{k}$, where $\hat{c}_{k}$ is the unique point of intersection between $y = F_{k}(x)$ and $y = J_{k}(x)+\chi(0)$ in $(0,c_{k-1})$. Moreover, $\nd_{k} > 0$ if and only if $\nl_{k} < c_{k}$ and $\md_{k} > 0$ if and only if $\ml_{k} < \hat{c}_{k}$. 
\end{theorem}

It is worthwhile to note that when the offspring distribution $\chi$ of $\mathcal{T}_{\chi}$ has expectation bounded above by $1$, $\mathcal{T}_{\chi}$ is finite almost surely, which, for any fixed $k$, forces the $k$-jump normal game starting at the root $\phi$ of $\mathcal{T}_{\chi}$ to end in a finite number of rounds almost surely. Consequently, the probability of draw in such a situation is $0$. In particular, this tells us that when $\chi$ is Poisson$(\lambda)$ for $\lambda \leqslant 1$, the probability that the $k$-jump normal game results in a draw is $0$. Theorem~\ref{main:normal_draw_probab_limit_Poisson} is of an asymptotic nature, asserting that when $\chi$ is Poisson$(\lambda)$, the probability of the $k$-jump normal game ending in a draw eventually becomes strictly positive as we keep increasing $\lambda$. Evidently, this gives rise to a phase transition phenomenon in that, the probability $\nd_{k} = \nd_{k,\lambda}$ of the event that a $k$-jump normal game played on a rooted Galton-Watson tree with Poisson$(\lambda)$ offspring results in a draw goes from being equal to $0$ for $\lambda \leqslant 1$ to being strictly positive for all $\lambda$ large enough. 
\begin{theorem}\label{main:normal_draw_probab_limit_Poisson}
Fix any $k \in \mathbb{N}$. When the offspring distribution $\chi$ is Poisson$(\lambda)$, we have $\nd_{k} = \nd_{k,\lambda} > 0$ for all $\lambda$ sufficiently large. We also have $\lambda^{k-1}\nl_{k} = \lambda^{k-1}\nl_{k,\lambda} \rightarrow 0$ as $\lambda \rightarrow \infty$. 
\end{theorem}
\noindent Theorem~\ref{main:Poisson_k=2_normal_phase_transition} provides a more nuanced insight into the phase transition phenomenon when $k=2$:
\begin{theorem}\label{main:Poisson_k=2_normal_phase_transition}
For $\lambda \geqslant 2$, the function $H_{2} = H_{2,\lambda}$ is strictly convex on the interval $[0,c_{2}] = [0,c_{2,\lambda}]$. The slope of $H_{2,\lambda}$ at $c_{2,\lambda}$ is strictly increasing as a function of $\lambda$, for all $\lambda \geqslant 1$. As a consequence of these two facts, there is a unique \emph{critical} $\lambda_{c} \approx 2.41$ such that for all $2 \leqslant \lambda < \lambda_{c}$ we have $\nd_{2} = \nd_{2,\lambda} = 0$, and for all $\lambda > \lambda_{c}$, we have $\nd_{2,\lambda} > 0$.
\end{theorem}

The next result is an especially strong one as it sheds light on the rate of decay of $\nl_{k,\lambda}$, when $k=2$ and the offspring distribution is Poisson$(\lambda)$, as $\lambda \rightarrow \infty$.
\begin{theorem}\label{lem:decay_rate_nl_{2,lambda}}
We have $\lim_{\lambda \rightarrow \infty} \lambda^{i}\nl_{2,\lambda} = 0$ for \emph{all} $i \in \mathbb{N}$.
\end{theorem}

Theorem~\ref{main:misere_normal_comparison} compares the $2$-jump normal game with the $2$-jump mis\`{e}re game (see Proposition~\ref{prop:ml_2_nl_2_more_refined} for a more precise description of the values of $\lambda$ for which the first of the three inequalities is shown to hold analytically), while Theorem~\ref{main:compare_1_vs_2} compares the $1$-jump normal game with the $2$-jump normal game, when all of these games are played on $\mathcal{T}_{\chi}$ with $\chi$ being Poisson$(\lambda)$. \begin{theorem}\label{main:misere_normal_comparison}
When $k=2$ and $\chi$ is Poisson$(\lambda)$, we have $\ml_{2,\lambda} \leqslant \nl_{2,\lambda}$, $\nd_{2,\lambda} < \md_{2,\lambda}$ and $\ml_{2,\lambda} \leqslant \nl_{2,\lambda} < \nw_{2,\lambda}$ for all $\lambda$ sufficiently large.
\end{theorem}
\begin{remark}\label{rem:ml_{2}_decay_rate}
From Theorems~\ref{lem:decay_rate_nl_{2,lambda}} and \ref{main:misere_normal_comparison}, we conclude that $\lim_{\lambda \rightarrow \infty}\lambda^{i}\ml_{2,\lambda} = 0$ for all $i \in \mathbb{N}$.
\end{remark}
\begin{theorem}\label{main:compare_1_vs_2}
When $k=2$ and $\chi$ is Poisson$(\lambda)$, we have $\nl_{2,\lambda} \leqslant \nl_{1,\lambda}$, $\nd_{2,\lambda} < \nd_{1,\lambda}$ and $\nw_{1,\lambda} < \nw_{2,\lambda}$ for all $\lambda$ sufficiently large.
\end{theorem}

Our final result goes back to general offspring distributions $\chi$, and concerns itself with average durations of $k$-jump normal games. We conjecture, from the patterns noticed in its proof, that the second assertion of Theorem~\ref{main:avg_dur} can be extended to \emph{any} $k \in \mathbb{N}$, though we cannot seem to provide an intuitive argument as to why such a relation should be true.
\begin{theorem}\label{main:avg_dur}
For any fixed $k$, when $\nl_{k} = c_{k}$ and $\max\left\{H'_{k}(c_{k}), \left|F'_{k}(c_{k})\right|\right\} < 1$, the expected duration of the $k$-jump normal game is finite. Moreover, for $k = 2, 3$, if $\nl_{k} = c_{k}$ and $\left|F'_{k}(c_{k})\right| < 1$, then once again, the expected duration is finite.
\end{theorem}
It is worthwhile to note here that while the condition $\nl_{k} = c_{k}$ alone does guarantee that $\nd_{k} = 0$ (by Theorem~\ref{main:bounds_on_nl_{k}_ml_{k}}) and hence the $k$-jump normal game ends almost surely in a finite number of rounds, it does not automatically imply that the expected number of rounds is going to be finite as well.

\subsection{A brief discussion of the literature on combinatorial games}\label{subsec:lit_surv} Before we plunge into our exploration of the $k$-jump normal and $k$-jump mis\`{e}re games, the rich and variegated literature on combinatorial games that has developed over the past several decades deserves some delineation. This extremely broad class of games (see, for example, \cite{survey_games} and \cite{complexity_appeal} for a general introduction) constitutes primarily two-player games with perfect information, no chance moves, and the possible outcomes being victory for one player (and loss for the other) and draw for both players. Aside from being utilized in studying mathematical problems that belong to complexity classes harder than NP, these games have intimate connections with disciplines such as mathematical logic, automata theory, complexity theory, graph and matroid theory, networks, error-correcting codes, online algorithms. Outside of mathematics, these games find applications in biology, psychology, economics, insurance, actuarial studies and political sciences. 

\cite{percolation_games} studies \emph{percolation games} on oriented Euclidean lattices. Each site of $\mathbb{Z}^{2}$, independent of all other sites, is marked a ``trap" or a ``target" or ``open" with probabilities $p$, $q$ and $1-p-q$ respectively, and the two players take turns to move a token from its current position $(x,y)$ to either $(x+1,y)$ or $(x,y+1)$. If a player moves to a target, she wins immediately, and if she moves to a trap, she loses immediately. The game's outcome can be interpreted in terms of the evolution of a one-dimensional discrete-time probablistic cellular automaton (PCA) -- specifically, the game having no chance of ending in a draw is shown to be equivalent to the ergodicity of this PCA. \cite{percolation_games} also establishes a connection between the \emph{trapping game} (i.e.\ where $q = 0$) on directed graphs in higher dimensions and the hard-core model on related undirected graphs with reduced dimensions. \cite{trapping_games} studies the trapping game on undirected graphs, where the players take turns to move the token from the vertex at which it is currently located to an adjacent vertex that has never been visited before, and the player unable to make a move loses (note the evident connection between this game and the normal game described above). The outcome of this game is shown to have close ties with maximum-cardinality matchings, and a draw in this game relates to the sensitivity of such matchings to boundary conditions. \cite{wastlund} studies a related, two-person zero-sum game called \emph{exploration} on a \emph{rooted distance model}, to analyze minimum-weight matchings in edge-weighted graphs. In a related game called \emph{slither} (\cite{slither}), the players take turns to claim yet-unclaimed edges of a simple, undirected graph, such that the chosen edges, at all times, form a path, and whoever fails to move, loses. This too serves as a tool for understanding maximum matchings in graphs. 

Bearing some resemblance to slither are the \emph{maker-breaker positional games} (\cite{positional_games_book}), involving a set $X$, a collection $\mathcal{F}$ of subsets of $X$, and positive integers $a$ and $b$. The players named \emph{Maker} and \emph{Breaker} take turns to claim yet-unclaimed elements of $X$, with Maker choosing $a$ elements at a time and Breaker $b$ elements at a time, until all elements of $X$ are exhausted. Maker wins if she has claimed all elements of a subset in $\mathcal{F}$. When this game is played on a graph, the players take turns to claim yet-unclaimed edges, and Maker wins if the subgraph induced by her claimed edges satisfies a desired property (e.g.\ it is connected, or it forms a clique of a given size, a Hamiltonian cycle, a perfect matching or a spanning tree). The game is \emph{unbiased} when $a = b$, and \emph{biased} otherwise. This game has intimate connections with existential fragments of first order and monadic second order logic on graphs. \cite{milos_thesis} and \cite{milos_tibor} study the threshold probability $p = p_{c}$ beyond which Maker has a winning strategy when this game is played on Erd\H{o}s-R\'{e}nyi random graphs $G(n,p)$; \cite{hamiltonian_maker_breaker} studies the game for Hamiltonian cycles on the complete graph $K_{n}$; \cite{maker_breaker_geometric} studies the game on random geometric graphs; \cite{biased_random_boards} studies the \emph{critical bias} $b^{*}$ of the $(1 : b)$ biased game on $G(n, p(n))$ for $p(n) = \Theta\left(\ln n/n\right)$. In addition, \cite{milos_thesis} studies  the game where Maker wins if she can claim a non-planar graph or a non-$k$-colourable graph. \cite{biased_positional} indicates a deep connection between positional games on complete graphs and the corresponding properties being satisfied by a random graph, and this is consistent with \emph{Erd\H{o}s' probabilistic intuition}, which states that the course of a combinatorial game between two players playing optimally often resembles the evolution of a purely random process. Finally, the \emph{Ehrenfeucht-Fra\"{i}ss\'{e} games} comprise yet another extensive subclass of combinatorial games that play a pivotal role in our understanding of first and monadic second order logic on random rooted trees and random graphs (see, for example, \cite{spencer_threshold, spencer_thoma, spencer_stjohn, kim, bohman, pikhurko, verbitsky, maksim_1, maksim_2, maksim_3, maksim_4, maksim_5, maksim_6, podder_1, podder_2, podder_3, podder_5}).

\subsection{Organization of the paper}\label{subsec:org} Theorem~\ref{thm:main_1_normal} is proved in \S\ref{sec:thm_1_normal_proof}, with the two main parts addressed in \S\ref{subsec:thm:main_1_normal_proof_part_1} and \S\ref{subsec:thm:main_1_normal_proof_part_2} (Lemma~\ref{lem:main_thm_1_1}, Lemma~\ref{lem:main_thm_1_3} and Equation~\eqref{belongs_to_D_{i}} are proved in \S\ref{appsec:thm_1_normal_proof} of the Appendix). The (very similar) proof of Theorem~\ref{thm:main_1_misere} is briefly discussed in \S\ref{sec:thm_1_misere_proof}. Theorem~\ref{main:bounds_on_nl_{k}_ml_{k}} is proved in \S\ref{sec:proof_of_main_3}, Theorem~\ref{main:normal_draw_probab_limit_Poisson} in \S\ref{sec:proof_of_main_4} (Lemmas~\ref{lem:draw_probab_limit_lemma_1} and \ref{lem:F_{i}_derivatives} proved in \S\ref{appsec:proof_of_main_4} of the Appendix), and Theorem~\ref{main:Poisson_k=2_normal_phase_transition} is proved in \S\ref{sec:Poisson_k=2} (Lemmas~\ref{lem:lambda c_{2,lambda} behaviour} through \ref{convexity_lem_5} proved in \S\ref{appsec:Poisson_k=2} of the Appendix). Theorems~\ref{lem:decay_rate_nl_{2,lambda}} and ~\ref{main:misere_normal_comparison} are proved in \S\ref{sec:misere_normal_comparison} (Lemma~\ref{lem:decay_rate_nl_{2,lambda}_precursor} proved in \S\ref{appsec:misere_normal_comparison} of the Appendix), Theorem~\ref{main:compare_1_vs_2} is proved in \S\ref{sec:compare_1_vs_2}, and the proof of Theorem~\ref{main:avg_dur} is covered in \S\ref{sec:avg_dur}.

\section{Proof of Theorem~\ref{thm:main_1_normal}}\label{sec:thm_1_normal_proof}
Fix $k \in \mathbb{N}$ throughout \S\ref{sec:thm_1_normal_proof}, and hence the subscript $k$ is dropped from notations used in this section (for instance, $\NW_{k}$ is replaced by $\NW$, $\NW_{k}^{(n)}$ is replaced by $\NW^{(n)}$ etc.). We begin by deducing the two most fundamental recurrence relations pertaining to the $k$-jump normal game. For a vertex $u$ to be in $\NW$, P1 must be able to move the token, in the first round, to some descendant $v$ of $u$ in $\Gamma_{k}(u)$ such that, if we now consider the game that starts at $v$ and P2 plays the first round, P2 loses. In other words, such a $v$ must be in $\NL$ (note here that the symmetric roles of the two players is crucial). Thus
\begin{align}\label{normal_main_recursion_1}
u \in \NW \Leftrightarrow \exists\ v \in \Gamma_{k}(u) \text{ such that } v \in \NL \Leftrightarrow \Gamma_{k}(u) \cap \NL \neq \emptyset.
\end{align}
For a vertex $u$ to be in $\NL$, either $u$ is childless, in which case P1 is unable to make her very first move, or else \emph{every} vertex $v$ in $\Gamma_{k}(u)$ is such that, if P1 moves the token there, then the game that begins at $v$ with P2 playing the first round is won by P2. In other words, every $v$ in $\Gamma_{k}(u)$ must belong to $\NW$. Thus we have
\begin{align}\label{normal_main_recursion_2}
u \in \NL \Leftrightarrow \Gamma_{1}(u) = \emptyset \text{ or } v \in \NW \text{ for every } v \in \Gamma_{k}(u) \Leftrightarrow \Gamma_{k}(u) \subset \NW.
\end{align}

Next, we establish a compactness result which shows that, if a player is able to win the $k$-jump normal game on a rooted tree which is \emph{locally finite} (i.e.\ every vertex of the tree has finite degree), then she can guarantee to do so within a finite number of rounds which can be specified in advance. Mathematically, this result can be stated as follows:
\begin{lemma}\label{lem_normal_compactness}
In any $k$-jump normal game, we have $\NW = \bigcup_{n=1}^{\infty}\NW^{(n)}$ and $\NL = \bigcup_{n=1}^{\infty}\NL^{(n)}$.
\end{lemma}
This result is proven essentially the same way as Proposition 7 of \cite{holroyd_martin}, but we include a proof nonetheless (see \S\ref{appsec:thm_1_normal_proof} of the Appendix) for the sake of completeness of this work. We note here that the offspring distribution $\chi$ that we consider for $\mathcal{T}_{\chi}$ is supported on $\mathbb{N}_{0}$, hence the (random) number of children of any vertex of $\mathcal{T}_{\chi}$ is almost surely finite.

As a consequence of Lemma~\ref{lem_normal_compactness}, we have $\NL^{(n)} \uparrow \NL$ and $\NW^{(n)} \uparrow \NW$ as $n \uparrow \infty$, which in turn yields
\begin{align}\label{nl_nw_limit}
\lim_{n \rightarrow \infty}\nl^{(n)} = \nl \text{ and } \lim_{n \rightarrow \infty} \nw^{(n)} = \nw.
\end{align}
The two main parts of the proof of Theorem~\ref{thm:main_1_normal} are outlined as follows: in \S\ref{subsec:thm:main_1_normal_proof_part_1}, we show that $\nl$ is a fixed point of the function $H_{k}$, and in \S\ref{subsec:thm:main_1_normal_proof_part_2}, we prove that $\nl$ is, in fact, the minimum positive fixed point of $H_{k}$. 

\subsection{Showing that $\nl_{k}$ is a fixed point of $H_{k}$}\label{subsec:thm:main_1_normal_proof_part_1} For a vertex $u$ to be in $\NW^{(n+1)}$ for any $n \in \mathbb{N}$, there must exist some $v \in \Gamma_{k}(u)$ such that, in the $k$-jump normal game that begins at $v$, the player who plays the first round loses in less than $n$ rounds. In other words,
\begin{align}\label{nw^{(n+2)}_recursion}
u \in \NW^{(n+1)} \Leftrightarrow \exists\ v \in \Gamma_{k}(u) \text{ with } v \in \NL^{(n)}.
\end{align}
For $u$ to be in $\NL^{(n+1)}$, either $u$ is childless, or every descendant $v \in \Gamma_{k}(u)$ must be such that the $k$-jump normal game that begins at $v$ is won in less than $n$ rounds by the player who plays the first round. Thus,
\begin{align}\label{nl^{(n+1)}_recursion}
u \in \NL^{(n+1)} \Leftrightarrow \Gamma_{1}(u) = \emptyset \text{ or } v \in \NW^{(n)} \text{ for every } v \in \Gamma_{k}(u).
\end{align}
We emphasize here that for any $m, n \in \mathbb{N}$, the subsets $\NW^{(m)}$ and $\NL^{(n)}$ are mutually exclusive. 

Figure~\ref{fig_1} (for $k = 2$) is included here to help the reader visualize the more refined subsets or \emph{classes} of vertices that we now introduce to carry out the full analysis. It may also be helpful for the reader to refer back to the discussion included right after the statement of Theorem~\ref{thm:main_1_normal} for the case of $k=2$, to keep in mind an outline of how we aim to proceed in the rest of \S\ref{subsec:thm:main_1_normal_proof_part_1}.
\begin{figure}[h!]
  \centering
    \includegraphics[width=0.6\textwidth]{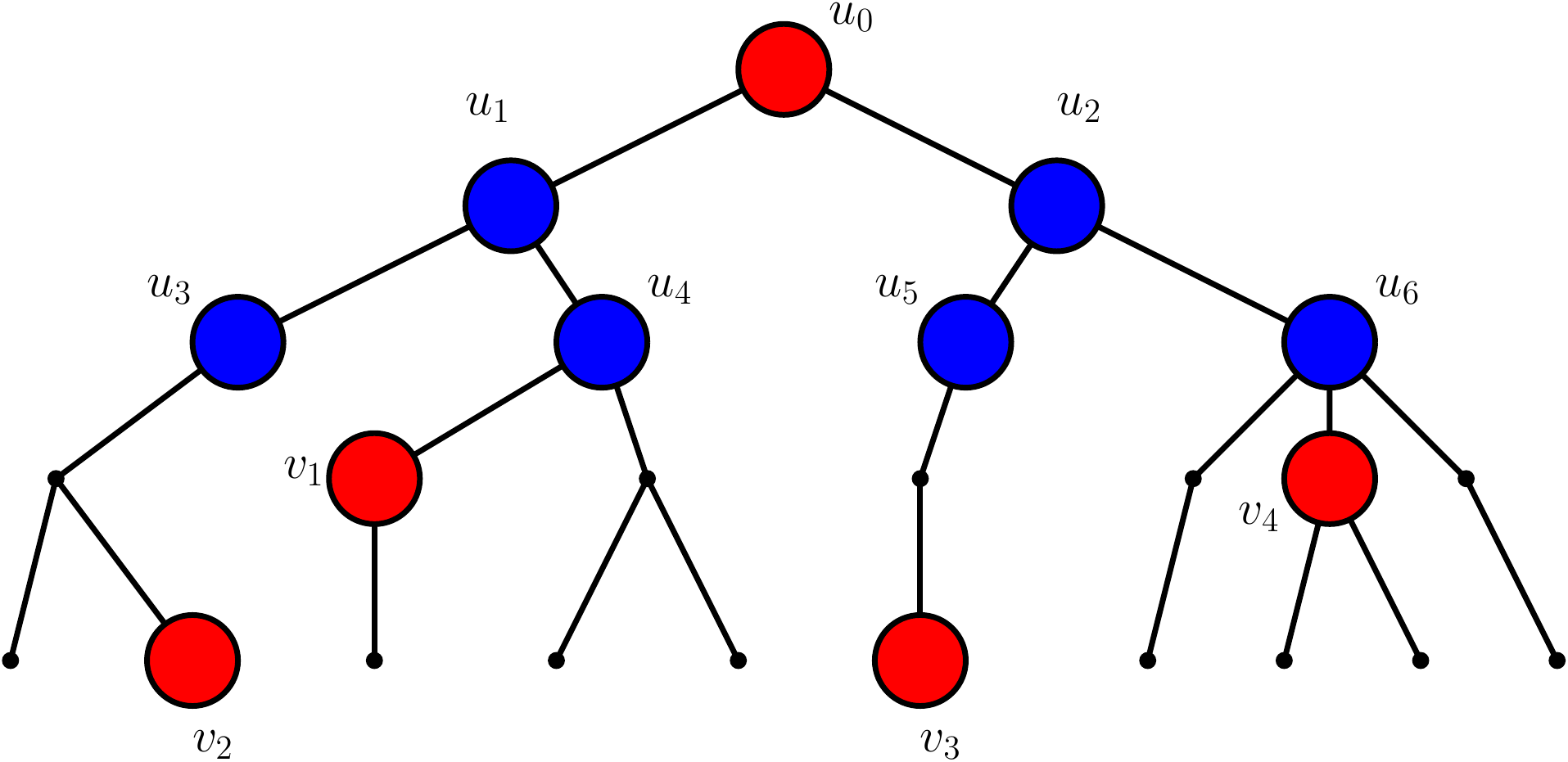}
\caption{Here, $k = 2$ and $u_{0} \in \NL^{(n+2)}$ (indicated in red), so that $u_{i} \in \NW^{(n+1)}$ for all $1 \leqslant i \leqslant 6$ (indicated in blue). Since $u_{1} \in \NW^{(n+1)}$ and its children $u_{3}$ and $u_{4}$ are in $\NW^{(n+1)}$, $u_{1}$ must have at least one grandchild, say $v_{1}$ (indicated in red) in $\NL^{(n)}$. Note that $v_{1} \in \NL^{(n)}$ ensures $u_{4} \in \NW^{(n+1)}$. Likewise, $u_{2}$ and its children $u_{5}$ and $u_{6}$ are in $\NW^{(n+1)}$, hence $u_{2}$ must have at least one grandchild, say $v_{4}$, in $\NL^{(n)}$. That $v_{4} \in \NL^{(n)}$ ensures $u_{6} \in \NW^{(n+1)}$. Let $u_{3}$ and $u_{5}$ have grandchildren $v_{2}$ and $v_{3}$ (respectively) in $\NL^{(n)}$, but no child in $\NL^{(n)}$. Then $u_{4}$ and $u_{6}$ are in $\mathcal{C}_{0,1,n}$, $u_{3}$ and $u_{5}$ are in $\mathcal{C}_{0,2,n}$, and $u_{1}$ and $u_{2}$ are in $\mathcal{C}_{1,2,n}$.}
  \label{fig_1}
\end{figure}

For $0 \leqslant i < j \leqslant k$, we define the subsets of vertices 
\begin{align}\label{C_{i,j,n}_defn}
\mathcal{C}_{i,j,n} = \{u: \Gamma_{i}(u) \subset \NW^{(n+1)}, \Gamma_{j-1}(u) \cap \NL^{(n)} = \emptyset, \Gamma_{j}(u) \cap \NL^{(n)} \neq \emptyset\}.
\end{align}
In other words, any vertex $u$ in $\mathcal{C}_{i,j,n}$ satisfies the following conditions:
\begin{itemize}
\item \emph{all} its descendants that are at distance at most $i$ away from it are in $\NW^{(n+1)}$,
\item \emph{none} of its descendants at distance at most $j-1$ away from it is in $\NL^{(n)}$,
\item and \emph{at least one} of its descendants at distance precisely $j$ away from it is in $\NL^{(n)}$. 
\end{itemize}
Since $j \leqslant k$, the third condition above ensures, via \eqref{nw^{(n+2)}_recursion}, that $\mathcal{C}_{i,j,n} \subset \NW^{(n+1)}$, so that $\mathcal{C}_{i,j,n} \cap \NL^{(n)} = \emptyset$. The second condition above implies $\mathcal{C}_{i,j,n} \cap \mathcal{C}_{i',j',n} = \emptyset$ for all $0 \leqslant i < j \leqslant k$, $0 \leqslant i' < j' \leqslant k$ and $j \neq j'$. We let $p_{i,j,n}$ denote the probability of the event that the root $\phi$ of $\mathcal{T}_{\chi}$ belongs to $\mathcal{C}_{i,j,n}$.

From \eqref{nw^{(n+2)}_recursion}, we see that $u \in \NW^{(n+1)}$ iff $u \in \bigcup_{j=1}^{k}\mathcal{C}_{0,j,n}$, i.e.\ either $u$ has a child $v \in \NL^{(n)}$, or $u$ has a child $v$ with at least one descendant $w$ such that $\rho(v,w) \leqslant k-1$ and $w \in \NL^{(n)}$, which means that $v \in \bigcup_{j=1}^{k-1}\mathcal{C}_{0,j,n}$. Henceforth, given that the vertex $u$ has $m$ children, $m \in \mathbb{N}$, we name them $u_{1}, u_{2}, \ldots, u_{m}$. Thus
\begin{align}\label{nw^{(n+1)}_in_terms_of_c_{0,j,n}}
\nw^{(n+1)} &= \sum_{m=1}^{\infty} \Prob\left[\text{at least one } u_{t} \in \bigcup_{j=1}^{k-1}\mathcal{C}_{0,j,n} \cup \NL^{(n)} \text{ for } 1 \leqslant t \leqslant m\right]\chi(m)\nonumber\\
&= \sum_{m=1}^{\infty} \left[1 - \left(1 - \nl^{(n)} - \sum_{j=1}^{k-1}p_{0,j,n}\right)^{m}\right]\chi(m) = 1 - G\left(1 - \nl^{(n)} - \sum_{j=1}^{k-1}p_{0,j,n}\right).
\end{align} 

From \eqref{nl^{(n+1)}_recursion}, we see that $u \in \NL^{(n+2)}$ if and only if either $u$ is childless, or every child $v$ of $u$ as well as every vertex in $\Gamma_{k-1}(v)$ is in $\NW^{(n+1)}$. However, $v \in \NW^{(n+1)}$ iff some vertex in $\Gamma_{k}(v)$ is in $\NL^{(n)}$. Thus $v$ must have a descendant $w$ such that $\rho(v,w) = k$ and $w \in \NL^{(n)}$, i.e.\ $v \in \mathcal{C}_{k-1,k,n}$. Thus
\begin{align}\label{nl^{(n+1)}_in_terms_of_c_{k-1,k,n}}
\nl^{(n+2)} &= \sum_{m=0}^{\infty}\Prob\left[u_{t} \in \mathcal{C}_{k-1,k,n} \text{ for all } 1 \leqslant t \leqslant m\right]\chi(m) = \sum_{m=0}^{\infty}p_{k-1,k,n}^{m} \chi(m) = G(p_{k-1,k,n}).
\end{align}

We now establish recurrence relations for the probabilities $p_{i,j,n}$. For a vertex $u$ to be in $\mathcal{C}_{0,1,n}$, it must have at least one child in $\NL^{(n)}$, i.e.\
\begin{align}\label{c_{0,1,n}_recursion}
p_{0,1,n} 
= \sum_{m=1}^{\infty}\left[1 - \left(1 - \nl^{(n)}\right)^{m}\right]\chi(m) = 1 - G\left(1 - \nl^{(n)}\right).
\end{align}
For $2 \leqslant j \leqslant k$, $u \in \mathcal{C}_{0,j,n}$ iff at least one child of $u$ is in $\mathcal{C}_{0,j-1,n}$ and no child is in $\bigcup_{\ell=1}^{j-2}\mathcal{C}_{0,\ell,n} \cup \NL^{(n)}$:
\begin{align}\label{c_{0,j,n}_recursion}
p_{0,j,n} &= 
\sum_{m=1}^{\infty}\Prob\left[u_{t} \notin \bigcup_{\ell=1}^{j-2}\mathcal{C}_{0,\ell,n} \cup \NL^{(n)}, 1 \leqslant t \leqslant m\right]\chi(m) - \sum_{m=1}^{\infty}\Prob\left[u_{t} \notin \bigcup_{\ell=1}^{j-1}\mathcal{C}_{0,\ell,n} \cup \NL^{(n)}, 1 \leqslant t \leqslant m\right]\chi(m)\nonumber\\
&= \sum_{m=1}^{\infty}\left(1 - \nl^{(n)} - \sum_{\ell=1}^{j-2}p_{0,\ell,n}\right)^{m}\chi(m) - \sum_{m=1}^{\infty}\left(1 - \nl^{(n)} - \sum_{\ell=1}^{j-1}p_{0,\ell,n}\right)^{m}\chi(m)\nonumber\\
&= G\left(1 - \nl^{(n)} - \sum_{\ell=1}^{j-2}p_{0,\ell,n}\right) - G\left(1 - \nl^{(n)} - \sum_{\ell=1}^{j-1}p_{0,\ell,n}\right).
\end{align} 
Finally, for a vertex $u$ to be in $\mathcal{C}_{i,j,n}$ for any $1 \leqslant i < j \leqslant k$, the following are necessary:
\begin{itemize}
\item \emph{every} child $v$ of $u$ must be in $\NW^{(n+1)}$,
\item \emph{every} vertex in $\Gamma_{i-1}(v)$ must be in $\NW^{(n+1)}$, for \emph{every} child $v$ of $u$,
\item \emph{no} vertex in $\Gamma_{j-2}(v)$ is in $\NL^{(n)}$ for \emph{any} child $v$ of $u$,
\item there exists \emph{at least one} child $v$ of $u$ with \emph{at least one} descendant $w$ such that $\rho(v,w) = j-1$ and $w \in \NL^{(n)}$. 
\end{itemize}
The first condition, along with \eqref{nw^{(n+2)}_recursion}, implies that every child $v$ of $u$ has a descendant in $\Gamma_{k}(v)$ that is in $\NL^{(n)}$. This, along with the second and third conditions, implies that each child of $u$ must be in $\bigcup_{\ell=j-1}^{k}\mathcal{C}_{i-1,\ell,n}$. The fourth condition implies that at least one child of $u$ must be in $\mathcal{C}_{i-1,j-1,n}$. Thus we have
\begin{align}\label{c_{i,j,n}_recursion}
&p_{i,j,n} 
= \sum_{m=1}^{\infty}\Prob\left[u_{t} \in \bigcup_{\ell=j-1}^{k}\mathcal{C}_{i-1,\ell,n}, 1 \leqslant t \leqslant m\right]\chi(m) - \sum_{m=1}^{\infty}\Prob\left[u_{t} \in \bigcup_{\ell=j}^{k}\mathcal{C}_{i-1,\ell,n}, 1 \leqslant t \leqslant m\right]\chi(m)\nonumber\\
&= \sum_{m=1}^{\infty}\left(\sum_{\ell=j-1}^{k}p_{i-1,\ell,n}\right)^{m}\chi(m) - \sum_{m=1}^{\infty}\left(\sum_{\ell=j}^{k}p_{i-1,\ell,n}\right)^{m}\chi(m) = G\left(\sum_{\ell=j-1}^{k}p_{i-1,\ell,n}\right) - G\left(\sum_{\ell=j}^{k}p_{i-1,\ell,n}\right).
\end{align}

We now state a couple of lemmas, with their proofs deferred to \S\ref{appsec:thm_1_normal_proof} of the Appendix, the first of which is concerned with important properties of the function sequence $\{F_{i}\}_{i \in \mathbb{N}_{0}}$, whereas the second provides expressions for the probabilities $p_{i,j,n}$ using \eqref{c_{0,1,n}_recursion}, \eqref{c_{0,j,n}_recursion} and \eqref{c_{i,j,n}_recursion} (once again, it helps if the reader recalls the discussion for $k=2$ presented after the statement of Theorem~\ref{thm:main_1_normal}). We then combine and consolidate the recurrence relations in \eqref{nw^{(n+1)}_in_terms_of_c_{0,j,n}} and \eqref{nl^{(n+1)}_in_terms_of_c_{k-1,k,n}} with the conclusion of Lemma~\ref{lem:main_thm_1_3} in order to obtain the final result.
\begin{lemma}\label{lem:main_thm_1_1}
Recall the functions $F_{i}$ defined in \eqref{F_{i}_defn}. For each $i \in \mathbb{N}$, $F_{i}$ is a strictly decreasing function on $[0,c_{i-1}]$, and consequently, $c_{i}$ is uniquely defined. Moreover, $\chi(0) < c_{i} < c_{i-1}$ for each $i \in \mathbb{N}$. 
\end{lemma}
\begin{lemma}\label{lem:main_thm_1_3}
\sloppy We have
$p_{i,j,n} = g_{i+1}(F_{j-i-1}(\nl^{(n)}), F_{j-i}(\nl^{(n)}), F_{k-i+1}(\nl^{(n)}), \ldots, F_{k}(\nl^{(n)}))$ for $0 \leqslant i < j \leqslant k$, where $F_{i}$s and $g_{i}$s are as defined in \eqref{F_{i}_defn} and \eqref{g_{i}_recursive_defn} respectively.
\end{lemma}
From \eqref{nl^{(n+1)}_in_terms_of_c_{k-1,k,n}}, Lemma~\ref{lem:main_thm_1_3} and \eqref{H_{k}_defn}, we get
\begin{align}\label{fixed_point_recursion_eq}
\nl^{(n+2)} = G(g_{k}(F_{0}(\nl^{(n)}), F_{1}(\nl^{(n)}), F_{2}(\nl^{(n)}), \ldots, F_{k}(\nl^{(n)}))) = H_{k}(\nl^{(n)}).
\end{align}
Taking the limit as $n \rightarrow \infty$ and using \eqref{nl_nw_limit}, we conclude that $\nl$ is a fixed point of $H_{k}$ (the continuity of the pgf $G$ guarantees the continuity of $H_{k}$). By \eqref{nw^{(n+1)}_in_terms_of_c_{0,j,n}}, Lemma~\ref{lem:main_thm_1_3} and \eqref{F_{i}_defn}, we have
\begin{equation}\label{nw^{(n+1)}_closed_form_in_terms_of_nl^{(n)}}
\nw^{(n+1)} 
= 1 - G(1 - \nl^{(n)} - \sum_{j=1}^{k-1}\{F_{j-1}(\nl^{(n)}) - F_{j}(\nl^{(n)})\}) 
= 1 - F_{k}(\nl^{(n)}),
\end{equation}
so that taking the limit as $n \rightarrow \infty$ and using \eqref{nl_nw_limit}, we have $\nw = 1 - F_{k}(\nl)$, as desired.

\subsection{Showing that $\nl_{k}$ is the minimum positive fixed point of $H_{k}$}\label{subsec:thm:main_1_normal_proof_part_2} 
We begin with an outline for the contents of \S\ref{subsec:thm:main_1_normal_proof_part_2}. The primary intention of \S\ref{subsec:thm:main_1_normal_proof_part_2} is to establish that $g_{k}(F_{0}(x),F_{1}(x),F_{2}(x), \ldots, F_{k}(x))$, and consequently $H_{k}$, are both increasing on $[0,c_{k-1}]$, for every $k \in \mathbb{N}$. Once established, this fact aids, as follows, in concluding that $\nl = \nl_{k}$ is the smallest positive fixed point of $H_{k}$. 

Let $0 < \eta \leqslant c_{k}$ be a fixed point of $H_{k}$. From \eqref{fixed_point_recursion_eq}, we have $\nl^{(2n)} = H_{k}^{(n)}(\nl^{(0)}) = H_{k}^{(n)}(0)$, where $H_{k}^{(n)}$ indicates the $n$-fold composition of $H_{k}$ with itself. Note that,
\begin{enumerate}
\item since $c_{k}$ is a fixed point of $H_{k}$ by Corollary~\ref{cor:c_{k}_fixed_point}, 
\item since the proof of Lemma~\ref{lem:main_thm_1_1} yields $0 < c_{k} < c_{k-1}$ and $H_{k}(0) = \chi(0)$ (which follows from the fact that $F_{i}(0) = 1$ for all $i$), 
\item and since $H_{k}$ is increasing on $[0,c_{k-1}]$, 
\end{enumerate}
we have $0 < \chi(0) \leqslant H_{k}^{(n)}(0) = \nl^{(2n)} \leqslant H_{k}^{(n)}(\eta) = \eta \leqslant H_{k}^{(n)}(c_{k}) = c_{k} < c_{k-1}$ for each $n \in \mathbb{N}$. Upon taking the limit as $n \rightarrow \infty$ and using \eqref{nl_nw_limit}, this yields $\nl = \lim_{n \rightarrow \infty}\nl^{(2n)} \leqslant \eta$, thus allowing us to conclude that $\nl$ is, indeed, the smallest positive fixed point of $H_{k}$.

We prove that $g_{k}(F_{0}(x),F_{1}(x),F_{2}(x), \ldots, F_{k}(x))$ is increasing on $[0,c_{k-1}]$ by showing that the derivative 
\begin{align}
&\frac{d}{dx}g_{k}\big(F_{0}(x),F_{1}(x),F_{2}(x), \ldots, F_{k}(x)\big) = G'\big(g_{k-1}\big(F_{0}(x),F_{2}(x),F_{3}(x),\ldots,F_{k}(x)\big)\big) \frac{d}{dx}g_{k-1}\big(F_{0}(x),F_{2}(x),F_{3}(x),\nonumber\\&\ldots,F_{k}(x)\big) - G'\big(g_{k-1}\big(F_{1}(x),F_{2}(x),F_{3}(x),\ldots,F_{k}(x)\big)\big) \frac{d}{dx}g_{k-1}\big(F_{1}(x),F_{2}(x),F_{3}(x),\ldots,F_{k}(x)\big),\nonumber
\end{align}
is non-negative for $x \in [0,c_{k-1}]$. We accomplish this by showing that, for each $x \in [0, c_{k-1}]$,  
\begin{enumerate}
\item \label{item_1} that $G'\big(g_{k-1}\big(F_{0}(x),F_{2}(x),F_{3}(x),\ldots,F_{k}(x)\big)\big) \geqslant G'\big(g_{k-1}\big(F_{1}(x),F_{2}(x),F_{3}(x),\ldots,F_{k}(x)\big)\big)$,
\item \label{item_2} that $\frac{d}{dx}g_{k-1}\big(F_{0}(x),F_{2}(x),F_{3}(x),\ldots,F_{k}(x)\big) \geqslant \frac{d}{dx}g_{k-1}\big(F_{1}(x),F_{2}(x),F_{3}(x),\ldots,F_{k}(x)\big)$,
\item \label{item_3} and the product $G'\big(g_{k-1}\big(F_{0}(x),F_{2}(x),F_{3}(x),\ldots,F_{k}(x)\big)\big) \frac{d}{dx}g_{k-1}\big(F_{0}(x),F_{2}(x),F_{3}(x),\ldots,F_{k}(x)\big)$ is non-negative.
\end{enumerate}
We note that $G'(x) = \sum_{i=1}^{\infty}i \chi(i) x^{i-1} \geqslant 0$ for all $x \in [0,1]$, so that $G'\big(g_{k-1}\big(1,F_{2}(x),F_{3}(x),\ldots,F_{k}(x)\big)\big) \geqslant 0$ for each $x \in [0, c_{k-1}]$. So, to show \eqref{item_3}, it suffices to prove that $\frac{d}{dx}g_{k-1}\big(1,F_{2}(x),F_{3}(x),\ldots,F_{k}(x)\big)$ is non-negative. Furthermore, $G''(x) = \sum_{i=2}^{\infty}i (i-1) \chi(i) x^{i-2} \geqslant 0$ for all $x \in [0,1]$, showing that $G'$ is increasing on $[0,1]$. Consequently, to show \eqref{item_1}, it suffices to show that $g_{k-1}\big(1,F_{2}(x),F_{3}(x),\ldots,F_{k}(x)\big) \geqslant g_{k-1}\big(F_{1}(x),F_{2}(x),F_{3}(x),\ldots,F_{k}(x)\big)$.

Each of \eqref{item_1}, \eqref{item_2} and \eqref{item_3} will be established if we show that 
for any $0 \leqslant i_{1} < i_{2} \leqslant k-j$ and all $x \in [0,c_{k-1}]$, 
\begin{align}\label{belongs_to_D_{i}}
1 \geqslant g_{j}\big(F_{i_{1}}(x), F_{k-j+1}(x), F_{k-j+2}(x), \ldots, F_{k}(x)\big) \geqslant g_{j}\big(F_{i_{2}}(x), F_{k-j+1}(x), F_{k-j+2}(x), \ldots, F_{k}(x)\big) \geqslant 0,
\end{align}
and for $k \geqslant 2$, $1 \leqslant i \leqslant k-j$ and all $x \in [0,c_{k-1}]$, we have
\begin{align}\label{g_{j}(1,F_{k-j+1},...,F_{k})_derivative_g_{j}(F_{1},F_{k-j+1},...,F_{k})}
\frac{d}{dx}g_{j}\big(F_{0}(x),F_{k-j+1}(x),F_{k-j+2}(x),\ldots,F_{k}(x)\big) \geqslant \max\big\{0, \frac{d}{dx}g_{j}\big(F_{i}(x),F_{k-j+1}(x),F_{k-j+2}(x),\ldots,F_{k}(x)\big)\big\}.
\end{align}
The proof of \eqref{belongs_to_D_{i}} is deferred to \S\ref{appsec:thm_1_normal_proof} of the Appendix, whereas \eqref{g_{j}(1,F_{k-j+1},...,F_{k})_derivative_g_{j}(F_{1},F_{k-j+1},...,F_{k})} is proved below via induction on $j$. Setting $j = k-1$, $i_{1}=0$ and $i_{2}=1$ in \eqref{belongs_to_D_{i}}, we obtain $g_{k-1}\big(F_{0}(x), F_{2}(x), F_{3}(x), \ldots, F_{k}(x)\big) \in [0,1]$ and $g_{k-1}\big(F_{1}(x), F_{2}(x), F_{3}(x), \ldots, F_{k}(x)\big) \in [0,1]$, so that $\big(F_{0}(x), F_{1}(x), F_{2}(x), \ldots, F_{k}(x)\big) \in \mathcal{D}_{k}$ due to \eqref{D_{i}_domain_defn}. Consequently, by \eqref{H_{k}_defn}, we conclude that $H_{k}$ is well-defined on $[0,c_{k-1}]$.
 
We now come to the inductive argument for proving \eqref{g_{j}(1,F_{k-j+1},...,F_{k})_derivative_g_{j}(F_{1},F_{k-j+1},...,F_{k})}. For $1 \leqslant i < k$, since $F'_{i}(x)$ and $F'_{k}(x)$ are both negative for $x \in [0,c_{k-1}]$ (due to Lemma~\ref{lem:main_thm_1_1}), we have 
\begin{equation}
\frac{d}{dx} g_{1}\big(F_{0}(x),F_{k}(x)\big) = -F'_{k}(x) \geqslant \max\big\{0,F'_{i}(x) - F'_{k}(x)\big\} = \max\big\{0, \frac{d}{dx} g_{1}\big(F_{i}(x),F_{k}(x)\big)\big\},\nonumber
\end{equation}
proving \eqref{g_{j}(1,F_{k-j+1},...,F_{k})_derivative_g_{j}(F_{1},F_{k-j+1},...,F_{k})} for $j=1$. Suppose \eqref{g_{j}(1,F_{k-j+1},...,F_{k})_derivative_g_{j}(F_{1},F_{k-j+1},...,F_{k})} holds for some $j < k-1$. From \eqref{g_{i}_recursive_defn}, for $x \in [0,c_{k-1}]$, we have
\begin{align}
&\frac{d}{dx}g_{j+1}\big(F_{0}(x),F_{k-j}(x),F_{k-j+1}(x),\ldots,F_{k}(x)\big)
= G'\big(g_{j}\big(F_{0}(x),F_{k-j+1}(x),\ldots,F_{k}(x)\big)\big) \frac{d}{dx}g_{j}\big(F_{0}(x),F_{k-j+1}(x),\nonumber\\&\ldots,F_{k}(x)\big) - G'\big(g_{j}\big(F_{k-j}(x),F_{k-j+1}(x),\ldots,F_{k}(x)\big)\big) \frac{d}{dx}g_{j}\big(F_{k-j}(x),F_{k-j+1}(x),\ldots,F_{k}(x)\big) \geqslant 0,\label{2.15_part_1}
\end{align}
\sloppy  
\begin{enumerate}
\item since $G'\big(g_{j}\big(F_{0}(x),F_{k-j+1}(x),\ldots,F_{k}(x)\big)\big) \geqslant G'\big(g_{j}\big(F_{k-j}(x),F_{k-j+1}(x),\ldots,F_{k}(x)\big)\big) \geqslant 0$ due to \eqref{belongs_to_D_{i}} (setting $i_{0}=0$ and $i_{2}=k-j$) and the fact (already shown above) that $G'$ is increasing on $[0,1]$,
\item and since, by the induction hypothesis \eqref{g_{j}(1,F_{k-j+1},...,F_{k})_derivative_g_{j}(F_{1},F_{k-j+1},...,F_{k})}, we have $\frac{d}{dx}g_{j}\big(F_{0}(x),F_{k-j+1}(x),\ldots,F_{k}(x)\big) \geqslant \max\big\{0,\frac{d}{dx}g_{j}\big(F_{k-j}(x),F_{k-j+1}(x),\ldots,F_{k}(x)\big)\big\}$.
\end{enumerate}
Next, for $1 \leqslant i < k-j$, using \eqref{g_{i}_recursive_defn} once again, we have
\begin{align}
& \frac{d}{dx}g_{j+1}\big(F_{0}(x),F_{k-j}(x),F_{k-j+1}(x),\ldots,F_{k}(x)\big) - \frac{d}{dx}g_{j+1}\big(F_{i}(x),F_{k-j}(x),F_{k-j+1}(x),\ldots,F_{k}(x)\big) \nonumber\\
&= \Big[G'\big(g_{j}\big(F_{0}(x),F_{k-j+1}(x),\ldots,F_{k}(x)\big)\big) \frac{d}{dx}g_{j}\big(F_{0}(x),F_{k-j+1}(x),\ldots,F_{k}(x)\big) \nonumber\\&- G'\big(g_{j}\big(F_{k-j}(x),F_{k-j+1}(x),\ldots,F_{k}(x)\big)\big) \frac{d}{dx}g_{j}\big(F_{k-j}(x),F_{k-j+1}(x),\ldots,F_{k}(x)\big)\Big] \nonumber\\& - \Big[G'\big(g_{j}\big(F_{i}(x),F_{k-j+1}(x),\ldots,F_{k}(x)\big)\big) \frac{d}{dx}g_{j}\big(F_{i}(x),F_{k-j+1}(x),\ldots,F_{k}(x)\big) \nonumber\\&- G'\big(g_{j}\big(F_{k-j}(x),F_{k-j+1}(x),\ldots,F_{k}(x)\big)\big) \frac{d}{dx}g_{j}\big(F_{k-j}(x),F_{k-j+1}(x),\ldots,F_{k}(x)\big)\Big]\nonumber\\
&= G'\big(g_{j}\big(1,F_{k-j+1}(x),\ldots,F_{k}(x)\big)\big) \frac{d}{dx}g_{j}\big(1,F_{k-j+1}(x),\ldots,F_{k}(x)\big) \nonumber\\&- G'\big(g_{j}\big(F_{i}(x),F_{k-j+1}(x),\ldots,F_{k}(x)\big)\big) \frac{d}{dx}g_{j}\big(F_{i}(x),F_{k-j+1}(x),\ldots,F_{k}(x)\big) \geqslant 0,\nonumber
\end{align}
\sloppy \begin{enumerate}
\item since $G'\big(g_{j}\big(1,F_{k-j+1}(x),\ldots,F_{k}(x)\big)\big) \geqslant G'\big(g_{j}\big(F_{i}(x),F_{k-j+1}(x),\ldots,F_{k}(x)\big)\big) \geqslant 0$ due to \eqref{belongs_to_D_{i}} and the fact (already justified above) that $G'$ is increasing on $[0,1]$, 
\item and since, by the induction hypothesis \eqref{g_{j}(1,F_{k-j+1},...,F_{k})_derivative_g_{j}(F_{1},F_{k-j+1},...,F_{k})}, we have $\frac{d}{dx}g_{j}\big(1,F_{k-j+1}(x),\ldots,F_{k}(x)\big) \geqslant \max\big\{0,\frac{d}{dx}g_{j}\big(F_{i}(x),F_{k-j+1}(x),\ldots,F_{k}(x)\big)\big\}$. 
\end{enumerate}
This completes the proof of \eqref{g_{j}(1,F_{k-j+1},...,F_{k})_derivative_g_{j}(F_{1},F_{k-j+1},...,F_{k})} by induction on $j$, as desired, and it brings us to the end of our proof that $\nl$ is the minimum positive fixed point of $H_{k}$, accomplishing the goal of \S\ref{subsec:thm:main_1_normal_proof_part_2}.


\section{Proof of Theorem~\ref{thm:main_1_misere}}\label{sec:thm_1_misere_proof}
As the proof of Theorem~\ref{thm:main_1_misere} closely resembles that of Theorem~\ref{thm:main_1_normal}, we only point out the major modifications. Instead of \eqref{normal_main_recursion_1} and \eqref{normal_main_recursion_2}, we now have (omitting the subscript $k$, as $k$ is fixed throughout \S\ref{sec:thm_1_misere_proof}):
\begin{align}
& u \in \MW \Leftrightarrow \Gamma_{1}(u) = \emptyset \text{ or } \exists\ v \in \Gamma_{k}(u) \text{ with } v \in \ML, \label{misere_main_recursion_1}\\
& u \in \ML \Leftrightarrow \Gamma_{1}(u) \neq \emptyset \text{ and } v \in \MW \text{ for every } v \in \Gamma_{k}(u). \label{misere_main_recursion_2}
\end{align}
The recursions \eqref{nw^{(n+2)}_recursion} and \eqref{nl^{(n+1)}_recursion} are also accordingly replaced by
\begin{align}
& u \in \MW^{(n+1)} \Leftrightarrow \Gamma_{1}(u) = \emptyset \text{ or } \exists\ v \in \Gamma_{k}(u) \text{ with } v \in \ML^{(n)}, \label{mw^{(n+2)}_recursion} \\
& u \in \ML^{(n+1)} \Leftrightarrow \Gamma_{1}(u) \neq \emptyset \text{ and } v \in \MW^{(n)} \text{ for every } v \in \Gamma_{k}(u). \label{ml^{(n+1)}_recursion}
\end{align}
Analogous to the classes of vertices $\mathcal{C}_{i,j,n}$ defined in \eqref{C_{i,j,n}_defn}, we now define, for $0 \leqslant i < j \leqslant k$,
\begin{align}\label{D_{i,j,n}_defn}
\mathcal{D}_{i,j,n} = \{u: \Gamma_{i}(u) \subset \MW^{(n+1)}, \Gamma_{j-1}(u) \cap \ML^{(n)} = \emptyset, \Gamma_{j}(u) \cap \ML^{(n)} \neq \emptyset\}.
\end{align}
Then $\mathcal{D}_{i,j,n} \cap \mathcal{D}_{i',j',n} = \emptyset$ when $j \neq j'$, and \eqref{mw^{(n+2)}_recursion} implies $\mathcal{D}_{i,j,n} \subset \MW^{(n+1)}$, so that $\mathcal{D}_{i,j,n} \cap \ML^{(n)} = \emptyset$. Let $q_{i,j,n}$ be the probability of the event that the root $\phi$ of $\mathcal{T}_{\chi}$ belongs to $\mathcal{D}_{i,j,n}$. From \eqref{mw^{(n+2)}_recursion} and \eqref{D_{i,j,n}_defn}:
\begin{align}\label{mw^{(n+1)}_recursion_in_terms_of_q_{0,j,n}}
& \mw^{(n+1)} = \chi(0) + \sum_{m=1}^{\infty}\Prob[\text{at least one } u_{t} \in \bigcup_{j=1}^{k-1}\mathcal{D}_{0,j,n} \cup \ML^{(n)} \text{ for } 1 \leqslant t \leqslant m]\chi(m)\nonumber\\
&= \chi(0) + \sum_{m=1}^{\infty}\{1 - (1 - \ml^{(n)} - \sum_{j=1}^{k-1}q_{0,j,n})^{m}\}\chi(m) = \chi(0) + 1 - G(1 - \ml^{(n)} - \sum_{j=1}^{k-1}q_{0,j,n}),
\end{align}
where, as in \S\ref{sec:thm_1_normal_proof}, we denote the $m$ children of $\phi$ by $u_{1}$, $u_{2}$, $\ldots$, $u_{m}$. From \eqref{ml^{(n+1)}_recursion} and \eqref{D_{i,j,n}_defn}, we have
\begin{multline}\label{ml^{(n+1)}_recursion_in_terms_of_q_{i,j,n}}
\ml^{(n+2)} = \sum_{m=1}^{\infty}\Prob[u_{t} \in \mathcal{D}_{k-1,k,n} \text{ or } \Gamma_{1}(u_{t}) = \emptyset, \text{ for each } 1 \leqslant t \leqslant m]\chi(m) \\= \sum_{m=1}^{\infty}(q_{k-1,k,n} + \chi(0))^{m}\chi(m) = G(q_{k-1,k,n}+\chi(0)) - \chi(0).
\end{multline}
The recurrence relations for $q_{i,j,n}$, $1 \leqslant i < j \leqslant k$, differs from \eqref{c_{i,j,n}_recursion} in that, for $u$ to be in $\mathcal{D}_{i,j,n}$, each child $v$ of $u$ is either childless or in $\bigcup_{\ell=j-1}^{k}\mathcal{D}_{i-1,\ell,n}$, and at least one child of $u$ must be in $\mathcal{D}_{i-1,j-1,n}$. Thus
\begin{align}\label{d_{i,j,n}_recursion}
& q_{i,j,n} 
= \sum_{m=1}^{\infty}\Prob[u_{t} \in \bigcup_{\ell=j-1}^{k}\mathcal{D}_{i-1,\ell,n} \text{ or } \Gamma_{1}(u_{t}) = \emptyset \text{ for each } 1 \leqslant t \leqslant m]\chi(m) - \sum_{m=1}^{\infty}\Prob[u_{t} \in \bigcup_{\ell=j}^{k}\mathcal{D}_{i-1,\ell,n} \text{ or } \nonumber\\&\Gamma_{1}(u_{t}) = \emptyset \text{ for each } 1 \leqslant t \leqslant m]\chi(m) 
= G(\chi(0) + \sum_{\ell=j-1}^{k}q_{i-1,\ell,n}) - G(\chi(0) + \sum_{\ell=j}^{k}q_{i-1,\ell,n}).
\end{align}

Analogous to Lemma~\ref{lem:main_thm_1_3}, we have the following relation: for every $0 \leqslant i < j \leqslant k$, 
\begin{equation}\label{q_{i,j,n}_closed_form_i>1}
q_{i,j,n} = \gamma_{i+1}(F_{j-i-1}(\ml^{(n)}), F_{j-i}(\ml^{(n)}), F_{k-i+1}(\ml^{(n)}), \ldots, F_{k}(\ml^{(n)})),
\end{equation}
where the functions $\gamma_{i}$ are as defined in \eqref{gamma_{i}_recursive_defn}. 
From \eqref{ml^{(n+1)}_recursion_in_terms_of_q_{i,j,n}} and \eqref{q_{i,j,n}_closed_form_i>1}, we have $\ml^{(n+2)} = J_{k}\left(\ml^{(n)}\right)$, where $J_{k}$ is as defined in \eqref{mathcal{J}_{k}_defn}. Taking the limit as $n \rightarrow \infty$, we see that $\ml$ is a fixed point of $J_{k}$. From \eqref{mw^{(n+1)}_recursion_in_terms_of_q_{0,j,n}} and \eqref{q_{i,j,n}_closed_form_i>1}, we have $\mw^{(n+1)} =\chi(0) + 1 - F_{k}\left(\ml^{(n)}\right)$, and taking the limit as $n \rightarrow \infty$, we get $\mw = \chi(0) + 1 - F_{k}(\ml)$. The approach to showing that $\ml$ is the minimum positive fixed point of $J_{k}$ is nearly identical to that adopted in \S\ref{subsec:thm:main_1_normal_proof_part_2}, and is therefore omitted. 

\section{Proof of Theorem~\ref{main:bounds_on_nl_{k}_ml_{k}}}\label{sec:proof_of_main_3} 
Fix any $k \in \mathbb{N}$. Recall that the objective of \S\ref{sec:proof_of_main_3} is to show that $\chi(0) < \nl_{k} \leqslant c_{k}$, and that the probability of draw $\nd_{k}$ is strictly positive if and only if $\nl_{k} < c_{k}$, in case of $k$-jump normal games, where $c_{k}$ is the unique fixed point of $F_{k}$ in $(0, c_{k-1})$; on the other hand, $\ml_{k} \leqslant \hat{c}_{k}$, and the probability of draw $\md_{k}$ is strictly positive if and only if $\ml_{k} < \hat{c}_{k}$, in case of $k$-jump mis\`{e}re games, where $\hat{c}_{k}$ is the unique point of intersection between $y=F_{k}(x)$ and $y = J_{k}(x)+\chi(0)$ in $(0,c_{k-1})$ (recall the definitions of the functions $F_{k}(x)$ and $J_{k}(x)$ from \eqref{F_{i}_defn} and \eqref{mathcal{J}_{k}_defn} respectively).
\begin{lemma}\label{lem:g_{j}_c_{k}_patterns}
For all $k \geqslant 2$, $j \leqslant k$ and $0 \leqslant i \leqslant k-j$, we have
\begin{align}\label{g_{j}_c_{k}_patterns}
g_{j}\left(F_{i}(c_{k}), F_{k-j+1}(c_{k}), F_{k-j+2}(c_{k}), \ldots, F_{k}(c_{k})\right) = F_{j+i-1}(c_{k}) - c_{k}.
\end{align}
\end{lemma}
\begin{proof}
When $j=1$, the left side of \eqref{g_{j}_c_{k}_patterns} equals $g_{1}\left(F_{i}(c_{k}), F_{k}(c_{k})\right) = F_{i}(c_{k}) - c_{k}$, since $c_{k}$ is the fixed point of $F_{k}$. Suppose \eqref{g_{j}_c_{k}_patterns} holds for some $j < k$ and all $0 \leqslant i \leqslant k-j$. For $0 \leqslant i \leqslant k-j-1$, we have 
\begin{align}
& g_{j+1}(F_{i}(c_{k}), F_{k-j}(c_{k}), F_{k-j+1}(c_{k}), \ldots, F_{k}(c_{k})) = G(g_{j}(F_{i}(c_{k}), F_{k-j+1}(c_{k}), \ldots, F_{k}(c_{k}))) - G(g_{j}(F_{k-j}(c_{k}), \nonumber\\& F_{k-j+1}(c_{k}), \ldots, F_{k}(c_{k}))) = G\left(F_{j+i-1}(c_{k}) - c_{k}\right) - G\left(F_{k-1}(c_{k}) - c_{k}\right) = F_{j+i}(c_{k}) - F_{k}(c_{k}) = F_{j+i}(c_{k}) -  c_{k}. \nonumber \qedhere
\end{align} 
\end{proof}
As an immediate consequence of Lemma~\ref{lem:g_{j}_c_{k}_patterns}, we get the following corollary:
\begin{corollary}\label{cor:c_{k}_fixed_point}
For any $k \in \mathbb{N}$, $c_{k}$ is a fixed point of $H_{k}$.
\end{corollary}
\begin{proof}
We set $j = k$ and $i=0$ in Lemma~\ref{lem:g_{j}_c_{k}_patterns} to get $g_{k}(F_{0}(c_{k}), F_{1}(c_{k}), \ldots, F_{k}(c_{k})) = F_{k-1}(c_{k}) - c_{k}$, which yields $H_{k}(c_{k}) = G(g_{k}(F_{0}(c_{k}), F_{1}(c_{k}), \ldots, F_{k}(c_{k}))) = G(F_{k-1}(c_{k}) - c_{k}) = F_{k}(c_{k}) = c_{k}$.
\end{proof}

By Theorem~\ref{thm:main_1_normal}, $\nl_{k}$ is the minimum positive fixed point of $H_{k}$, and by Corollary~\ref{cor:c_{k}_fixed_point} and Lemma~\ref{lem:main_thm_1_1}, $c_{k}$ is a positive fixed point of $H_{k}$. Hence $\nl_{k} \leqslant c_{k}$. The lower bound on $\nl_{k}$ in Theorem~\ref{main:bounds_on_nl_{k}_ml_{k}} follows simply from observing that if the root has no child, which happens with probability $\chi(0)$, then P1 loses the game starting at the root. From Theorem~\ref{thm:main_1_normal}, we have $\nd_{k} = 1 - \nw_{k} - \nl_{k} = F_{k}(\nl_{k}) - \nl_{k}$. We already know from Lemma~\ref{lem:main_thm_1_1} that $F_{k}$ is strictly decreasing on $[0,c_{k-1}]$ and $c_{k} \in (0,c_{k-1})$ is its unique fixed point, which is equivalent to saying that for $x \in [0,c_{k-1}]$, we have $F_{k}(x) - x$ strictly positive if and only if $x < c_{k}$. Therefore, $\nd_{k}$ is strictly positive if and only if $\nl_{k} < c_{k}$.

Note that $F_{k}(0) = 1$ (since the proof of Lemma~\ref{lem:main_thm_1_1} yields $F_{i}(0) = 1$ for all $i \in \mathbb{N}$) and $J_{k}(0) + \chi(0) = G(\chi(0)) < 1$ as $\chi(0) < 1$ (which also ensures that $G$ is strictly increasing on $[0,1]$), implying that the curve $y = F_{k}(x)$ lies \emph{above} the curve $y = J_{k}(x) + \chi(0)$ at $x = 0$. On the other hand, $F_{k}(c_{k-1}) = G(F_{k-1}(c_{k-1}) - c_{k-1}) = G(0)$ as $c_{k-1}$ is the fixed point of $F_{k-1}$, whereas $J_{k}(c_{k-1}) + \chi(0) = G\big(\chi(0) + \gamma_{k}\big(1,F_{1}(c_{k-1}),\ldots,F_{k}(c_{k-1})\big)\big) > G(0)$ since $\chi(0) + \gamma_{k}\big(1,F_{1}(c_{k-1}),\ldots,F_{k}(c_{k-1})\big) > 0$, thus implying that the curve $y = F_{k}(x)$ lies \emph{below} the curve $y = J_{k}(x)+\chi(0)$ at $x=c_{k-1}$. Lemma~\ref{lem:main_thm_1_1} shows that $F_{k-1}(x)$ is strictly decreasing on $[0,c_{k-1}]$, whereas an argument analogous to that used for showing that $H_{k}$ is increasing on $[0,c_{k-1}]$ (as outlined in \S\ref{subsec:thm:main_1_normal_proof_part_2}) can be employed to show that $J_{k}$, and hence $J_{k}+\chi(0)$, is increasing on $[0,c_{k-1}]$. Thus $y = F_{k}(x)$ and $y = J_{k}(x)+\chi(0)$ intersect at a \emph{unique} point inside $(0,c_{k-1})$, which we call $\hat{c}_{k}$. From Theorem~\ref{thm:main_1_misere}, we have $\md_{k} = 1 - \mw_{k} - \ml_{k} = F_{k}(\ml_{k}) - \{J_{k}(\ml_{k}) + \chi(0)\}$. Since $F_{k}(x)$ is strictly decreasing on $[0,c_{k-1}]$ and $J_{k}(x) + \chi(0)$ is increasing on $[0,c_{k-1}]$ and they intersect at $\hat{c}_{k}$, we must have $\ml_{k} \leqslant \hat{c}_{k}$ to ensure that $\md_{k} \geqslant 0$, and $\md_{k} > 0$ iff $\ml_{k} < \hat{c}_{k}$. This brings us to the end of the proof of Theorem~\ref{main:bounds_on_nl_{k}_ml_{k}}. 

\section{Proof of Theorem~\ref{main:normal_draw_probab_limit_Poisson}}\label{sec:proof_of_main_4}
Throughout \S\ref{sec:proof_of_main_4}, we fix any $k \in \mathbb{N}$ and let the offspring distribution $\chi$ of the GW tree $\mathcal{T}_{\chi}$ be Poisson$(\lambda)$. In order to emphasize the dependence of all functions and quantities involved on $\lambda$, we replace, from the third paragraph of \S\ref{sec:proof_of_main_4} onward, all of $G$, $F_{i}$, $H_{k}$, $g_{i}$, $c_{i}$, $\nl_{k}$, $\nw_{k}$ and $\nd_{k}$ by $G_{\lambda}$, $F_{i,\lambda}$, $H_{k,\lambda}$, $g_{i,\lambda}$, $c_{i,\lambda}$, $\nl_{k,\lambda}$, $\nw_{k,\lambda}$ and $\nd_{k,\lambda}$ respectively (for all $1 \leqslant i \leqslant k$).

The proof of Lemma~\ref{lem:main_thm_1_1} shows that $F_{i}(0) = 1$ for all $i \in \mathbb{N}$, so that $H_{k}(0) = \chi(0) > 0$. The curve $y = H_{k}(x)$ thus lies \emph{above} the curve $y=x$ at $x = 0$. By Theorem~\ref{thm:main_1_normal}, we know that $x = \nl_{k}$ is the smallest positive value of $x$ at which the curve $y=H_{k}(x)$ either touches or starts going \emph{beneath} the curve $y=x$, so that the slope of $y=H_{k}(x)$ at $x = \nl_{k}$ has to be less than or equal to the slope of $y=x$. Therefore, we must have $H'_{k}(\nl_{k}) \leqslant 1$. 

The goal of \S\ref{sec:proof_of_main_4} is to establish that $H'_{k,\lambda}(c_{k,\lambda}) > 1$ for all $\lambda$ sufficiently large. This ensures, via the conclusion drawn in the previous paragraph, that $\nl_{k,\lambda} \neq c_{k,\lambda}$. By Theorem~\ref{main:bounds_on_nl_{k}_ml_{k}}, we conclude that $\nl_{k,\lambda} < c_{k,\lambda}$ and hence $\nd_{k,\lambda} > 0$ for all such values of $\lambda$. This would then conclude the proof of the first part of the statement of Theorem~\ref{main:normal_draw_probab_limit_Poisson}.  

We outline here the salient steps of the argument employed to prove that $H'_{k,\lambda}(c_{k,\lambda}) > 1$ for all $\lambda$ sufficiently large. In Lemma~\ref{lem:draw_probab_limit_lemma_1}, we obtain an expression for the derivative of the function $g_{k,\lambda}\big(r_{0}(x),r_{1}(x),\ldots,r_{k}(x)\big)$ with respect to $x$, where $r_{i}(x)$, for $0 \leqslant i \leqslant k$, are differentiable, and $\big(r_{0}(x),r_{1}(x),\ldots,r_{k}(x)\big) \in \mathcal{D}_{k}$ (which is necessary because of how we define the function $g_{k}$ in \eqref{g_{i}_recursive_defn}). Letting $r_{i}(x)$ be the function $F_{i,\lambda}(x)$ for all $0 \leqslant i \leqslant k$, and using Lemma~\ref{lem:F_{i}_derivatives} that reveals a pattern in the derivatives of the functions $F_{i,\lambda}(x)$, we show (via \eqref{H_{k}_defn}) that the leading term in the expansion of $H'_{k,\lambda}(c_{k,\lambda})$ is of the same order of magnitude as $\lambda^{2k} c^{2}_{k,\lambda}$, while the remaining terms are $O(\lambda^{2k-1}c^{2}_{k,\lambda})$. Our final task is to show that 
\begin{equation}\label{lambda^{i}c_{k,lambda}_limit_various_i}
\lim_{\lambda \rightarrow \infty}\lambda^{k-1}c_{k,\lambda} = 0 \text{ and } \lim_{\lambda \rightarrow \infty}\lambda^{k}c_{k,\lambda} = \infty,
\end{equation}
which allows us to conclude, in fact, that $\lim_{\lambda \rightarrow \infty}H'_{k,\lambda}(c_{k,\lambda}) = \infty$. Note that the second part of the statement of Theorem~\ref{main:normal_draw_probab_limit_Poisson}, asserting $\lambda^{k-1}\nl_{k,\lambda} \rightarrow 0$ as $\lambda \rightarrow \infty$, follows immediately from the first part of \eqref{lambda^{i}c_{k,lambda}_limit_various_i} and the fact that $\nl_{k,\lambda} \leqslant c_{k,\lambda}$ that we obtain from Theorem~\ref{main:bounds_on_nl_{k}_ml_{k}}.

\subsection{Understanding the behaviour of $c_{k,\lambda}$ as a function of $\lambda$}\label{subsec:proof_of_main_4_part_1} We begin the proof of \eqref{lambda^{i}c_{k,lambda}_limit_various_i} by attempting to understand how $c_{k,\lambda}$ behaves as a function of $\lambda$. The first task we accomplish in \S\ref{subsec:proof_of_main_4_part_1} is showing that $c_{k,\lambda}$ is, in fact, differentiable with respect to $\lambda$, for which we implement the well-known implicit function theorem.

To this end, we \emph{redefine} the functions $F_{i,\lambda}(x)$ on the \emph{extended} interval $[0,1]$ (instead of only on the sub-interval $[0,c_{i-1,\lambda}]$, as done in \eqref{F_{i}_defn}) as follows: 
\begin{equation}
F_{1,\lambda}(x) = e^{-\lambda x} \text{ and } F_{i+1,\lambda}(x) = \exp\{\lambda F_{i,\lambda}(x) - \lambda x - \lambda\} \text{ for all } x \in [0,1], \text{ for } i \in \mathbb{N}.\label{F_{i,lambda}_extended}
\end{equation}
Note that these functions are well-defined. It is immediate that $0 < F_{1,\lambda}(x) < 1$ for all $x \in (0,1)$. We now show, via induction on $i$, that the inequalities $0 < F_{i,\lambda}(x) < 1$ hold for all $x \in (0,1)$, for each $i \in \mathbb{N}$. Suppose we have already shown that $0 < F_{i,\lambda}(x) < 1$ holds for every $x \in (0,1)$, for some $i \in \mathbb{N}$. This yields 
\begin{equation}
\lambda F_{i,\lambda}(x) - \lambda x - \lambda < -\lambda x < 0 \implies 0 < F_{i+1,\lambda}(x) < 1 \text{ for all } x \in (0,1),\label{claim_2}
\end{equation}
completing the inductive argument. Next, we note that $F'_{1,\lambda}(x) = -\lambda e^{-\lambda x} < 0$. We show, by induction on $i$, that $F'_{i,\lambda}(x) < 0$ for all $x \in (0,1)$, for each $i \in \mathbb{N}$. Suppose we have already shown that $F'_{i,\lambda}(x) < 0$ for all $x \in (0,1)$, for some $i \in \mathbb{N}$. We then have 
\begin{equation}
F'_{i+1,\lambda}(x) = (\lambda F'_{i,\lambda}(x) - \lambda) F_{i+1,\lambda}(x) < 0 \text{ for all } x \in (0,1).\nonumber
\end{equation}
This completes the inductive argument. Let us define $f_{i}(\lambda, x) = F_{i,\lambda}(x) - x$ on $\Omega = (0,\infty) \times (0,1)$, so that $(\lambda, c_{i,\lambda})$ is a point on the curve $f_{i}(\lambda, x) = 0$, and $\frac{\partial}{\partial x}f_{i}(\lambda, x) = F'_{i,\lambda}(x) - 1 < 0$ for all $x \in (0,1)$. By the implicit function theorem, for every $\lambda > 0$, there exists an open $U \times V \subset \Omega$, containing $(\lambda, c_{i,\lambda})$, and a function $h: U \mapsto V$, differentiable on $U$, such that $c_{i,\lambda} = h(\lambda)$. This concludes our first task, i.e.\ showing that $c_{k,\lambda}$ is differentiable with respect to $\lambda$.

The second task we accomplish in \S\ref{subsec:proof_of_main_4_part_1} is showing that $c_{k,\lambda}$ is a strictly decreasing function of $\lambda$, by showing that $\frac{d}{d\lambda}c_{k,\lambda}$ is strictly negative for all $\lambda > 0$. For any function $f(x)$ that is defined and differentiable for all $x > 0$, and $0 < f(x) < 1$, we show, by induction on $i$, that
\begin{equation}
\frac{d}{d\lambda}F_{i,\lambda}(f(\lambda)) = A_{i}(f(\lambda)) + B_{i}(f(\lambda)) f'(\lambda) \text{ where } A_{i}(f(\lambda)) < 0 \text{ and } B_{i}(f(\lambda)) < 0,\label{claim_1}
\end{equation}
for all $\lambda > 0$ and $i \in \mathbb{N}$. We note that $A_{1}(f(\lambda)) = -f(\lambda) e^{-\lambda f(\lambda)}$ and $B_{1}(f(\lambda)) = -\lambda e^{-\lambda f(\lambda)}$, so that the base case is verified. Assuming that \eqref{claim_1} holds for some $i \in \mathbb{N}$, we have
\begin{align}
&\frac{d}{d\lambda}F_{i+1,\lambda}\big(f(\lambda)\big) = \left[F_{i,\lambda}\big(f(\lambda)\big) - f(\lambda) - 1 + \lambda\left\{\frac{d}{d\lambda}F_{i,\lambda}\big(f(\lambda)\big) - f'(\lambda)\right\}\right] F_{i+1,\lambda}\big(f(\lambda)\big)\nonumber\\
&= \left[F_{i,\lambda}\big(f(\lambda)\big) - f(\lambda) - 1 + \lambda\left\{A_{i}\big(f(\lambda)\big) + B_{i}\big(f(\lambda)\big)f'(\lambda) - f'(\lambda)\right\}\right]F_{i+1,\lambda}\big(f(\lambda)\big)\nonumber\\
&= \left[F_{i,\lambda}\big(f(\lambda)\big) - 1 - f(\lambda) + \lambda A_{i}\big(f(\lambda)\big)\right]F_{i+1,\lambda}\big(f(\lambda)\big) + \lambda\left[B_{i}\big(f(\lambda)\big) - 1\right]F_{i+1,\lambda}\big(f(\lambda)\big)f'(\lambda),\nonumber
\end{align}
so that 
\begin{equation}
A_{i+1}(f(\lambda)) = \left[F_{i,\lambda}\big(f(\lambda)\big) - 1 - f(\lambda) + \lambda A_{i}\big(f(\lambda)\big)\right]F_{i+1,\lambda}(f(\lambda))\nonumber
\end{equation}
and 
\begin{equation}
B_{i+1}\big(f(\lambda)\big) = \lambda\left[B_{i}\big(f(\lambda)\big) - 1\right]F_{i+1,\lambda}\big(f(\lambda)\big).\nonumber
\end{equation}
These are both negative due to the induction hypothesis and because $0 < F_{i,\lambda}\big(f(\lambda)\big) < 1$ for all $\lambda > 0$ (due to \eqref{claim_2}). This completes the proof by induction.

When $f(\lambda) = c_{k,\lambda}$, differentiating both sides of the identity $F_{k,\lambda}(c_{k,\lambda}) = c_{k,\lambda}$ (since $c_{k,\lambda}$, recall, is the unique fixed point of $F_{k,\lambda}$), we have 
\begin{equation}
c'_{k,\lambda} = \frac{d}{d\lambda}c_{k,\lambda} = \frac{A_{k}(c_{k,\lambda})}{1-B_{k}(c_{k,\lambda})}.\nonumber
\end{equation}
By \eqref{claim_1}, we see that the numerator is strictly negative whereas the denominator is strictly positive, thus ensuring that $c'_{k,\lambda} < 0$ for all $\lambda > 0$. This concludes our proof of the fact that $c_{k,\lambda}$ is strictly decreasing in $\lambda$ for all $\lambda > 0$. 

As an immediate consequence of this observation, we note that the limit $\lim_{\lambda \rightarrow \infty} c_{k,\lambda}$ exists as $c_{k,\lambda}$ is bounded below by $0$ for all $\lambda > 0$. From \eqref{F_{i,lambda}_extended} and the fact that $c_{k,\lambda}$ is the fixed point of $F_{k,\lambda}$, we obtain 
\begin{equation}
F_{k,\lambda}(c_{k,\lambda}) = \exp\left\{\lambda F_{k-1,\lambda}(c_{k,\lambda}) - \lambda c_{k,\lambda} - \lambda\right\} = c_{k,\lambda} \Longleftrightarrow F_{k-1,\lambda}(c_{k,\lambda}) - c_{k,\lambda} - 1 = \frac{\ln c_{k,\lambda}}{\lambda}.\label{claim_3}
\end{equation}
If $\lim_{\lambda \rightarrow \infty}c_{k,\lambda} = c$ for some $c > 0$, then the right side will go to $0$, while the left side remains bounded above by $-c$ since $F_{k-1,\lambda}(c_{k,\lambda}) < 1$ (due to \eqref{claim_2}), yielding a contradiction. Therefore, we must have $\lim_{\lambda \rightarrow \infty} c_{k,\lambda} = 0$. 

\subsection{Understanding the behaviour of $F_{j,\lambda}(c_{k,\lambda})$ for all $1 \leqslant j \leqslant k-1$}\label{subsec:proof_of_main_4_part_2} Before we can establish the claims made in \eqref{lambda^{i}c_{k,lambda}_limit_various_i}, we need to understand the behaviour of $F_{j,\lambda}(c_{k,\lambda})$ as a function of $\lambda$, as $\lambda \rightarrow \infty$, for each $1 \leqslant j \leqslant k-1$. To this end, note that, given \emph{any} infinite sequence $\{\lambda_{n}\}_{n}$ of positive reals with $\lambda_{n} \rightarrow \infty$, since $0 < F_{k-1,\lambda_{n}}\left(c_{k,\lambda_{n}}\right) < 1$ for every $n \in \mathbb{N}$ due to \eqref{claim_2}, the Bolzano-Weierstrass Theorem guarantees the existence of an infinite subsequence $\left\{\lambda_{n_{i}}\right\}_{i}$ such that 
\begin{equation}
\lim_{i \rightarrow \infty}F_{k-1,\lambda_{n_{i}}}(c_{k,\lambda_{n_{i}}}) \text{ exists and is in } [0,1] \implies \lim_{i \rightarrow \infty} \frac{\ln c_{k,\lambda_{n_{i}}}}{\lambda_{n_{i}}} \text{ exists and is in } [-1,0], \text{ due to \eqref{claim_3}}.\nonumber
\end{equation}
Suppose we assume that
\begin{equation}
\lim_{i \rightarrow \infty} \frac{\ln c_{k,\lambda_{n_{i}}}}{\lambda_{n_{i}}} = -c \text{ for some } 0 < c \leqslant 1.\label{contradiction_1}
\end{equation}
In what follows, our aim is to show that \eqref{contradiction_1} leads to a contradiction.

\subsubsection{Proving that \eqref{contradiction_1} leads to a contradiction}\label{subsec:proof_of_main_4_part_2_subpart_1} Given any $0 < \epsilon < c$, \eqref{contradiction_1} implies that there exists $i_{\epsilon}$ such that $c_{k,\lambda_{n_{i}}} < e^{(-c+\epsilon)\lambda_{n_{i}}}$ for all $i \geqslant i_{\epsilon}$. We then show, using an inductive argument with respect to the index $j$, that
\begin{equation}\label{F_{i,lambda}(c_{k,lambda})_lower_bound}
F_{j,\lambda_{n_{i}}}\left(c_{k,\lambda_{n_{i}}}\right) > \exp\left\{-\sum_{t=1}^{j}\lambda_{n_{i}}^{t}e^{(-c+\epsilon)\lambda_{n_{i}}}\right\} \text{ for all } i \geqslant i_{\epsilon}, \text{ for } 1 \leqslant j \leqslant k.
\end{equation}
For $i \geqslant i_{\epsilon}$, we have 
\begin{equation}
F_{1,\lambda_{n_{i}}}\left(c_{k,\lambda_{n_{i}}}\right) = e^{-\lambda_{n_{i}}c_{k,\lambda_{n_{i}}}} > \exp\left\{-\lambda_{n_{i}}e^{(-c+\epsilon)\lambda_{n_{i}}}\right\},\nonumber
\end{equation}
so that \eqref{F_{i,lambda}(c_{k,lambda})_lower_bound} holds for $j=1$, and the base case for the induction is thus verified. Assuming that \eqref{F_{i,lambda}(c_{k,lambda})_lower_bound} holds for some $j < k$ and using $e^{-x} - 1 > -x$ for all $x > 0$, we have, for all $i \geqslant i_{\epsilon}$,
\begin{align}
F_{j+1,\lambda_{n_{i}}}\left(c_{k,\lambda_{n_{i}}}\right) &= \exp\left\{\lambda_{n_{i}}F_{j,\lambda_{n_{i}}}\left(c_{k,\lambda_{n_{i}}}\right) - \lambda_{n_{i}}c_{k,\lambda_{n_{i}}} - \lambda_{n_{i}}\right\} \nonumber\\
&> \exp\left\{\lambda_{n_{i}}\exp\left\{-\sum_{t=1}^{j}\lambda_{n_{i}}^{t}e^{(-c+\epsilon)\lambda_{n_{i}}}\right\} - \lambda_{n_{i}}e^{(-c+\epsilon)\lambda_{n_{i}}} - \lambda_{n_{i}}\right\} \nonumber\\
&\geqslant \exp\left\{-\lambda_{n_{i}}\sum_{t=1}^{j}\lambda_{n_{i}}^{t}e^{(-c+\epsilon)\lambda_{n_{i}}} - \lambda_{n_{i}}e^{(-c+\epsilon)\lambda_{n_{i}}}\right\} = \exp\left\{-\sum_{t=1}^{j+1}\lambda_{n_{i}}^{t}e^{(-c+\epsilon)\lambda_{n_{i}}}\right\},\nonumber
\end{align}
thus proving \eqref{F_{i,lambda}(c_{k,lambda})_lower_bound} by induction. Setting $j=k$, since $c_{k,\lambda_{n_{i}}}$ is the fixed point of $F_{k,\lambda_{n_{i}}}$, we have, for $i \geqslant i_{\epsilon}$,
\begin{equation}
c_{k,\lambda_{n_{i}}} = F_{k,\lambda_{n_{i}}}\left(c_{k,\lambda_{n_{i}}}\right) > \exp\left\{-\sum_{t=1}^{k}\lambda_{n_{i}}^{t}e^{(-c+\epsilon)\lambda_{n_{i}}}\right\}.\nonumber
\end{equation}
We know from \S\ref{subsec:proof_of_main_4_part_1} that the left side of the above inequality goes to $0$ as $i \rightarrow \infty$, whereas the right side approaches $1$, since $\lambda_{n_{i}}^{t}e^{(-c+\epsilon)\lambda_{n_{i}}} \rightarrow 0$ as $\lambda_{n_{i}} \rightarrow \infty$ for every $1 \leqslant t \leqslant k$ (since $-c+\epsilon < 0$). This yields the desired contradiction.

\subsubsection{Concluding about the limit of $F_{j,\lambda}(c_{k,\lambda})$, for all $1 \leqslant j \leqslant k-1$}\label{subsec:proof_of_main_4_part_2_subpart_2} The contradiction obtained in \S\ref{subsec:proof_of_main_4_part_2_subpart_1} tells us that our assumption in \eqref{contradiction_1} is wrong, which in turn implies that
\begin{equation}\label{contradiction_1_corrected}
\lim_{\lambda \rightarrow \infty}\frac{\ln c_{k,\lambda}}{\lambda} = 0.
\end{equation}
By \eqref{claim_3} and \eqref{contradiction_1_corrected}, we conclude that 
\begin{equation}\label{k-1_limit}
\lim_{\lambda \rightarrow \infty}F_{k-1,\lambda}(c_{k,\lambda}) = 1.
\end{equation}
Recall that in the proof of \eqref{belongs_to_D_{i}}, we have shown that $F_{i,\lambda}(x) \geqslant F_{i+1,\lambda}(x)$ for all $0 \leqslant i \leqslant k-1$ and $x \in [0,c_{k-1,\lambda}]$. Using this fact and \eqref{k-1_limit}, we obtain 
\begin{equation}\label{F_{i}(c_{k,lambda})_limit}
1 > F_{1,\lambda}(c_{k,\lambda}) \geqslant F_{2,\lambda}(c_{k,\lambda}) \geqslant \cdots \geqslant F_{k-1,\lambda}(c_{k,\lambda}) \text{ for } \lambda > 0 \implies \lim_{\lambda \rightarrow \infty}F_{i,\lambda}(c_{k,\lambda}) = 1 \text{ for } 1 \leqslant i \leqslant k-1.
\end{equation}

\subsection{The behaviour of $\lambda^{i}c_{k,\lambda}$, for $1 \leqslant i \leqslant k$, as functions of $\lambda$}\label{subsec:proof_of_main_4_part_3} We lay down the final steps of the proof of \eqref{lambda^{i}c_{k,lambda}_limit_various_i}. Setting $i=1$ in \eqref{F_{i}(c_{k,lambda})_limit} and using \eqref{F_{i,lambda}_extended}, we have $\lim_{\lambda \rightarrow \infty}\lambda c_{k,\lambda} = 0$. We show, via induction on $i$, that 
\begin{equation}\label{F_{i}(c_{k,lambda})_order_of_magnitude}
F_{i,\lambda}(c_{k,\lambda}) = \exp\{-\lambda^{i}c_{k,\lambda}(1 + O(\lambda^{-1}))\} \text{ for } 1 \leqslant i \leqslant k-1.
\end{equation}
Assuming that \eqref{F_{i}(c_{k,lambda})_order_of_magnitude} holds for some $i \leqslant k-2$, we have 
\begin{equation}\label{ind_hyp_1}
\lim_{\lambda \rightarrow \infty}\lambda^{i}c_{k,\lambda} = 0
\end{equation}
due to \eqref{F_{i}(c_{k,lambda})_limit}, which further ensures that $c_{k,\lambda} = O(\lambda^{-i})$ for all $\lambda$ large enough. Using this fact, \eqref{F_{i,lambda}_extended} and a Taylor expansion, we have
\begin{align}\label{F_{i+1}(c_{k,lambda})_order_of_magnitude_inductive_step}
F_{i+1,\lambda}(c_{k,\lambda}) &= \exp\left\{\lambda \left[e^{-\lambda^{i}c_{k,\lambda}(1 + O(\lambda^{-1}))} - c_{k,\lambda} - 1\right]\right\} \nonumber\\
&=\exp\left\{\lambda\left[-\lambda^{i} c_{k,\lambda}(1 + O(\lambda^{-1})) + O\left(\left\{\lambda^{i} c_{k,\lambda}(1 + O(\lambda^{-1}))\right\}^{2}\right) - c_{k,\lambda}\right]\right\} \nonumber\\
&= \exp\left\{-\lambda^{i+1}c_{k,\lambda}(1 + O(\lambda^{-1}))\left[1 + O\left(\lambda^{i}c_{k,\lambda}(1 + O(\lambda^{-1}))\right) + O(\lambda^{-i})\right]\right\} \nonumber\\
&= \exp\left\{-\lambda^{i+1}c_{k,\lambda}\left[1 + O(\lambda^{-1}) + O(\lambda^{i}c_{k,\lambda}) + O(\lambda^{-i})\right]\right\}
\end{align}
for all $\lambda$ sufficiently large. Note that each of the terms $O(\lambda^{-1})$, $O(\lambda^{i}c_{k,\lambda})$ (by \eqref{ind_hyp_1}) and $O(\lambda^{-i})$ is $o(1)$ as $\lambda \rightarrow \infty$, so that the dominant term in the exponent of \eqref{F_{i+1}(c_{k,lambda})_order_of_magnitude_inductive_step} is $-\lambda^{i+1}c_{k,\lambda}$. This fact, along with \eqref{F_{i}(c_{k,lambda})_limit}, yields $\lim_{\lambda \rightarrow \infty}\lambda^{i+1}c_{k,\lambda} = 0$. This in turn yields $\lambda^{i}c_{k,\lambda} = o(\lambda^{-1})$, so that the dominant term out of $O(\lambda^{-1})$, $O(\lambda^{i}c_{k,\lambda})$ and $O(\lambda^{-i})$ in \eqref{F_{i+1}(c_{k,lambda})_order_of_magnitude_inductive_step} is $O(\lambda^{-1})$. This completes the proof of \eqref{F_{i}(c_{k,lambda})_order_of_magnitude} by induction. 

Combining \eqref{F_{i}(c_{k,lambda})_limit} and \eqref{F_{i}(c_{k,lambda})_order_of_magnitude} for $i=k-1$, we conclude that $\lim_{\lambda \rightarrow \infty}\lambda^{k-1}c_{k,\lambda} = 0$. Using this fact and setting $i=k-1$ in \eqref{F_{i+1}(c_{k,lambda})_order_of_magnitude_inductive_step}, we have $c_{k,\lambda} = F_{k,\lambda}(c_{k,\lambda}) = \exp\{-\lambda^{k}c_{k,\lambda}[1 + O(\lambda^{-1}) + O(\lambda^{k-1}c_{k,\lambda})]\}$, which leads to $\lambda^{k}c_{k,\lambda} \rightarrow \infty$ because of the final conclusion of \S\ref{subsec:proof_of_main_4_part_1}. The two conclusions drawn in this paragraph bring us to the end of the proof of \eqref{lambda^{i}c_{k,lambda}_limit_various_i}.

\subsection{Stating the lemmas and connecting the dots}\label{subsec:proof_of_main_4_part_4} As promised in the outline of our argument chalked out right before \eqref{lambda^{i}c_{k,lambda}_limit_various_i}, we now state two important lemmas (whose proofs are deferred to \S\ref{appsec:proof_of_main_4} of the Appendix).
\begin{lemma}\label{lem:draw_probab_limit_lemma_1}
Let $\{r_{i}\}_{0 \leqslant i \leqslant k}$ be a sequence of functions defined and differentiable on an interval $I$, with $(r_{i}(x), r_{k-j+1}(x), r_{k-j+2}(x), \ldots, r_{k}(x)) \in \mathcal{D}_{j}$ (see \eqref{D_{i}_domain_defn}) for all $x \in I$ and all $0 \leqslant i < i+j \leqslant k$. Then
\begin{align}
&\frac{d}{dx}g_{k,\lambda}(r_{0}(x),r_{1}(x),\ldots,r_{k}(x)) = \lambda^{k-1}\sum_{i=0}^{k-1}f_{k,i,\lambda}(r_{0}(x),r_{1}(x),\ldots,r_{k}(x))(r'_{i}(x) - r'_{k}(x)), \label{g_{k}_derivative_general_Poisson}
\end{align}
in which
\begin{align}
& f_{k,0,\lambda}(r_{0}(x),r_{1}(x),\ldots,r_{k}(x)) = \prod_{t=1}^{k-1}G_{\lambda}\left(g_{t,\lambda}\left(r_{0}(x),r_{k-t+1}(x),r_{k-t+2}(x),\ldots,r_{k}(x)\right)\right) \label{f_{k,0,lambda}_detailed}
\end{align}
and
\begin{align}
& f_{k,i,\lambda}(r_{0}(x),r_{1}(x),\ldots,r_{k}(x)) = \prod_{t=1}^{k-i}G_{\lambda}\left(g_{t,\lambda}(r_{i}(x), r_{k-t+1}(x), r_{k-t+2}(x), \ldots, r_{k}(x))\right) \alpha_{k,i,\lambda}\left(r_{0}(x),r_{1}(x),\ldots,r_{k}(x)\right) \label{f_{k,i,lambda}_detailed}
\end{align}
for $1 \leqslant i \leqslant k-1$, where $\alpha_{k,1,\lambda}(r_{0}(x),r_{1}(x),\ldots,r_{k}(x)) = -1$ and $\left|\alpha_{k,i,\lambda}(r_{0}(x),r_{1}(x),\ldots,r_{k}(x))\right|$ is bounded above by a constant $a_{k,i}$ that depends on $k$ and $i$ but not on $\lambda$ nor on the functions $r_{0}, \ldots, r_{k}$, for $2 \leqslant i \leqslant k$. 
\end{lemma}

\begin{lemma}\label{lem:F_{i}_derivatives}
For $i \geqslant 1$ and $x \in (0,c_{i-1})$, we have 
\begin{equation}
F'_{i,\lambda}(x) = -\lambda\left[\sum_{t=1}^{i-1}\lambda^{i-t}\prod_{j=t}^{i-1}F_{j,\lambda}(x)+1\right]F_{i,\lambda}(x).
\end{equation}
\end{lemma}

We set $r_{i}(x)$ to be the function $F_{i,\lambda}(x)$ for each $i$, and $x = c_{k,\lambda}$, in Lemma~\ref{lem:draw_probab_limit_lemma_1}. Recall, from the outline chalked out right above \S\ref{subsec:proof_of_main_4_part_1}, that we aim to show that the leading term in \eqref{g_{k}_derivative_general_Poisson}, in this case, will be of the same order of magnitude as $\lambda^{2k-1}c_{k,\lambda}$, so that the leading term of $H'_{k,\lambda}(c_{k,\lambda})$ is shown to be of the same order of magnitude as $\lambda^{2k}c_{k,\lambda}^{2}$. To this end, from \eqref{g_{j}_c_{k}_patterns}, we have 
\begin{equation}
G_{\lambda}(g_{k-i,\lambda}(F_{i,\lambda}(c_{k,\lambda}),F_{i+1,\lambda}(c_{k,\lambda}),\ldots,F_{k,\lambda}(c_{k,\lambda}))) = G_{\lambda}(F_{k-1,\lambda}(c_{k,\lambda}) - c_{k,\lambda}) = F_{k,\lambda}(c_{k,\lambda}) = c_{k,\lambda},\nonumber
\end{equation}
so that from \eqref{f_{k,i,lambda}_detailed}, using the fact that both $\left|\alpha_{k,i,\lambda}(r_{0}(x),r_{1}(x),\ldots,r_{k}(x))\right|$ and $G_{\lambda}(x)$ are $O(1)$, we obtain
\begin{equation}
f_{k,i,\lambda}(r_{0}(x),r_{1}(x),\ldots,r_{k}(x)) = O(c_{k,\lambda}) \text{ for } 1 \leqslant i \leqslant k-1.\label{claim_4}
\end{equation}
Next, from Lemma~\ref{lem:F_{i}_derivatives} and \eqref{F_{i}(c_{k,lambda})_limit}, and using \eqref{claim_4} in the second step, we deduce that 
\begin{align}\label{assemble_1}
& F'_{i,\lambda}(c_{k,\lambda}) = O\left(\lambda^{i} \prod_{j=1}^{i}F_{j,\lambda}(c_{k,\lambda})\right) = O(\lambda^{i}) \text{ as }\lambda \rightarrow \infty, \text{ for }1 \leqslant i \leqslant k-1\nonumber\\
& \implies \sum_{i=1}^{k-1}f_{k,i,\lambda}\left(F_{0,\lambda}(c_{k,\lambda}),F_{1,\lambda}(c_{k,\lambda}),\ldots,F_{k,\lambda}(c_{k,\lambda})\right)F'_{i,\lambda}(c_{k,\lambda}) = \sum_{i=1}^{k-1}O\left(\lambda^{i}c_{k,\lambda}\right).
\end{align}
Next, for any $0 < \epsilon < 1$, Lemma~\ref{lem:F_{i}_derivatives}, the fact that $c_{k,\lambda}$ is the fixed point of $F_{k,\lambda}$, and \eqref{F_{i}(c_{k,lambda})_limit} together yield 
\begin{align}
O(\lambda^{k}c_{k,\lambda}) \geqslant -F'_{k,\lambda}(c_{k,\lambda}) = \lambda^{k} \prod_{j=1}^{k}F_{j,\lambda}(c_{k,\lambda}) + O\left(\sum_{t=2}^{k}\lambda^{k-t+1}\prod_{j=t}^{k}F_{j,\lambda}(c_{k,\lambda})\right) \geqslant \lambda^{k}c_{k,\lambda}(1-\epsilon) + O(\lambda^{k-1} c_{k,\lambda})\label{claim_5}
\end{align}
as $\lambda \rightarrow \infty$. From \eqref{claim_4} and \eqref{claim_5}, we deduce that 
\begin{equation}\label{assemble_2}
\sum_{i=1}^{k-1}f_{k,i,\lambda}(F_{0,\lambda}(c_{k,\lambda}),F_{1,\lambda}(c_{k,\lambda}),\ldots,F_{k,\lambda}(c_{k,\lambda}))(-F'_{k,\lambda}(c_{k,\lambda})) = O(\lambda^{k} c_{k,\lambda}^{2}).
\end{equation}
Finally, \eqref{f_{k,0,lambda}_detailed}, \eqref{g_{j}_c_{k}_patterns} and \eqref{F_{i}(c_{k,lambda})_limit} together yield, as $\lambda \rightarrow \infty$, 
\begin{equation}\label{claim_6}
f_{k,0,\lambda}\left(F_{0,\lambda}(c_{k,\lambda}),F_{1,\lambda}(c_{k,\lambda}),\ldots,F_{k,\lambda}(c_{k,\lambda})\right) = \prod_{t=1}^{k-1}G_{\lambda}\left(F_{t-1,\lambda}(c_{k,\lambda})-c_{k,\lambda}\right) = \prod_{t=1}^{k-1}F_{t,\lambda}(c_{k,\lambda}) \rightarrow 1.
\end{equation}
Hence, from \eqref{claim_5} and \eqref{claim_6}, for any $0 < \epsilon < 1$, as $\lambda \rightarrow \infty$, we have
\begin{equation}\label{assemble_3}
f_{k,0,\lambda}(F_{0,\lambda}(c_{k,\lambda}),F_{1,\lambda}(c_{k,\lambda}),\ldots,F_{k,\lambda}(c_{k,\lambda})) (-F'_{k,\lambda}(c_{k,\lambda})) \geqslant (1-\epsilon)\lambda^{k}c_{k,\lambda} + O(\lambda^{k-1} c_{k,\lambda}).
\end{equation}
Substituting \eqref{assemble_1}, \eqref{assemble_2} and \eqref{assemble_3} in \eqref{g_{k}_derivative_general_Poisson}, using the fact that $G_{\lambda}(x) = e^{\lambda(x-1)}$ so that $G'_{\lambda}(x) = \lambda G_{\lambda}(x)$, and using Corollary~\ref{cor:c_{k}_fixed_point}, we have, as $\lambda \rightarrow \infty$, 
\begin{align}
&H'_{k,\lambda}(c_{k,\lambda}) = G'_{\lambda}(g_{k,\lambda}(F_{0,\lambda}(c_{k,\lambda}),F_{1,\lambda}(c_{k,\lambda}),\ldots,F_{k,\lambda}(c_{k,\lambda})))\frac{d}{dx}g_{k,\lambda}(F_{0,\lambda}(x),F_{1,\lambda}(x),\ldots,F_{k,\lambda}(x))\big|_{x=c_{k,\lambda}}\nonumber\\
&\geqslant \lambda H_{k,\lambda}(c_{k,\lambda})\lambda^{k-1}\left[\sum_{i=1}^{k-1}O(\lambda^{i}c_{k,\lambda}) + O(\lambda^{k} c_{k,\lambda}^{2}) + (1-\epsilon)\lambda^{k}c_{k,\lambda} + O(\lambda^{k-1} c_{k,\lambda})\right] \nonumber\\
&= \lambda^{k} c_{k,\lambda}\left[O(\lambda^{k-1}c_{k,\lambda}) + O(\lambda^{k} c_{k,\lambda}^{2}) + (1-\epsilon)\lambda^{k}c_{k,\lambda}\right] = (1-\epsilon)\lambda^{2k}c_{k,\lambda}^{2} + O(\lambda^{2k-1}c_{k,\lambda}^{2}),\nonumber
\end{align}
so that the leading term of $H'_{k,\lambda}(c_{k,\lambda})$ is indeed of the same order of magnitude as $\lambda^{2k}c_{k,\lambda}^{2}$, and $H'_{k,\lambda}(c_{k,\lambda}) \rightarrow \infty$ due to the second assertion made in \eqref{lambda^{i}c_{k,lambda}_limit_various_i}. This concludes the proof of Theorem~\ref{main:normal_draw_probab_limit_Poisson}.

\section{Proof of Theorem~\ref{main:Poisson_k=2_normal_phase_transition}}\label{sec:Poisson_k=2}
\subsection{Showing strict convexity of $H_{2,\lambda}$ on $[0,c_{2,\lambda}]$ for $\lambda \geqslant 2$}\label{subsec:Poisson_k=2_part_1} We begin by stating the first of the three objectives we wish to achieve in \S\ref{sec:Poisson_k=2}. We fix $k=2$, and we let the offspring distribution $\chi$ of $\mathcal{T}_{\chi}$ be Poisson$(\lambda)$. We show that the curve $y = H_{2,\lambda}(x)$ is strictly convex for all $x \in [0,c_{2,\lambda}]$, whenever $\lambda \geqslant 2$ -- we accomplish this by proving that the second derivative $H''_{2,\lambda}(x)$ is strictly positive for all $x \in [0,c_{2,\lambda}]$, for every $\lambda \geqslant 2$.

In order to keep the expression for $H''_{2,\lambda}(x)$ as uncluttered as possible, we set 
\begin{equation}
\alpha(x) = G_{\lambda}(1-F_{2,\lambda}(x)) \text{ and } \beta(x) = G_{\lambda}(F_{1,\lambda}(x) - F_{2,\lambda}(x)),\nonumber
\end{equation}
so that, from \eqref{H_{k}_defn}, we obtain 
\begin{equation}
H_{2,\lambda}(x) = G_{\lambda}(\alpha(x) - \beta(x)).\nonumber
\end{equation}
Note, at the very outset, that since $F_{1,\lambda}(x) \leqslant 1$ (evident from \eqref{F_{i,lambda}_extended}) and $G_{\lambda}$ is increasing, hence $\alpha(x) \geqslant \beta(x)$ for $x \in [0,c_{1,\lambda}]$. From Lemma~\ref{lem:F_{i}_derivatives}, which yields $F'_{1,\lambda}(x) = -\lambda F_{1,\lambda}(x)$ and $F'_{2,\lambda}(x) = -\lambda\left(\lambda F_{1,\lambda}(x) + 1\right) F_{2,\lambda}(x)$, we have
\begin{align}
&\alpha'(x) = \lambda^{2}\alpha(x)\left(\lambda F_{1,\lambda}(x) + 1\right) F_{2,\lambda}(x),\nonumber\\
&\beta'(x) = \lambda^{2}\beta(x)\left\{-F_{1,\lambda}(x) + \left(\lambda F_{1,\lambda}(x) + 1\right) F_{2,\lambda}(x)\right\}.\nonumber
\end{align}
Utilizing these expressions, we have
\begin{align}
H'_{2,\lambda}(x) &= \frac{d}{dx} G_{\lambda}(\alpha(x) - \beta(x)) = \lambda H_{2,\lambda}(x) (\alpha'(x) - \beta'(x)),\label{H_{2,lambda}_derivative}
\end{align}
and substituting the expressions for $\alpha'(x)$ and $\beta'(x)$ in \eqref{H_{2,lambda}_derivative}, then differentiating again,
\begin{align}
H''_{2,\lambda}(x) &= \lambda^{4} H_{2,\lambda}(x) \big[\lambda^{2}\{\alpha(x) - \beta(x)\}^{2}\left(\lambda F_{1,\lambda}(x) + 1\right)^{2}(F_{2,\lambda}(x))^{2} + \lambda^{2}(\beta(x))^{2}(F_{1,\lambda}(x))^{2} +\nonumber\\& 2\lambda^{2}\beta(x)\{\alpha(x) - \beta(x)\}\left(\lambda F_{1,\lambda}(x) + 1\right)F_{1,\lambda}(x)F_{2,\lambda}(x) + \lambda \{\alpha(x) - \beta(x)\} \left(\lambda F_{1,\lambda}(x) + 1\right)^{2} (F_{2,\lambda}(x))^{2} \nonumber\\& - \lambda \{\alpha(x) - \beta(x)\} F_{1,\lambda}(x) F_{2,\lambda}(x) - \{\alpha(x) - \beta(x)\}(\lambda F_{1,\lambda}(x)+1)^{2} F_{2,\lambda}(x) \nonumber\\&+ 2\lambda \beta(x) \left(\lambda F_{1,\lambda}(x) + 1\right) F_{1,\lambda}(x) F_{2,\lambda}(x) - \beta(x)F_{1,\lambda}(x)(\lambda F_{1,\lambda}(x)+1)\big]\label{H_{2,lambda}_double_derivative}\\
&= \lambda^{4} H_{2,\lambda}(x) \big[A_{1} + A_{2} + 2A_{3} + A_{4} - A_{5} - A_{6} + 2A_{7} - A_{8}\big],\label{H_{2,lambda}_double_derivative_abbreviated}
\end{align} 
where
\begin{align}
& A_{1} = \lambda^{2}\{\alpha(x) - \beta(x)\}^{2}\left(\lambda F_{1,\lambda}(x) + 1\right)^{2}(F_{2,\lambda}(x))^{2},\nonumber\\
& A_{2} = \lambda^{2}(\beta(x))^{2}(F_{1,\lambda}(x))^{2},\nonumber\\
& A_{3} = \lambda^{2}\beta(x)\{\alpha(x) - \beta(x)\}\left(\lambda F_{1,\lambda}(x) + 1\right)F_{1,\lambda}(x)F_{2,\lambda}(x),\nonumber\\
& A_{4} = \lambda \{\alpha(x) - \beta(x)\} \left(\lambda F_{1,\lambda}(x) + 1\right)^{2} (F_{2,\lambda}(x))^{2},\nonumber\\
& A_{5} = \lambda \{\alpha(x) - \beta(x)\} F_{1,\lambda}(x) F_{2,\lambda}(x),\nonumber\\
& A_{6} = \{\alpha(x) - \beta(x)\}(\lambda F_{1,\lambda}(x)+1)^{2} F_{2,\lambda}(x),\nonumber\\
& A_{7} = \lambda \beta(x) \left(\lambda F_{1,\lambda}(x) + 1\right) F_{1,\lambda}(x) F_{2,\lambda}(x),\nonumber\\
& A_{8} = \beta(x)F_{1,\lambda}(x)(\lambda F_{1,\lambda}(x)+1),\nonumber
\end{align}
and each $A_{i}$ is non-negative (in fact, strictly positive except for $A_{1}, A_{3}, A_{4}, A_{5}$ and $A_{6}$ at $x=0$). Thus, our aim now is to establish that the sum within the square brackets in \eqref{H_{2,lambda}_double_derivative_abbreviated} is strictly positive for each $x \in [0,c_{2,\lambda}]$, for every $\lambda \geqslant 2$. This is where Lemmas~\ref{convexity_lem_1} through \ref{convexity_lem_5}, all of whose proofs are deferred to \S\ref{appsec:Poisson_k=2} of the Appendix, come in. The objective each of them accomplishes is the collection of various terms from \eqref{H_{2,lambda}_double_derivative_abbreviated} and showing that their sums are strictly positive for $\lambda \geqslant 2$. 

The primary idea we employ here is as follows: we split the interval $[0,c_{2,\lambda}]$ into three pairwise disjoint sub-intervals (in some cases, we may combine two consecutive sub-intervals), namely $[0,\delta_{\lambda}]$, $(\delta_{\lambda}, \gamma_{\lambda}]$ and $(\gamma_{\lambda}, c_{2,\lambda}]$, where we define $\gamma_{\lambda}$ and $\delta_{\lambda}$ as follows:
\begin{equation} 
F_{2,\lambda}(\gamma_{\lambda}) = \frac{1}{\lambda} \text{ and } F_{2,\lambda}(\delta_{\lambda}) = \frac{5}{4\lambda}.\label{gamma_delta_defn}
\end{equation}
Evidently, to be able to do the above, we require 
\begin{equation}\label{gamma_delta_c_ineq}
0 < \delta_{\lambda} < \gamma_{\lambda} < c_{2,\lambda}.
\end{equation}

Several aspects of the above paragraph need justification right away, before we can proceed. By Lemma~\ref{lem:main_thm_1_1}, we know that $F_{2,\lambda}$ is strictly decreasing on $[0,c_{1,\lambda}]$, and $F_{2,\lambda}(0) = 1$. By definition of $c_{2,\lambda}$, we have $F_{2,\lambda}(c_{2,\lambda}) = c_{2,\lambda}$. Therefore, if we can show that
\begin{equation}
c_{2,\lambda} < \frac{1}{\lambda} < \frac{5}{4\lambda} < 1 \text{ for all } \lambda \geqslant 2,\label{gamma_delta_c_ineq_requirement}
\end{equation}
of which the second and third inequalities are immediately seen to hold, we can conclude that $\gamma_{\lambda}$ and $\delta_{\lambda}$, defined via \eqref{gamma_delta_defn}, exist and are unique, and that \eqref{gamma_delta_c_ineq} holds as well. This is where Lemma~\ref{lem:lambda c_{2,lambda} behaviour} proves useful:
\begin{lemma}\label{lem:lambda c_{2,lambda} behaviour}
The function $\eta_{\lambda} = \lambda c_{2,\lambda}$ is strictly increasing for $\lambda \in (0,\lambda_{0})$ and strictly decreasing for $\lambda \in (\lambda_{0}, \infty)$, where $\lambda_{0} \approx 2.43634$. The maximum value of $\eta_{\lambda}$ is $\approx 0.52839925$.
\end{lemma}
The proof of this lemma is deferred to \S\ref{appsec:Poisson_k=2}. This lemma guarantees that for all $\lambda > 0$, we have $\eta_{\lambda} < 1 \implies c_{2,\lambda} < \frac{1}{\lambda}$, thus proving that \eqref{gamma_delta_c_ineq_requirement} indeed holds.

Let us come back to chalking out an outline of our argument for showing that the expression in \eqref{H_{2,lambda}_double_derivative_abbreviated} is strictly positive for all $\lambda \geqslant 2$: for each $\lambda \geqslant 2$, on each of the sub-intervals $[0,\delta_{\lambda}]$, $(\delta_{\lambda}, \gamma_{\lambda}]$ and $(\gamma_{\lambda}, c_{2,\lambda}]$, we group the terms within the square brackets in the expression of \eqref{H_{2,lambda}_double_derivative_abbreviated} judiciously, so that the sum of the terms in each such group is strictly positive. We are now ready to state the lemmas that help accomplish this task. 

\begin{lemma}\label{convexity_lem_1}
For each $\lambda \geqslant 2.5$, for all $x \in (\gamma_{\lambda},c_{2,\lambda}]$, we have $A_{1} + A_{4} - A_{6} > 0$.
\end{lemma}
\begin{lemma}\label{convexity_lem_2}
For each $\lambda \geqslant 2$ and all $x \in (\gamma_{\lambda},c_{2,\lambda}]$, we have $A_{3} - A_{5} > 0$.
\end{lemma}
\begin{lemma}\label{convexity_lem_3}
For $\lambda \geqslant 2$ and $x \in (\gamma_{\lambda},c_{2,\lambda}]$, we have $A_{3} + 2A_{7} - A_{8} > 0$.
\end{lemma}
Before we state the last couple of lemmas, we note that if we add up the expressions to the left of the inequalities in Lemmas~\ref{convexity_lem_1}, \ref{convexity_lem_2} and \ref{convexity_lem_3}, we can conclude that
\begin{equation}
A_{1} + 2A_{3} + A_{4} - A_{5} - A_{6} + 2A_{7} - A_{8} > 0 \text{ for all } x \in (\gamma_{\lambda},c_{2,\lambda}], \text{ for }  \lambda \geqslant 2.5.\label{convexity_case_1}
\end{equation}

\begin{lemma}\label{convexity_lem_4}
For $2 \leqslant \lambda < 2.5$ and $x \in (\gamma_{\lambda},c_{2,\lambda}]$, we have $A_{1} + A_{4} - A_{6} + 2A_{3} - A_{5} + 2A_{7} - A_{8} > 0$.
\end{lemma}
\begin{lemma}\label{convexity_lem_5}
For $\lambda \geqslant 2$ and $x \in [\delta_{\lambda}, \gamma_{\lambda}]$, we have $2A_{3} - A_{5} > 0$.
\end{lemma}

Next, we note that as $F_{2,\lambda}$ is strictly decreasing (Lemma~\ref{lem:main_thm_1_1}), we have $F_{2,\lambda}(x) \geqslant \frac{1}{\lambda}$ for $x \in [0,\gamma_{\lambda}]$, so that
\begin{align}\label{convexity_case_2}
& A_{4} - A_{6} \geqslant 0 \text{ and } 2A_{7} - A_{8} > 0.
\end{align}
Likewise, we have $F_{2,\lambda}(x) \geqslant \frac{5}{4\lambda}$ for $x \in [0,\delta_{\lambda}]$, so that 
\begin{align}
&A_{4} - A_{5} - A_{6} \geqslant \frac{1}{4}\{\alpha(x) - \beta(x)\} \left(\lambda F_{1,\lambda}(x) + 1\right)^{2} F_{2,\lambda}(x) - \lambda \{\alpha(x) - \beta(x)\} F_{1,\lambda}(x) F_{2,\lambda}(x) > 0\label{convexity_case_3}
\end{align}
by an application of the AM-GM inequality. 

We are now ready to consolidate all of the findings above to achieve the desired conclusion. Combining \eqref{convexity_case_1} and Lemma~\ref{convexity_lem_4}, we obtain
\begin{equation}
A_{1} + 2A_{3} + A_{4} - A_{5} - A_{6} + 2A_{7} - A_{8} > 0 \text{ for all } x \in (\gamma_{\lambda},c_{2,\lambda}], \text{ for }  \lambda \geqslant 2.\label{convexity_case_4}
\end{equation}
Adding the inequalities in Lemma~\ref{convexity_lem_5} and \eqref{convexity_case_2} (for the shorter interval $(\delta_{\lambda}, \gamma_{\lambda}]$), we obtain
\begin{align}
2A_{3} - A_{5} + A_{4} - A_{6} + 2A_{7} - A_{8} > 0 \text{ for all } x \in (\delta_{\lambda},\gamma_{\lambda}], \text{ for }  \lambda \geqslant 2.\label{convexity_case_5}
\end{align}
Adding the inequality in \eqref{convexity_case_3} to the second inequality in \eqref{convexity_case_2} (again, for the shorter interval $[0,\delta_{\lambda}]$), we obtain
\begin{equation}
2A_{7} - A_{8} + A_{4} - A_{5} - A_{6} > 0 \text{ for all } x \in [0,\delta_{\lambda}], \text{ for } \lambda \geqslant 2.\label{convexity_case_6}
\end{equation}
Combining the conclusions of \eqref{convexity_case_4}, \eqref{convexity_case_5} and \eqref{convexity_case_6}, we complete the proof of the desired claim that the expression in \eqref{H_{2,lambda}_double_derivative_abbreviated} is strictly positive for all $x \in [0,c_{2,\lambda}]$, for all $\lambda \geqslant 2$. This concludes the proof of the fact that $H_{2,\lambda}$ is strictly convex on $[0,c_{2,\lambda}]$, for all $\lambda \geqslant 2$.

\subsection{Studying the behaviour of the slope of $H_{2,\lambda}$ at $c_{2,\lambda}$ as a function of $\lambda$}\label{subsec:Poisson_k=2_part_2} The second objective of \S\ref{sec:Poisson_k=2} is to show that the slope of $H_{2,\lambda}$ at $x = c_{2,\lambda}$ is strictly increasing in $\lambda$ for all $\lambda \geqslant 1$. We start by noting, since $c_{2,\lambda}$ is the fixed point of $F_{2,\lambda}$, that $\alpha(c_{2,\lambda}) = F_{1,\lambda}(c_{2,\lambda})$ whereas $\beta(c_{2,\lambda}) = c_{2,\lambda}$ (using \eqref{F_{i}_defn}). Using these observations, the expression from \eqref{H_{2,lambda}_derivative}, the conclusion of Corollary~\ref{cor:c_{k}_fixed_point} and the notation $\eta_{\lambda} = \lambda c_{2,\lambda}$ introduced in Lemma~\ref{lem:lambda c_{2,lambda} behaviour}, we obtain
\begin{align}
&H'_{2,\lambda}(c_{2,\lambda}) = \lambda H_{2,\lambda}(c_{2,\lambda}) (\alpha'(c_{2,\lambda}) - \beta'(c_{2,\lambda}))
= \lambda^{2} \eta_{\lambda}^{2} e^{-2\eta_{\lambda}} + 2\lambda \eta_{\lambda}^{2} e^{-\eta_{\lambda}} - \lambda \eta_{\lambda}^{3} e^{-\eta_{\lambda}} - \eta_{\lambda}^{3}.\label{H'_{2,lambda}(c_{2,lambda})}
\end{align}
For ease of computation, we perform a term-by-term differentiation of the expression in \eqref{H'_{2,lambda}(c_{2,lambda})} with respect to $\lambda$, and substitute from \eqref{lambda c_{2,lambda} derivative} the expression for $\eta'_{\lambda}$ (the derivative of $\eta_{\lambda}$ with respect to $\lambda$). This yields
\begin{align}
&\frac{d}{d\lambda}[\lambda^{2} \eta_{\lambda}^{2} e^{-2\eta_{\lambda}}] 
= \frac{4\lambda \eta_{\lambda}^{2} e^{-2\eta_{\lambda}} + 2 \lambda^{2}\eta_{\lambda}^{2}e^{-3\eta_{\lambda}} - 2\lambda^{2}\eta_{\lambda}^{2}e^{-2\eta_{\lambda}} + 2\lambda^{2}\eta_{\lambda}^{3}e^{-2\eta_{\lambda}}}{1 + \lambda \eta_{\lambda} e^{-\eta_{\lambda}} + \eta_{\lambda}};\nonumber\\
& \frac{d}{d\lambda}[2\lambda \eta_{\lambda}^{2} e^{-\eta_{\lambda}}] 
= \frac{6\eta_{\lambda}^{2}e^{-\eta_{\lambda}} + 4\lambda \eta_{\lambda}^{2} e^{-2\eta_{\lambda}} - 4\lambda \eta_{\lambda}^{2}e^{-\eta_{\lambda}} + 2\lambda \eta_{\lambda}^{3}e^{-\eta_{\lambda}}}{1 + \lambda \eta_{\lambda} e^{-\eta_{\lambda}} + \eta_{\lambda}};\nonumber\\
& \frac{d}{d\lambda}[-\lambda \eta_{\lambda}^{3} e^{-\eta_{\lambda}}] 
= \frac{-4 \eta_{\lambda}^{3} e^{-\eta_{\lambda}} - 3 \lambda \eta_{\lambda}^{3} e^{-2\eta_{\lambda}} + 3 \lambda \eta_{\lambda}^{3} e^{-\eta_{\lambda}} - \lambda \eta_{\lambda}^{4} e^{-\eta_{\lambda}}}{1 + \lambda \eta_{\lambda} e^{-\eta_{\lambda}} + \eta_{\lambda}};\nonumber\\
& \frac{d}{d\lambda}[-\eta_{\lambda}^{3}] 
= \frac{-3\eta_{\lambda}^{3} - 3\lambda \eta_{\lambda}^{3} e^{-\eta_{\lambda}} + 3\lambda \eta_{\lambda}^{3}}{\lambda\left(1 + \lambda \eta_{\lambda} e^{-\eta_{\lambda}} + \eta_{\lambda}\right)}.\nonumber
\end{align}
The numerator of $\frac{d}{d\lambda}H'_{2,\lambda}(c_{2,\lambda})$ is then given by
\begin{align}
& \eta_{\lambda}^{2}\big[4\lambda^{2} e^{-2\eta_{\lambda}} + 2 \lambda^{3}e^{-3\eta_{\lambda}} - 2\lambda^{3}e^{-2\eta_{\lambda}} + 2\lambda^{3}\eta_{\lambda} e^{-2\eta_{\lambda}} + 6\lambda e^{-\eta_{\lambda}} + 4\lambda^{2} e^{-2\eta_{\lambda}} - 4\lambda^{2} e^{-\eta_{\lambda}} +\nonumber\\& 2\lambda^{2} \eta_{\lambda}e^{-\eta_{\lambda}} - 4 \lambda \eta_{\lambda} e^{-\eta_{\lambda}} - 3 \lambda^{2} \eta_{\lambda} e^{-2\eta_{\lambda}} + 3 \lambda^{2} \eta_{\lambda} e^{-\eta_{\lambda}} - \lambda^{2} \eta_{\lambda}^{2} e^{-\eta_{\lambda}} - 3\eta_{\lambda} - 3\lambda \eta_{\lambda} e^{-\eta_{\lambda}} + 3\lambda \eta_{\lambda}\big]\nonumber\\
&= \eta_{\lambda}^{2}\big\{2\lambda^{3}e^{-2\eta_{\lambda}}[e^{-\eta_{\lambda}} - 1 + \eta_{\lambda}] + \lambda^{2}e^{-\eta_{\lambda}}[8e^{-\eta_{\lambda}} - 3\eta_{\lambda}e^{-\eta_{\lambda}} - 4 + 5\eta_{\lambda} - \eta_{\lambda}^{2}] \nonumber\\&+ \lambda e^{-\eta_{\lambda}}[6 - 7\eta_{\lambda}] + 3\eta_{\lambda}(\lambda - 1)\big\}\nonumber\\
&= \eta_{\lambda}^{2}\big\{2\lambda^{3}e^{-2\eta_{\lambda}}[e^{-\eta_{\lambda}} - 1 + \eta_{\lambda}] + \lambda^{2}e^{-\eta_{\lambda}}[8e^{-\eta_{\lambda}} - 4] + \lambda^{2}e^{-\eta_{\lambda}}[5\eta_{\lambda} - 3\eta_{\lambda}e^{-\eta_{\lambda}} - \eta_{\lambda}^{2}] \nonumber\\&+ \lambda e^{-\eta_{\lambda}}[6 - 7\eta_{\lambda}] + 3\eta_{\lambda}(\lambda - 1)\big\},\nonumber
\end{align}
and our aim now is to prove that this entire expression is strictly positive for $\lambda \geqslant 1$.

From Lemma~\ref{lem:lambda c_{2,lambda} behaviour}, we have $8e^{-\eta_{\lambda}} - 4 \geqslant 8 e^{-0.5284} - 4 \approx 0.7164 > 0$ and $6 - 7\eta_{\lambda} \geqslant 6 - 7 \cdot 0.5284 \approx 2.3012 > 0$. We also note that $5\eta_{\lambda} - 3\eta_{\lambda}e^{-\eta_{\lambda}} - \eta_{\lambda}^{2} = 3\eta_{\lambda}(1 - e^{-\eta_{\lambda}}) + \eta_{\lambda}(1-\eta_{\lambda}) + \eta_{\lambda} > 0$ and $e^{-\eta_{\lambda}} \geqslant 1 - \eta_{\lambda}$. Thus $\frac{d}{d\lambda}H'_{2,\lambda}(c_{2,\lambda}) > 0$ for $\lambda \geqslant 1$, hence proving the second part of Theorem~\ref{main:Poisson_k=2_normal_phase_transition}.

\subsection{Proving the third and final part of Theorem~\ref{main:Poisson_k=2_normal_phase_transition}}\label{subsec:Poisson_k=2_part_3}
The objective of \S\ref{subsec:Poisson_k=2_part_3} is to prove that there exists a critical value $\lambda_{c}$ of $\lambda$ such that $\nl_{2,\lambda} = c_{2,\lambda}$ (and hence, by Theorem~\ref{main:bounds_on_nl_{k}_ml_{k}}, $\nd_{2,\lambda} = 0$) for all $2 \leqslant \lambda < \lambda_{c}$, and $\nl_{2,\lambda} < c_{2,\lambda}$ (and hence, $\nd_{2,\lambda} > 0$) for all $\lambda > \lambda_{c}$.

The argument can be outlined as follows. The convexity of $H_{2,\lambda}$ on $[0,c_{2,\lambda}]$ for $\lambda \geqslant 2$, as proved in \S\ref{subsec:Poisson_k=2_part_1}, guarantees that $y = H_{2,\lambda}(x)$ intersects $y=x$ at most twice in $[0,c_{2,\lambda}]$. If the intersection happens twice, it must happen at the points $x=\nl_{2,\lambda}$ and $x=c_{2,\lambda}$, in which case $\nl_{2,\lambda} < c_{2,\lambda}$. If the intersection takes place only once, then this must be at the point $x=c_{2,\lambda}$, in which case $\nl_{2,\lambda} = c_{2,\lambda}$. Recall, from \S\ref{sec:proof_of_main_4}, the justification as to why $H'_{2,\lambda}(\nl_{2,\lambda}) \leqslant 1$ for every $\lambda > 0$.  

If we can show that $H'_{2,\lambda}(c_{2,\lambda})$ is strictly less than $1$ at $\lambda = 2$ and it is strictly greater than $1$ for some value of $\lambda$ strictly exceeding $2$, then using our conclusion from \S\ref{subsec:Poisson_k=2_part_2}, we can deduce that there is precisely one $\lambda_{c} > 2$ such that $H'_{2,\lambda}(c_{2,\lambda}) < 1$ for $\lambda < \lambda_{c}$ and $H'_{2,\lambda}(c_{2,\lambda}) > 1$ for $\lambda > \lambda_{c}$. In the former case, $c_{2,\lambda}$ has to be the only point of intersection between $y = H_{2,\lambda}(x)$ and $y=x$ within the interval $[0,c_{2,\lambda}]$ (this follows by noticing that the curve $y = H_{2,\lambda}(x)$ travels from above $y=x$ to below $y=x$ at $x=\nl_{2,\lambda}$, and travels from below $y=x$ to above $y=x$ at $x=c_{2,\lambda}$ when $\nl_{2,\lambda} < c_{2,\lambda}$, so that in such a scenario, we would have $H'_{2,\lambda}(c_{2,\lambda}) > 1$). In the latter case, there have to be two points of intersection between $y = H_{2,\lambda}(x)$ and $y=x$ within the interval $[0,c_{2,\lambda}]$.

At $\lambda = 2$, we have $H'_{2,\lambda}(c_{2,\lambda}) \approx 0.721129 < 1$, and at $\lambda = 2.5$, we have $H'_{2,\lambda}(c_{2,\lambda}) \approx 1.06445 > 1$. This accomplishes the task we set for ourselves in the previous paragraph. Solving $H'_{2,\lambda_{c}}\left(c_{2,\lambda_{c}}\right) = 1$ numerically, we obtain $\lambda_{c} \approx 2.41$. This accomplishes what we set out to prove in \S\ref{subsec:Poisson_k=2_part_3}, and concludes the proof of Theorem~\ref{main:Poisson_k=2_normal_phase_transition}. 

\begin{remark}
Although our proof technique does not cover the range $1 < \lambda < 2$, plotting $H_{2,\lambda}$ on $[0,c_{2,\lambda}]$ for various values of $\lambda$ in $(1,2)$ does seem to suggest that $H_{2,\lambda}$ is strictly convex on $[0,c_{2,\lambda}]$ for $\lambda \in (1,2)$ (see Figure~\ref{convex_lambda_1_2}). If this were to be proved analytically, it would give us a \emph{unique} threshold for the phase transition phenomenon, showing that $\nd_{2,\lambda} = 0$ for all $0 < \lambda < \lambda_{c}$ and $\nd_{2,\lambda} > 0$ for all $\lambda > \lambda_{c}$. 
\begin{figure}[h!]
  \centering
    \includegraphics[width=0.6\textwidth]{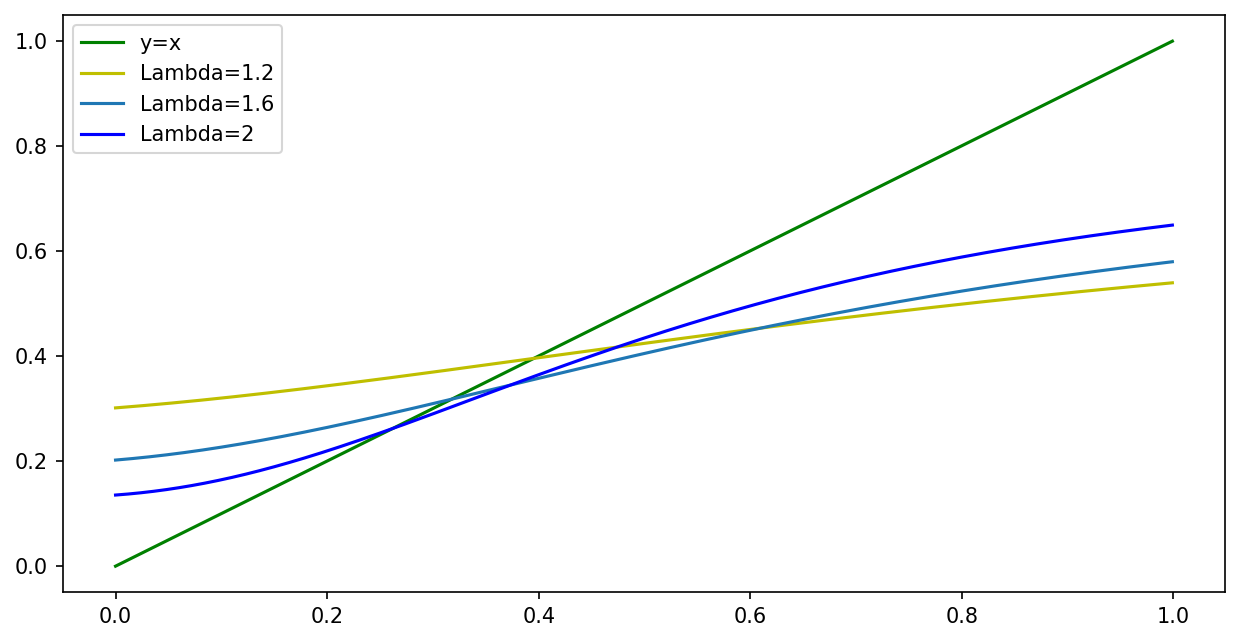}
\caption{$H_{2,\lambda}(x)$ strictly convex for $x \in [0,c_{2,\lambda}]$, for $\lambda = 1.2$ and $\lambda = 1.6$}
  \label{convex_lambda_1_2}
\end{figure}
\end{remark}

\section{Proofs of Theorems~\ref{lem:decay_rate_nl_{2,lambda}} and ~\ref{main:misere_normal_comparison}}\label{sec:misere_normal_comparison}
Recall that the two objectives we accomplish in \S\ref{sec:misere_normal_comparison} are 
\begin{enumerate*}
\item showing that $\lambda^{i}\nl_{2,\lambda} \rightarrow 0$ as $\lambda \rightarrow \infty$, for every $i \in \mathbb{N}$, which is a far stronger claim than the second part of Theorem~\ref{main:normal_draw_probab_limit_Poisson} as far as $k=2$ is concerned,
\item establishing the inequalities $\ml_{2,\lambda} \leqslant \nl_{2,\lambda}$, $\nd_{2,\lambda} < \md_{2,\lambda}$ and $\ml_{2,\lambda} \leqslant \nl_{2,\lambda} < \nw_{2,\lambda}$ for all $\lambda$ sufficiently large, giving us a means to compare the probabilities of the various outcomes of the $2$-jump normal game with those of the $2$-jump mis\`{e}re game when both are played on $\mathcal{T}_{\chi}$ with $\chi$ being Poisson$(\lambda)$.
\end{enumerate*}

\subsection{Proof of Theorem~\ref{lem:decay_rate_nl_{2,lambda}}}\label{subsec:misere_normal_comparison_part_1}
We outline the key idea for our argument. First and foremost, we note that the proof is already done for $i=1$ due to Theorem~\ref{main:normal_draw_probab_limit_Poisson}. For any fixed $i \in \mathbb{N}$ with $i \geqslant 2$, to show that $\lambda^{i}\nl_{2,\lambda} \rightarrow 0$ as $\lambda \rightarrow \infty$, it suffices to show, for \emph{any arbitrary} $c > 0$, that $\lambda^{i}\nl_{2,\lambda} < c \Longleftrightarrow \nl_{2,\lambda} < c \lambda^{-i}$ for all $\lambda$ sufficiently large.

For the same value of $i \geqslant 2$ and $c > 0$ considered above, by the second part of \eqref{lambda^{i}c_{k,lambda}_limit_various_i}, we know that $\lambda^{i} c_{2,\lambda} > c \Longleftrightarrow c_{2,\lambda} > c \lambda^{-i}$ for all $\lambda$ large enough. Recall, from Theorem~\ref{main:Poisson_k=2_normal_phase_transition} and the discussion included towards the beginning of \S\ref{sec:proof_of_main_4}, that $\nl_{2,\lambda} < c_{2,\lambda}$ for all $\lambda$ sufficiently large, and that the curve $y = H_{2,\lambda}(x)$ stays \emph{above} the curve $y=x$ for all $x \in [0,\nl_{2,\lambda})$, then stays \emph{beneath} it for all $x \in (\nl_{2,\lambda},c_{2,\lambda})$, for all $\lambda$ sufficiently large. Therefore, it suffices for us to show that the curve $y = H_{2,\lambda}(x)$ lies beneath the curve $y=x$ at $x = c \lambda^{-i}$, for all $\lambda$ sufficiently large. In other words,
\begin{lemma}\label{lem:decay_rate_nl_{2,lambda}_precursor}
Fix $c > 0$ and $i \in \mathbb{N}$ with $i \geqslant 2$. Then $H_{2,\lambda}(c \lambda^{-i}) < c \lambda^{-i}$ for all $\lambda$ sufficiently large.
\end{lemma}
The proof of this lemma is deferred to \S\ref{appsec:misere_normal_comparison} of the Appendix.

\subsection{Proof of Theorem~\ref{main:misere_normal_comparison}}\label{subsec:misere_normal_comparison_part_2}
We first show that $\ml_{2,\lambda} \leqslant \nl_{2,\lambda}$ for all $\lambda$ sufficiently large. In fact, we prove that the inequality $\ml_{2,\lambda}^{(n)} < \nl_{2,\lambda}^{(n)}$ holds for all $n \in \mathbb{N}$, via induction on $n$ (recall the definitions of these probabilities from \S\ref{subsec:notations}). It is straightforward to see that $\ml_{2,\lambda}^{(1)} = 0$ and $\nl_{2,\lambda}^{(1)} = e^{-\lambda}$, thus verifying the base case for the inductive argument. Suppose $\ml_{2,\lambda}^{(n)} < \nl_{2,\lambda}^{(n)}$ for some $n \in \mathbb{N}$. From \eqref{fixed_point_recursion_eq} and \eqref{H_{k}_defn} (which is applied to $k=2$ and Poisson$(\lambda)$ offspring distribution), we have 
\begin{align}
\nl_{2,\lambda}^{(n+2)} = H_{2,\lambda}(\nl_{2,\lambda}^{(n)}) = G_{\lambda}\left(g_{2,\lambda}\left(1,F_{1,\lambda}(\nl_{2,\lambda}^{(n)}),F_{2,\lambda}(\nl_{2,\lambda}^{(n)})\right)\right).\label{nl^{n+2}}
\end{align}
We also note (to be used later) that as $\nl_{2,\lambda}$ is a fixed point of $H_{2,\lambda}$ (by Theorem~\ref{thm:main_1_normal}), we have 
\begin{equation}\label{nl_lambda}
\nl_{2,\lambda} = G_{\lambda}\left(g_{2,\lambda}\left(1,F_{1,\lambda}(\nl_{2,\lambda}),F_{2,\lambda}(\nl_{2,\lambda})\right)\right) = \exp\left\{\lambda g_{2,\lambda}\left(1,F_{1,\lambda}(\nl_{2,\lambda}),F_{2,\lambda}(\nl_{2,\lambda})\right) - \lambda\right\}.
\end{equation}
A relation analogous to \eqref{fixed_point_recursion_eq} derived in \S\ref{sec:thm_1_misere_proof}, \eqref{mathcal{J}_{k}_defn} (again, for $k=2$ and Poisson$(\lambda)$ offspring), the fact (mentioned in \S\ref{sec:thm_1_misere_proof} and argued the same way as in \S\ref{subsec:thm:main_1_normal_proof_part_2}) that $J_{2,\lambda}$ is increasing on $[0,c_{1,\lambda}]$, the fact (deduced from \eqref{g_{i}_recursive_defn} and \eqref{gamma_{i}_recursive_defn}) that $\gamma_{2}(1,F_{1,\lambda}(x),F_{2,\lambda}(x)) = e^{\lambda e^{-\lambda}}g_{2,\lambda}(1,F_{1,\lambda}(x),F_{2,\lambda}(x))$, and the induction hypothesis together yield 
\begin{align}
\ml_{2,\lambda}^{(n+2)} = J_{2,\lambda}(\ml_{2,\lambda}^{(n)}) \leqslant J_{2,\lambda}(\nl_{2,\lambda}^{(n)}) = G_{\lambda}\left(e^{\lambda e^{-\lambda}}g_{2,\lambda}\left(1,F_{1,\lambda}(\nl_{2,\lambda}^{(n)}),F_{2,\lambda}(\nl_{2,\lambda}^{(n)})\right) + e^{-\lambda}\right) - e^{-\lambda}.\label{ml^{n+2}}
\end{align}
Therefore, to complete the inductive step of showing that $\ml_{2,\lambda}^{(n+2)} < \nl_{2,\lambda}^{(n+2)}$, it suffices to show that the expression in \eqref{ml^{n+2}} is strictly less than that in \eqref{nl^{n+2}} for all $\lambda$ large enough.

Note that $G'_{\lambda}(x) = \lambda G_{\lambda}(x)$ for all $x \in [0,1]$ when we consider Poisson$(\lambda)$ to be our offspring distribution. Since $G''_{\lambda}(x) = \sum_{i=2}^{\infty}i (i-1) \chi(i) x^{i} \geqslant 0$ for $x \in [0,1]$, hence $G_{\lambda}$ is convex. Theorem 21.2 of \cite{simon_blume} states that a continuously differentiable function $f$ defined on an interval $I$ of $\mathbb{R}$ is convex if and only if $f(y) - f(x) \geqslant f'(x)(y-x)$ for all $x, y \in I$. Applying this result with $f$ replaced by $G_{\lambda}$, we obtain, from \eqref{nl^{n+2}} and \eqref{ml^{n+2}}: 
\begin{align}\label{intermediate_7}
\nl_{2,\lambda}^{(n+2)} - \ml_{2,\lambda}^{(n+2)} &\geqslant -\lambda G_{\lambda}\left(e^{\lambda e^{-\lambda}}g_{2,\lambda}\left(1,F_{1,\lambda}(\nl_{2,\lambda}^{(n)}),F_{2,\lambda}(\nl_{2,\lambda}^{(n)})\right) + e^{-\lambda}\right)\nonumber\\&\left[g_{2,\lambda}\left(1,F_{1,\lambda}(\nl_{2,\lambda}^{(n)}),F_{2,\lambda}(\nl_{2,\lambda}^{(n)})\right)\left(e^{\lambda e^{-\lambda}} - 1\right) + e^{-\lambda}\right] + e^{-\lambda}.
\end{align}
In what follows, we show that the expression on the right side of \eqref{intermediate_7} is non-negative for all $\lambda$ large enough.

\subsubsection{Bound on the first factor of the first term of \eqref{intermediate_7}}\label{subsubsec:misere_normal_comparison_part_2_subpart_1}
Since $J_{2,\lambda}$ is increasing on $[0,c_{1,\lambda}]$ and $\nl_{2,\lambda}^{(n)} \leqslant \nl_{2,\lambda}$ (evident from \eqref{increasing_sequence_nl_nw} and \eqref{nl_nw_limit}), given any $\epsilon > 0$, we have
\begin{align}\label{intermediate_8}
&G_{\lambda}\left(e^{\lambda e^{-\lambda}}g_{2,\lambda}\left(1,F_{1,\lambda}(\nl_{2,\lambda}^{(n)}),F_{2,\lambda}(\nl_{2,\lambda}^{(n)})\right) + e^{-\lambda}\right) \leqslant G_{\lambda}\left(e^{\lambda e^{-\lambda}}g_{2,\lambda}\left(1,F_{1,\lambda}(\nl_{2,\lambda}),F_{2,\lambda}(\nl_{2,\lambda})\right) + e^{-\lambda}\right)\nonumber\\
&= \exp\left\{e^{\lambda e^{-\lambda}} \lambda \left[g_{2,\lambda}\left(1,F_{1,\lambda}(\nl_{2,\lambda}),F_{2,\lambda}(\nl_{2,\lambda})\right) - 1\right] + \lambda e^{-\lambda} + \lambda e^{\lambda e^{-\lambda}} - \lambda\right\}\nonumber\\
&= (\nl_{2,\lambda})^{e^{\lambda e^{-\lambda}}} \exp\{\lambda e^{\lambda e^{-\lambda}} + \lambda e^{-\lambda} - \lambda\}, \text{ using \eqref{nl_lambda}}; \nonumber\\
&\leqslant \nl_{2,\lambda} \exp\left\{\lambda \sum_{i=1}^{\infty}\frac{(\lambda e^{-\lambda})^{i}}{i!} + \lambda e^{-\lambda}\right\}, \quad \text{since } e^{\lambda e^{-\lambda}} > 1 \text{ and } \nl_{2,\lambda} < 1 \text{ implies } (\nl_{2,\lambda})^{e^{\lambda e^{-\lambda}}} \leqslant \nl_{2,\lambda};\nonumber\\
&= \nl_{2,\lambda} \exp\left\{\lambda^{2}e^{-\lambda} \sum_{j=0}^{\infty}\frac{(\lambda e^{-\lambda})^{j}}{(j+1)!} + \lambda e^{-\lambda}\right\} \leqslant \nl_{2,\lambda} \exp\left\{\lambda^{2}e^{-\lambda} \sum_{j=0}^{\infty}\frac{(\lambda e^{-\lambda})^{j}}{j!} + \lambda e^{-\lambda}\right\}\nonumber\\
&\leqslant \nl_{2,\lambda} \exp\{\lambda^{2} e^{-\lambda} e^{\lambda e^{-\lambda}} + \lambda e^{-\lambda}\} \leqslant \nl_{2,\lambda} \exp\{\lambda^{2} e^{-\lambda} e^{e^{-1}} + \lambda e^{-\lambda}\} < (1+\epsilon)\nl_{2,\lambda}
\end{align}
for all $\lambda$ sufficiently large, where, in the last line, we utilize the fact that the maximum value of $\lambda e^{-\lambda}$ is $e^{-1}$, and that both $\lambda^{2} e^{-\lambda}$ and $\lambda e^{-\lambda}$ converge to $0$ as $\lambda \rightarrow \infty$. It is crucial to note that how large $\lambda$ needs to be for \eqref{intermediate_8} to hold depends on $\epsilon$ alone, and \emph{not} on $n$.

Note that, as a step in the derivation of \eqref{intermediate_8}, we obtain the bound $e^{\lambda e^{-\lambda}} - 1 \leqslant \lambda e^{-\lambda} e^{\lambda e^{-\lambda}} \leqslant \lambda e^{-\lambda} e^{e^{-1}}$. This will prove useful in obtaining a bound for the entire expression on the right side of \eqref{intermediate_7}.

\subsubsection{Bound on the entire right side of \eqref{intermediate_7}}\label{subsubsec:misere_normal_comparison_part_2_subpart_2}
Applying the bound from \eqref{intermediate_8}, the bound mentioned at the very end of \S\ref{subsubsec:misere_normal_comparison_part_2_subpart_1}, and the rather crude bound $g_{2,\lambda}(1,F_{1,\lambda}(\nl_{2,\lambda}^{(n)}),F_{2,\lambda}(\nl_{2,\lambda}^{(n)})) \leqslant 1$ to \eqref{intermediate_7}, we obtain
\begin{align}
\nl_{2,\lambda}^{(n+2)} - \ml_{2,\lambda}^{(n+2)} &\geqslant -\lambda (1+\epsilon) \nl_{2,\lambda}[\lambda e^{-\lambda} e^{e^{-1}} + e^{-\lambda}] + e^{-\lambda} = e^{-\lambda}[1 - (1+\epsilon)\{e^{e^{-1}}\lambda^{2}\nl_{2,\lambda} + \lambda \nl_{2,\lambda}\}]\label{claim_7}
\end{align}
for all $\lambda$ sufficiently large. Theorem~\ref{lem:decay_rate_nl_{2,lambda}} guarantees that both $\lambda^{2}\nl_{2,\lambda}$ and $\lambda \nl_{2,\lambda}$ converge to $0$ as $\lambda \rightarrow \infty$, so that $(1+\epsilon)\{e^{e^{-1}}\lambda^{2}\nl_{2,\lambda} + \lambda \nl_{2,\lambda}\}$ converges to $0$ as well. Consequently, the final expression of \eqref{claim_7} is strictly positive for all $\lambda$ sufficiently large. Note, once again, that how large $\lambda$ needs to be for this to happen depends on $\epsilon$ alone, and \emph{not} on $n$. This completes the inductive argument, showing that for all $\lambda$ sufficiently large, we have $\ml_{2,\lambda}^{(n)} < \nl_{2,\lambda}^{(n)}$ for every $n \in \mathbb{N}$. Taking the limit as $n \rightarrow \infty$, using \eqref{nl_nw_limit} and an analogous relation for the mis\`{e}re games, we deduce that $\ml_{2,\lambda} \leqslant \nl_{2,\lambda}$, thus completing the proof of the first part of Theorem~\ref{main:misere_normal_comparison}.

\subsubsection{Proving that $\nd_{2,\lambda} < \md_{2,\lambda}$ for all $\lambda$ large enough}\label{subsubsec:misere_normal_comparison_part_2_subpart_3} Proving $\nd_{2,\lambda} < \md_{2,\lambda}$ for all $\lambda$ sufficiently large establishes the second claim of Theorem~\ref{main:misere_normal_comparison}. From \S\ref{sec:proof_of_main_3}, we know that $\nd_{2,\lambda} = F_{2,\lambda}(\nl_{2,\lambda}) - \nl_{2,\lambda}$, whereas $\md_{2,\lambda} = F_{2,\lambda}(\ml_{2,\lambda}) - \ml_{2,\lambda} - e^{-\lambda}$. The rest of \S\ref{subsubsec:misere_normal_comparison_part_2_subpart_3} is dedicated to comparing these two expressions.

We set out to find a suitable upper bound on $\nd_{2,\lambda}$ for all $\lambda$ sufficiently large. We start by obtaining suitable approximations to a couple of infinite series. From Theorem~\ref{lem:decay_rate_nl_{2,lambda}}, we know that $\lambda \nl_{2,\lambda} \rightarrow 0$ as $\lambda \rightarrow \infty$, implying that for all $\lambda$ sufficiently large,
\begin{align}
\lambda \nl_{2,\lambda} < i+1 \Longleftrightarrow \frac{(\lambda \nl_{2,\lambda})^{i+1}}{(i+1)!} < \frac{(\lambda \nl_{2,\lambda})^{i}}{i!} \text{ for all } i \in \mathbb{N}_{0} \text{ and } \lim_{\rightarrow \infty}\frac{(\lambda \nl_{2,\lambda})^{i}}{i!} = 0,\nonumber
\end{align}
so that by the well-known alternating series estimation theorem (see, for instance, \S 8.4 of \cite{stewart2012essential}), we have
\begin{equation}
e^{-\lambda \nl_{2,\lambda}} - 1 \leqslant -\lambda \nl_{2,\lambda} + \frac{\lambda^{2}\nl_{2,\lambda}^{2}}{2}.\label{alternating_approx_1}
\end{equation}
Applying Theorem~\ref{lem:decay_rate_nl_{2,lambda}} one more time, we know that each of $\lambda^{2}\nl_{2,\lambda}$, $\lambda^{3}\nl_{2,\lambda}$ and $\lambda \nl_{2,\lambda}$ approaches $0$ in the limit as $\lambda \rightarrow \infty$. Consequently, $\left|-\lambda^{2}\nl_{2,\lambda} + \lambda^{3}\nl_{2,\lambda}^{2}/2 - \lambda \nl_{2,\lambda}\right| \rightarrow 0$ as $\lambda \rightarrow \infty$. Moreover, 
\begin{equation}
\lim_{\lambda \rightarrow \infty} \frac{\lambda^{3}\nl_{2,\lambda}^{2}}{\lambda^{2}\nl_{2,\lambda}} = \lim_{\lambda \rightarrow \infty} \lambda \nl_{2,\lambda} = 0, \text{ again by Theorem~\ref{lem:decay_rate_nl_{2,lambda}},}\label{claim_8}
\end{equation}
thus showing that $\lambda^{3}\nl_{2,\lambda}^{2} = o(\lambda^{2}\nl_{2,\lambda})$ as $\lambda \rightarrow \infty$. Therefore, we have $-\lambda^{2}\nl_{2,\lambda} + \lambda^{3}\nl_{2,\lambda}^{2}/2 - \lambda \nl_{2,\lambda} < 0$ for all $\lambda$ sufficiently large. Applying the alternating series estimation theorem and arguing the same way as above, we deduce that
\begin{align}
\exp\left\{-\lambda^{2}\nl_{2,\lambda} + \frac{\lambda^{3}\nl_{2,\lambda}^{2}}{2} - \lambda \nl_{2,\lambda}\right\} &\leqslant 1 - \lambda^{2}\nl_{2,\lambda} + \frac{\lambda^{3}\nl_{2,\lambda}^{2}}{2} - \lambda \nl_{2,\lambda} \nonumber\\&+ \frac{1}{2}\left(-\lambda^{2}\nl_{2,\lambda} + \frac{\lambda^{3}\nl_{2,\lambda}^{2}}{2} - \lambda \nl_{2,\lambda}\right)^{2}.\label{alternating_approx_2}
\end{align}
It is evident from \eqref{claim_8} that the leading term of $-\lambda^{2}\nl_{2,\lambda} + \lambda^{3}\nl_{2,\lambda}^{2}/2 - \lambda \nl_{2,\lambda}$ is $-\lambda^{2}\nl_{2,\lambda}$, so that we can write $(-\lambda^{2}\nl_{2,\lambda} + \lambda^{3}\nl_{2,\lambda}^{2}/2 - \lambda \nl_{2,\lambda})^{2} = O(\lambda^{4} \nl_{2,\lambda}^{2})$. This observation, along with Theorem~\ref{lem:decay_rate_nl_{2,lambda}}, reveals that for any fixed $\epsilon > 0$, for all $\lambda$ large enough, we have
\begin{align}
&\lim_{\lambda \rightarrow \infty} \frac{\lambda^{4}\nl_{2,\lambda}^{2}}{\nl_{2,\lambda}} = \lim_{\lambda \rightarrow \infty} \lambda^{4} \nl_{2,\lambda} = 0 \implies \frac{\lambda^{3}\nl_{2,\lambda}^{2}}{2} + \frac{1}{2}\left(-\lambda^{2}\nl_{2,\lambda} + \frac{\lambda^{3}\nl_{2,\lambda}^{2}}{2} - \lambda \nl_{2,\lambda}\right)^{2} \leqslant \epsilon \nl_{2,\lambda}.\label{claim_9}
\end{align}

We have now gathered all the bounds / approximations we need in order to proceed to find a suitable bound for $\nd_{2,\lambda}$: given any $\epsilon > 0$, applying \eqref{alternating_approx_1}, followed by \eqref{alternating_approx_2} and finally \eqref{claim_9}, we have
\begin{align}\label{new_intermediate_1}
\nd_{2,\lambda} &= F_{2,\lambda}(\nl_{2,\lambda}) - \nl_{2,\lambda} = \exp\{\lambda e^{-\lambda \nl_{2,\lambda}} - \lambda \nl_{2,\lambda} - \lambda\} - \nl_{2,\lambda} \nonumber\\
&\leqslant \exp\left\{-\lambda^{2}\nl_{2,\lambda} + \frac{\lambda^{3}\nl_{2,\lambda}^{2}}{2} - \lambda \nl_{2,\lambda}\right\} - \nl_{2,\lambda}\nonumber\\
&\leqslant 1 - \lambda^{2}\nl_{2,\lambda} + \frac{\lambda^{3}\nl_{2,\lambda}^{2}}{2} - \lambda \nl_{2,\lambda} + \frac{1}{2}\left(-\lambda^{2}\nl_{2,\lambda} + \frac{\lambda^{3}\nl_{2,\lambda}^{2}}{2} - \lambda \nl_{2,\lambda}\right)^{2} - \nl_{2,\lambda}\nonumber\\
&\leqslant 1 - \lambda^{2}\nl_{2,\lambda} - \lambda \nl_{2,\lambda} - (1-\epsilon)\nl_{2,\lambda}. 
\end{align}
This completes our deduction of the desired upper bound on $\nd_{2,\lambda}$ for $\lambda$ large enough.

We now come to the derivation of a suitable lower bound on $\md_{2,\lambda}$. We first obtain a rather simple bound, using the inequality $e^{-x} > 1-x$ for all $x > 0$:
\begin{align}\label{new_intermediate_2}
&\md_{2,\lambda} = F_{2,\lambda}(\ml_{2,\lambda}) - \ml_{2,\lambda} - e^{-\lambda} = \exp\left\{\lambda e^{-\lambda \ml_{2,\lambda}} - \lambda \ml_{2,\lambda} - \lambda\right\} - \ml_{2,\lambda} - e^{-\lambda} \nonumber\\
&\geqslant \exp\{-\lambda^{2}\ml_{2,\lambda} - \lambda \ml_{2,\lambda}\} - \ml_{2,\lambda} - e^{-\lambda} \geqslant 1 - \lambda^{2}\ml_{2,\lambda} - \lambda \ml_{2,\lambda} - \ml_{2,\lambda} - e^{-\lambda}.
\end{align}
Note that, since we have already proved while concluding \S\ref{subsubsec:misere_normal_comparison_part_2_subpart_2} that $\ml_{2,\lambda} \leqslant \nl_{2,\lambda}$ for all $\lambda$ large enough, we can further bound the expression on the right side of \eqref{new_intermediate_2} to obtain $\md_{2,\lambda} \geqslant 1 - (\lambda^{2}+\lambda+1)\nl_{2,\lambda} - e^{-\lambda}$. However, this simple bound is not quite enough to compare with the bound obtained in \eqref{new_intermediate_1} and arrive at our desired conclusion. This is the difficulty we overcome in what follows.

We know that $\nl_{2,\lambda}$ is a fixed point of $H_{2,\lambda}$ by Theorem~\ref{thm:main_1_normal}, that $\ml_{2,\lambda}$ is a fixed point of $J_{2,\lambda}$ by Theorem~\ref{thm:main_1_misere}, and that $J_{2,\lambda}$ is increasing on $[0,c_{1,\lambda}]$ (mentioned in \S\ref{sec:thm_1_misere_proof} and argued the same way as in \S\ref{subsec:thm:main_1_normal_proof_part_2}). Using the expressions for $J_{2,\lambda}$ and $H_{2,\lambda}$ that we obtain by considering $k=2$ and Poisson$(\lambda)$ offspring in \eqref{mathcal{J}_{k}_defn} and \eqref{H_{k}_defn}, we have, given any $\epsilon > 0$,
\begin{align}\label{new_intermediate_3}
& (\lambda^{2}+\lambda+1)\ml_{2,\lambda} + e^{-\lambda} = (\lambda^{2} + \lambda + 1)J_{2,\lambda}(\ml_{2,\lambda}) + e^{-\lambda} \leqslant (\lambda^{2} + \lambda + 1)J_{2,\lambda}(\nl_{2,\lambda}) + e^{-\lambda}\nonumber\\
&= (\lambda^{2} + \lambda + 1) G_{\lambda}\left(e^{\lambda e^{-\lambda}}g_{2,\lambda}(1,F_{1,\lambda}(\nl_{2,\lambda}),F_{2,\lambda}(\nl_{2,\lambda})) + e^{-\lambda}\right) - (\lambda^{2}+\lambda)e^{-\lambda} \nonumber\\
&= (\lambda^{2} + \lambda + 1)\exp\left\{e^{\lambda e^{-\lambda}} \lambda \left[g_{2,\lambda}(1,F_{1,\lambda}(\nl_{2,\lambda}),F_{2,\lambda}(\nl_{2,\lambda})) - 1\right] + \lambda e^{\lambda e^{-\lambda}} + \lambda e^{-\lambda} - \lambda\right\} - (\lambda^{2}+\lambda)e^{-\lambda}\nonumber\\
&< (\lambda^{2} + \lambda + 1)\exp\left\{e^{\lambda e^{-\lambda}} \lambda \left[g_{2,\lambda}(1,F_{1,\lambda}(\nl_{2,\lambda}),F_{2,\lambda}(\nl_{2,\lambda})) - 1\right] + \lambda e^{-\lambda}\right\} - (\lambda^{2}+\lambda)e^{-\lambda} \left(\text{as } e^{\lambda e^{-\lambda}} > 1\right);\nonumber\\
&= (\lambda^{2} + \lambda + 1)(\nl_{2,\lambda})^{e^{\lambda e^{-\lambda}}} e^{\lambda e^{-\lambda}} - (\lambda^{2}+\lambda)e^{-\lambda} < (\lambda^{2} + \lambda + 1)(1+\epsilon)\nl_{2,\lambda}-(\lambda^{2}+\lambda)e^{-\lambda}.
\end{align}
for all $\lambda$ large enough. From \eqref{new_intermediate_2} and \eqref{new_intermediate_3}, we obtain
\begin{equation}\label{new_intermediate_2_3_combined}
\md_{2,\lambda} > 1 - (\lambda^{2} + \lambda + 1)(1+\epsilon)\nl_{2,\lambda} + (\lambda^{2}+\lambda)e^{-\lambda},
\end{equation}
so that from \eqref{new_intermediate_1} and \eqref{new_intermediate_2_3_combined}, we have
\begin{align}
\md_{2,\lambda} - \nd_{2,\lambda} > -(\lambda^{2}+\lambda+2)\epsilon \nl_{2,\lambda} + (\lambda^{2}+\lambda)e^{-\lambda},\label{new_intermediate_1_2_3_combined}
\end{align}
and our objective is to show that the expression on the right side of \eqref{new_intermediate_1_2_3_combined} is non-negative for all $\lambda$ sufficiently large. Given any $\epsilon > 0$, we have $\lambda^{2} + \lambda + 2 < (1+\epsilon)(\lambda^{2}+\lambda)$ for all $\lambda$ large enough, and from Theorem~\ref{main:compare_1_vs_2} (which is yet to be proved), we have $\nl_{2,\lambda} \leqslant \nl_{1,\lambda}$ for all $\lambda$ large enough. To establish \eqref{new_intermediate_1_2_3_combined}, it thus suffices to prove that, given any $\epsilon > 0$ such that $\epsilon (1+\epsilon) < 1$, 
\begin{equation}\label{new_intermediate_1_2_3_combined_modified}
\frac{e^{-\lambda}}{\epsilon (1+\epsilon)} > \nl_{1,\lambda} \text{ for all } \lambda \text{ sufficiently large}.
\end{equation}

To establish \eqref{new_intermediate_1_2_3_combined_modified}, we need to deduce certain properties of $\nl_{1,\lambda}$, which we do by examining the function $H_{1,\lambda}$ (evident from Theorem~\ref{thm:main_1_normal}). Setting $k=1$ and $\chi$ to be Poisson$(\lambda)$ in \eqref{H_{k}_defn}, we obtain 
\begin{equation}
H_{1,\lambda}(x) = e^{-\lambda e^{-\lambda x}} \implies \lambda^{3}(\lambda e^{-\lambda x} - 1) e^{-\lambda x - \lambda e^{-\lambda x}}.\label{H''_{1,lambda}}
\end{equation}
Note that, using \eqref{F_{i,lambda}_extended}, for all $\lambda > e$, we have
\begin{equation}
F_{1,\lambda}\left(\frac{\ln \lambda}{\lambda}\right) = \exp\left\{-\lambda \cdot \frac{\ln \lambda}{\lambda}\right\} = \frac{1}{\lambda} < \frac{\ln \lambda}{\lambda},\nonumber
\end{equation}
and as Lemma~\ref{lem:main_thm_1_1} states that $F_{1,\lambda}$ is strictly decreasing on $[0,1]$ and its unique fixed point is $c_{1,\lambda}$, we conclude that 
\begin{equation}
c_{1,\lambda} < \frac{\ln \lambda}{\lambda} \text{ for all } \lambda > e \implies \lambda e^{-\lambda x} > 1 \text{ for all } x \in [0, c_{1,\lambda}], \text{ for all } \lambda > e.\nonumber
\end{equation}
This finding, when applied to \eqref{H''_{1,lambda}}, reveals that $H_{1,\lambda}$ is strictly convex on $[0,c_{1,\lambda}]$ for all $\lambda > e$. Consequently, the curve $y=H_{1,\lambda}(x)$ intersects the line $y=x$ at at most two points inside the interval $[0,c_{1,\lambda}]$ -- if there are two intersections, then these happen at $\nl_{1,\lambda}$ and $c_{1,\lambda}$, and if there is only one intersection, then this happens at $c_{1,\lambda}$ (in which case $\nl_{1,\lambda} = c_{1,\lambda}$). This conclusion, along with Theorem~\ref{main:normal_draw_probab_limit_Poisson} whose proof guarantees that $\nl_{1,\lambda} < c_{1,\lambda}$ for all $\lambda$ large enough, reveals that the curve $y=H_{1,\lambda}(x)$ stays \emph{above} the line $y=x$ when $x \in [0,\nl_{1,\lambda})$, and \emph{beneath} it when $x \in (\nl_{1,\lambda},c_{1,\lambda})$, for $\lambda$ sufficiently large.

We now note that $\lambda c_{1,\lambda} \rightarrow \infty$ as $\lambda \rightarrow \infty$, due to \eqref{lambda^{i}c_{k,lambda}_limit_various_i}, so that $\frac{e^{-\lambda}}{\epsilon(1+\epsilon)} < c_{1,\lambda}$ for all $\lambda$ sufficiently large. Therefore, to establish \eqref{new_intermediate_1_2_3_combined_modified}, it suffices to show that $y=H_{1,\lambda}(x)$ lies beneath $y=x$ at $x = \frac{e^{-\lambda}}{\epsilon(1+\epsilon)}$. Using $x > 1 - e^{-x}$ for all $x > 0$, we have
\begin{align}
e^{\lambda} H_{1,\lambda}\left(\frac{e^{-\lambda}}{\epsilon(1+\epsilon)}\right) &= e^{\lambda} \exp\left\{-\lambda \exp\left\{-\lambda \cdot \frac{e^{-\lambda}}{\epsilon(1+\epsilon)}\right\}\right\} \nonumber\\
&= \exp\left\{\lambda\left[1 - \exp\left\{-\frac{\lambda e^{-\lambda}}{\epsilon(1+\epsilon)}\right\}\right]\right\} < \exp\left\{\frac{\lambda^{2} e^{-\lambda}}{\epsilon(1+\epsilon)}\right\}.\nonumber
\end{align}
Since the right side of the above inequality goes to $1$ as $\lambda \rightarrow \infty$, and we chose $\epsilon$ above such that $\epsilon (1+\epsilon) < 1$, the right side of the above inequality is strictly less than $\frac{1}{\epsilon (1+\epsilon)}$ for all $\lambda$ sufficiently large, thus proving that $H_{1,\lambda}\left(\frac{e^{-\lambda}}{\epsilon(1+\epsilon)}\right) < \frac{e^{-\lambda}}{\epsilon (1+\epsilon)}$. In other words, we have proved that $y=H_{1,\lambda}(x)$ lies beneath $y=x$ at $x = \frac{e^{-\lambda}}{\epsilon(1+\epsilon)}$, as desired. This concludes \S\ref{subsubsec:misere_normal_comparison_part_2_subpart_3}, and with it, the proof of the fact that $\nd_{2,\lambda} < \md_{2,\lambda}$ for all $\lambda$ large enough. 

\subsubsection{Proving that $\nl_{2,\lambda} < \nw_{2,\lambda}$ for all $\lambda$ large enough}\label{subsubsec:misere_normal_comparison_part_2_subpart_4} Proving $\nl_{2,\lambda} < \nw_{2,\lambda}$ for all $\lambda$ sufficiently large establishes the third and final claim made in the statement of Theorem~\ref{main:misere_normal_comparison}. By Theorem~\ref{thm:main_1_normal}, we have $\nw_{2,\lambda} = 1 - F_{2,\lambda}(\nl_{2,\lambda})$, so that it suffices to show that $F_{2,\lambda}(\nl_{2,\lambda}) < 1 - \nl_{2,\lambda}$. From \eqref{F_{i,lambda}_extended}, we obtain
\begin{align}
F_{2,\lambda}(\nl_{2,\lambda}) &= \exp\{\lambda e^{-\lambda \nl_{2,\lambda}} - \lambda \nl_{2,\lambda} - \lambda\} \nonumber\\
&\leqslant \exp\left\{-\lambda^{2} \nl_{2,\lambda} + \frac{\lambda^{3}\nl_{2,\lambda}^{2}}{2} - \lambda \nl_{2,\lambda}\right\}, \text{ by \eqref{alternating_approx_1}};\nonumber\\
&\leqslant 1 - \lambda^{2} \nl_{2,\lambda} + \frac{\lambda^{3}\nl_{2,\lambda}^{2}}{2} - \lambda \nl_{2,\lambda} + \frac{1}{2}\left(-\lambda^{2} \nl_{2,\lambda} + \frac{\lambda^{3}\nl_{2,\lambda}^{2}}{2} - \lambda \nl_{2,\lambda}\right)^{2}, \text{ by \eqref{alternating_approx_2}};\nonumber\\
&\leqslant 1 - \lambda^{2} \nl_{2,\lambda} - \lambda \nl_{2,\lambda} + \epsilon \nl_{2,\lambda} \text{ by \eqref{claim_9};}\nonumber
\end{align}
and this is less than $1-\nl_{2,\lambda}$ for all $\lambda$ large enough since $\nl_{2,\lambda}=o\left(\lambda^{2} \nl_{2,\lambda}\right)$. This concludes \S\ref{subsubsec:misere_normal_comparison_part_2_subpart_4} and the proof of the entire Theorem~\ref{main:misere_normal_comparison}.

\section{Proof of Theorem~\ref{main:compare_1_vs_2}}\label{sec:compare_1_vs_2}
Recall that the objectives of \S\ref{sec:compare_1_vs_2} is to show that the inequalities $\nl_{2,\lambda} \leqslant \nl_{1,\lambda}$, $\nd_{2,\lambda} < \nd_{1,\lambda}$ and $\nw_{1,\lambda} < \nw_{2,\lambda}$ hold for all $\lambda$ sufficiently large. 

\subsection{Showing that $\nl_{2,\lambda} \leqslant \nl_{1,\lambda}$ for all $\lambda$ sufficiently large}\label{subsec:compare_1_vs_2_part_1} We begin with the first inequality, and we establish this by first showing that $\nl_{2,\lambda}^{(n)} \leqslant \nl_{1,\lambda}^{(n)}$ for all $n \in \mathbb{N}$, which is proved by induction on $n$, then taking the limit as $n \rightarrow \infty$ and using \eqref{nl_nw_limit}.

Since $\nl_{1,\lambda}^{(1)} = \nl_{2,\lambda}^{(1)} = e^{-\lambda}$, the base case for the inductive argument is verified. Suppose $\nl_{2,\lambda}^{(n)} \leqslant \nl_{1,\lambda}^{(n)}$ for some $n \in \mathbb{N}$. From \eqref{fixed_point_recursion_eq}, we have $\nl_{1,\lambda}^{(n+2)} = H_{1,\lambda}(\nl_{1,\lambda}^{(n)})$ and $\nl_{2,\lambda}^{(n+2)} = H_{2,\lambda}(\nl_{2,\lambda}^{(n)})$, so that our task now is to compare these two quantities. To this end, we observe, setting $k=2$ and considering Poisson$(\lambda)$ offspring in \eqref{H_{k}_defn},
\begin{align}
H_{2,\lambda}(x) &= G_{\lambda}\big(G_{\lambda}(1-F_{2,\lambda}(x)) - G_{\lambda}(F_{1,\lambda}(x)-F_{2,\lambda}(x))\big)\nonumber\\
&= G_{\lambda}\left(e^{-\lambda F_{2,\lambda}(x)} - e^{\lambda F_{1,\lambda}(x) - \lambda F_{2,\lambda}(x) - \lambda}\right) = G_{\lambda}\left(e^{-\lambda F_{2,\lambda}(x)}\left[1 - e^{\lambda F_{1,\lambda}(x) - \lambda}\right]\right).\nonumber
\end{align}
Applying this, the induction hypothesis and the increasing nature of $H_{1,\lambda}$ on $[0,1]$ (as proved in \S\ref{subsec:thm:main_1_normal_proof_part_2}), we obtain
\begin{align}\label{intermediate_10}
\nl_{1,\lambda}^{(n+2)} - \nl_{2,\lambda}^{(n+2)} &\geqslant H_{1,\lambda}\left(\nl_{2,\lambda}^{(n)}\right) - H_{2,\lambda}\left(\nl_{2,\lambda}^{(n)}\right) \nonumber\\
&= G_{\lambda}\left(1 - F_{1,\lambda}(\nl_{2,\lambda}^{(n)})\right) - G_{\lambda}\left(e^{-\lambda F_{2,\lambda}(\nl_{2,\lambda}^{(n)})}\left[1 - e^{\lambda F_{1,\lambda}(\nl_{2,\lambda}^{(n)}) - \lambda}\right]\right).
\end{align}
We now construct a non-negative lower bound for the expression in \eqref{intermediate_10}. 

Since $\nl_{2,\lambda}^{(n)} \leqslant \nl_{2,\lambda}$ (evident from \eqref{increasing_sequence_nl_nw} and \eqref{nl_nw_limit}) and since $F_{2,\lambda}$ is strictly decreasing (by Lemma~\ref{lem:main_thm_1_1}), given any $0 < \epsilon < 1$, we have, using $e^{-x}-1 > -x$ for all $x > 0$,
\begin{align}\label{intermediate_11}
F_{2,\lambda}(\nl_{2,\lambda}^{(n)}) \geqslant F_{2,\lambda}(\nl_{2,\lambda}) = \exp\{\lambda e^{-\lambda \nl_{2,\lambda}} - \lambda \nl_{2,\lambda} - \lambda\} \geqslant \exp\{-\lambda^{2}\nl_{2,\lambda} - \lambda \nl_{2,\lambda}\} > 1-\epsilon
\end{align}
for all $\lambda$ sufficiently large, since both $\lambda^{2}\nl_{2,\lambda}$ and $\lambda \nl_{2,\lambda}$ converges to $0$ as $\lambda \rightarrow \infty$ due to Theorem~\ref{lem:decay_rate_nl_{2,lambda}}. Note, crucially, that how large $\lambda$ needs to be for \eqref{intermediate_11} to hold depends on $\epsilon$ alone, and not on $n$. Using \eqref{intermediate_11} and $1-e^{-x} < x$ for all $x > 0$, the expression on the right side of \eqref{intermediate_10} is bounded below by
\begin{align}
& G_{\lambda}\left(1 - F_{1,\lambda}(\nl_{2,\lambda}^{(n)})\right) - G_{\lambda}\left(e^{-\lambda(1-\epsilon)}\left[1 - e^{\lambda F_{1,\lambda}(\nl_{2,\lambda}^{(n)}) - \lambda}\right]\right)\nonumber\\
& \geqslant G_{\lambda}\left(1 - F_{1,\lambda}(\nl_{2,\lambda}^{(n)})\right) - G_{\lambda}\left(e^{-\lambda(1-\epsilon)}\lambda\{1 - F_{1,\lambda}(\nl_{2,\lambda}^{(n)})\}\right),\nonumber
\end{align}
and this is strictly positive for all $\lambda$ sufficiently large since $\lambda e^{-\lambda(1-\epsilon)} \rightarrow 0$ as $\lambda \rightarrow \infty$ (and hence eventually becomes strictly less than $1$) and $G_{\lambda}$ is strictly increasing. This completes the inductive proof and brings us to the end of \S\ref{subsec:compare_1_vs_2_part_1}.

\subsection{Showing that $\nd_{2,\lambda} < \nd_{1,\lambda}$ and $\nw_{1,\lambda} < \nw_{2,\lambda}$ for all $\lambda$ sufficiently large}\label{subsec:compare_1_vs_2_part_2} From Theorem~\ref{thm:main_1_normal}, we have the following lower bound on $\nd_{1,\lambda}$:
\begin{align}\label{intermediate_13}
\nd_{1,\lambda} = F_{1,\lambda}(\nl_{1,\lambda}) - \nl_{1,\lambda} = e^{-\lambda \nl_{1,\lambda}} - \nl_{1,\lambda} \geqslant 1 - \lambda \nl_{1,\lambda} - \nl_{1,\lambda},
\end{align}
and comparing this with the lower bound on $\nd_{2,\lambda}$ obtained in \eqref{new_intermediate_1}, we see that the objective of \S\ref{subsec:compare_1_vs_2_part_2} will be accomplished if we can show that $\nl_{1,\lambda} < \lambda \nl_{2,\lambda}$ for all $\lambda$ sufficiently large.

Recall that we estabilshed in \S\ref{subsubsec:misere_normal_comparison_part_2_subpart_3} that $H_{1,\lambda}$ is strictly convex on $[0,c_{1,\lambda}]$ for $\lambda > e$, so that $y=H_{1,\lambda}(x)$ lies \emph{above} $y=x$ on $[0,\nl_{1,\lambda})$, and \emph{beneath} $y=x$ on $(\nl_{1,\lambda},c_{1,\lambda})$, for all $\lambda$ sufficiently large. From \eqref{lambda^{i}c_{k,lambda}_limit_various_i}, we have $\lambda c_{1,\lambda} \rightarrow \infty$, whereas Theorem~\ref{lem:decay_rate_nl_{2,lambda}} yields $\lambda^{2}\nl_{2,\lambda} \rightarrow 0$, so that $\lambda \nl_{2,\lambda} < c_{1,\lambda}$ for all $\lambda$ sufficiently large. The goal set in the previous paragraph will, therefore, be accomplished, if we can show that $y=H_{1,\lambda}(x)$ lies beneath $y=x$ at $x = \lambda \nl_{2,\lambda}$.

Using $e^{-x} > 1 - x$ for $x > 0$, we have 
\begin{equation}
H_{1,\lambda}(\lambda \nl_{2,\lambda}) = \exp\left\{-\lambda e^{-\lambda^{2}\nl_{2,\lambda}}\right\} < \exp\left\{-\lambda + \lambda^{3}\nl_{2,\lambda}\right\} < \lambda e^{-\lambda} = \lambda \nl_{2,\lambda}^{(1)} \leqslant \lambda \nl_{2,\lambda}\nonumber
\end{equation}
for all $\lambda$ sufficiently large, since we have $\lambda^{3}\nl_{2,\lambda} \rightarrow 0$ by Theorem~\ref{lem:decay_rate_nl_{2,lambda}}. This completes the proof of $\nd_{1,\lambda} > \nd_{2,\lambda}$ for all $\lambda$ sufficiently large. 

Since $\nw_{1,\lambda} = 1 - \nl_{1,\lambda} - \nd_{1,\lambda}$ and $\nw_{2,\lambda} = 1 - \nl_{2,\lambda} - \nd_{2,\lambda}$, it follows immediately from the previous conclusion and the conclusion drawn in \S\ref{subsec:compare_1_vs_2_part_1} that $\nw_{1,\lambda} < \nw_{2,\lambda}$ for all $\lambda$ sufficiently large.

\section{Proof of Theorem~\ref{main:avg_dur}}\label{sec:avg_dur}
Recall that we intend to show that when $\nl_{k} = c_{k}$ and $\max\left\{H'_{k}(c_{k}), \left|F'_{k}(c_{k})\right|\right\} < 1$, the expected duration of the $k$-jump normal game is finite. Here, unlike the previous few sections, we consider \emph{any} offspring distribution $\chi$ that satisfies the restrictions discussed in \S\ref{subsec:notations}. 

Suppose $\nl_{k} = c_{k}$, $H'_{k}(c_{k}) \leqslant \gamma$ and $|F'_{k}(c_{k})| \leqslant \gamma$ for some $\gamma < 1$. Given $0 < \epsilon < 1-\gamma$, from \eqref{nl_nw_limit} and the continuity of $H'_{k}$ and $F'_{k}$ (due to the continuity of $G'$), we know there exists $N \in \mathbb{N}$ such that $H'_{k}(x) \leqslant \gamma+\epsilon$ and $|F'_{k}(x)| \leqslant \gamma+\epsilon$ for all $\nl_{k}^{(N)} \leqslant x \leqslant \nl_{k} = c_{k}$. Using \eqref{fixed_point_recursion_eq} and \eqref{increasing_sequence_nl_nw}, for all $n \geqslant N$, we have
\begin{align}\label{intermediate_14}
\nl_{k} - \nl_{k}^{(n+2)} = H_{k}(\nl_{k}) - H_{k}(\nl_{k}^{(n)}) = H'_{k}(\xi_{n}) (\nl_{k} - \nl_{k}^{(n)}) \leqslant (\epsilon + \gamma)(\nl_{k} - \nl_{k}^{(n)}),
\end{align}
where $\nl_{k}^{(n)} < \xi_{n} < \nl_{k}$ (by the mean value theorem). Likewise, from \eqref{nw^{(n+1)}_closed_form_in_terms_of_nl^{(n)}} and \eqref{increasing_sequence_nl_nw}, for all $n \geqslant N$, we have
\begin{align}\label{intermediate_15}
\nw_{k} - \nw_{k}^{(n+2)} \leqslant \nw_{k} - \nw_{k}^{(n+1)} = F_{k}(\nl_{k}^{(n)}) - F_{k}(\nl_{k}) \leqslant (\epsilon + \gamma)(\nl_{k} - \nl_{k}^{(n)}).
\end{align}
Denoting the (random) duration of the game (starting at the root $\phi$ of $\mathcal{T}_{\chi}$) by $T$, recalling the definition of $\ND_{k}^{(n)}$ from \S\ref{subsec:notations}, and letting $C = \sum_{n=1}^{N+1}(1 - \nw_{k}^{(n)} - \nl_{k}^{(n)})$, we have
\begin{align}
\E[T] &= \sum_{n=1}^{\infty}\Prob[T \geqslant n] = \sum_{n=1}^{\infty}\Prob[\phi \in \ND_{k}^{(n)}] = \sum_{n=1}^{\infty}\nd_{k}^{(n)} \nonumber\\
&= C + \sum_{n=N+2}^{\infty}(1 - \nw_{k}^{(n)} - \nl_{k}^{(n)}) = C + \sum_{n=N+2}^{\infty}\{(\nw_{k} - \nw_{k}^{(n)}) + (\nl_{k} - \nl_{k}^{(n)})\},\nonumber
\end{align}
since $1 - \nw_{k} - \nl_{k} = \nd_{k} = 0$ (by Theorem~\ref{main:bounds_on_nl_{k}_ml_{k}}, as $\nl_{k} = c_{k}$). It converges as \eqref{intermediate_14} and \eqref{intermediate_15} guarantee exponential decay of the summands as $n \rightarrow \infty$. 

To establish the second part of Theorem~\ref{main:avg_dur}, we show that $|F'_{k}(c_{k})| < 1 \implies H'_{k}(c_{k}) < 1$ when $k = 2, 3$ (which then guarantees that the first part of Theorem~\ref{main:avg_dur} holds). Letting 
\begin{equation}
a_{0} = G'(1-c_{2}) = |F'_{1}(c_{2})| \text{ and } a_{1} = G'(F_{1}(c_{2})-c_{2}),\nonumber
\end{equation}
we have (also making use of using Lemma~\ref{g_{j}_c_{k}_patterns})
\begin{equation}
F'_{2}(c_{2}) = -a_{1}(a_{0}+1) \text{ and } H'_{2}(c_{2}) = a_{1}^{2}(a_{0}^{2} + 2a_{0} - a_{1}a_{0} - a_{1}) < \{a_{1}(a_{0}+1)\}^{2},\nonumber
\end{equation}
thus proving our claim for $k=2$. 

The proof is similar for $k=3$, albeit requiring more involved computations, as follows. Letting 
\begin{equation}
b_{0} = \left|F'_{1}(c_{3})\right| = G'(1-c_{3}), \ b_{1} = G'(F_{1}(c_{3})-c_{3}) \text{ and } b_{2} = G'(F_{2}(c_{3})-c_{3}),\nonumber
\end{equation}
we have $F'_{3}(c_{3}) = -b_{2}(b_{1}b_{0} + b_{1} + 1)$. This, along with Lemma~\ref{g_{j}_c_{k}_patterns}, yields
\begin{align}
& \frac{d}{dx}G(g_{2}(1,F_{2}(x),F_{3}(x)))\big|_{x=c_{3}} = b_{2}[b_{1}^{2}b_{0}^{2} + 2b_{0}b_{1}^{2} + b_{0}b_{1} + b_{1}^{2} - b_{2}b_{1}^{2}b_{0} - b_{2}b_{1}^{2} - b_{2}b_{1}],\nonumber\\
& \frac{d}{dx}G(g_{2}(F_{1}(x),F_{2}(x),F_{3}(x)))\big|_{x=c_{3}} = b_{2}[-b_{1}b_{0} + b_{2}b_{1}^{2}b_{0} + b_{2}b_{1}^{2} + 2b_{2}b_{1} + b_{2}b_{1}b_{0} - b_{2}^{2}b_{1}b_{0} - b_{2}^{2}b_{1} - b_{2}^{2}].\nonumber
\end{align}
These together give us 
\begin{align}
H'_{3}(c_{3}) &= b_{2}^{2}[b_{1}^{2}b_{0}^{2} + 2b_{0}b_{1}^{2} + 2b_{0}b_{1} + b_{1}^{2} + 2b_{1} + 1 + \{- 2b_{2}b_{1}^{2}b_{0} - 2b_{2}b_{1}^{2} - 3b_{2}b_{1} - b_{2}b_{1}b_{0} + b_{2}^{2}b_{1}b_{0} + b_{2}^{2}b_{1} \nonumber\\&+ b_{2}^{2} - 2b_{1} - 1\}] < \{b_{2}[b_{1}b_{0} + b_{1} + 1]\}^{2}\nonumber
\end{align}
The above inequality is obtained as follows. We use $F_{2}(c_{3}) < F_{1}(c_{3})$ (since we have already shown that $F_{i}(x) \leqslant F_{i-1}(x)$ holds for $x \in [0,c_{i-1}]$, in the proof of Equation~\eqref{belongs_to_D_{i}}) and the increasing nature of $G'$ to deduce $b_{2} < b_{1}$. This, in turn, yields $b_{2}^{2}b_{1} < b_{2}b_{1}^{2}$, $b_{2}^{2} < b_{2}b_{1}$ and $b_{2}^{2}b_{1}b_{0} < b_{2}b_{1}^{2}b_{0}$.

\begin{remark}
We have verified the inequality $H'_{k}(c_{k}) < |F'_{k}(c_{k})|^{2}$ for a few higher values of $k$ as well, and we conjecture that this is true for all $k \in \mathbb{N}$, but have been unable to discern a pattern in the expression for $H'_{k}(c_{k})$ that is not too complicated to work with in order to prove this conjecture.
\end{remark}

\clearpage \section{Appendix}\label{sec:appendix}

\subsection{Proofs of lemmas from \S\ref{sec:thm_1_normal_proof}}\label{appsec:thm_1_normal_proof}
\begin{proof}[Proof of Lemma~\ref{lem_normal_compactness}]
As mentioned previously, the argument in this proof resembles that employed to prove Proposition 7 of \cite{holroyd_martin}. We show that the sets $\widetilde{\NL} = \NL \setminus (\bigcup_{n=1}^{\infty}\NL^{(n)})$ and $\widetilde{\NW} = \NW \setminus (\bigcup_{n=1}^{\infty}\NW^{(n)})$ are empty. From \eqref{normal_main_recursion_1}, for $u$ to be in $\widetilde{\NW}$, there exists some $v \in \Gamma_{k}(u) \cap \NL$. If $v$ were in $\NL^{(n)}$ for some $n \in \mathbb{N}$, then the game starting at $u$ is won by P1 (playing the first round) in less than $n+1$ rounds by moving the token from $u$ to $v$ in the first round. This implies $u \in \NW^{(n+1)}$, contradicting the assumption that $u \in \widetilde{\NW}$. Thus $v$ must be in $\widetilde{\NL}$. Equivalently, we have $u \in \widetilde{\NW}$ iff $\Gamma_{k}(u) \cap \NL = \Gamma_{k}(u) \cap \widetilde{\NL} \neq \emptyset$. Moreover, since the players are assumed to play optimally, P1 moves the token from $u$ to some $v \in \Gamma_{k}(u) \cap \widetilde{\NL}$ in the first round. 

From \eqref{normal_main_recursion_2}, for $v$ to be in $\widetilde{\NL}$, every vertex in $\Gamma_{k}(v)$ must be in $\NW$. If for some $n \in \mathbb{N}$, $w \in \NW^{(n)}$ for every $w \in \Gamma_{k}(v)$, then no matter how P1 moves in the first round of the game starting at $v$, she loses in less than $n+1$ rounds, thus implying $v \in \NL^{(n+1)}$ and contradicting the assumption that $v \in \widetilde{\NL}$. Thus there exists some $w \in \Gamma_{k}(v)$ with $w \in \widetilde{\NW}$. Equivalently, $v \in \widetilde{\NL}$ iff $\Gamma_{k}(v) \subset \NW$ and $\Gamma_{k}(v) \cap \widetilde{\NW} \neq \emptyset$. Moreover, under optimal play, P1 moves the token from $v$ to some $w \in \Gamma_{k}(v) \cap \widetilde{\NW}$ in the first round in her attempt to prolong the game as much as possible.

These observations reveal that if the game begins at $u_{0} \in \widetilde{\NL}$, the token gets moved to some $u_{1} \in \Gamma_{k}(u_{0}) \cap \widetilde{\NW}$ in the first round, to some $u_{2} \in \Gamma_{k}(u_{1}) \cap \widetilde{\NL}$ in the second round, and so on. It continues till eternity, with each player \emph{always} being able to make a move. The game thus ends in a draw, contradicting the definition of $\widetilde{\NL}$. Therefore, $\widetilde{\NL} = \emptyset$, and likewise, $\widetilde{\NW} = \emptyset$ as well.
\end{proof}

\begin{proof}[Proof of Lemma~\ref{lem:main_thm_1_1}]
We show, for all $i \in \mathbb{N}$, 
\begin{enumerate}
\item \label{F_{i}_claim_1} that $F_{i}$ is a strictly decreasing function on $[0,c_{i-1}]$, where $F_{i}$ is as defined in \eqref{F_{i}_defn} and $c_{i-1}$ is the unique fixed point of $F_{i-1}$, 
\item \label{F_{i}_claim_extra} that $F_{i}(0) = 1$,
\item \label{F_{i}_claim_2} that $c_{i}$ exists and is uniquely defined,
\item \label{F_{i}_claim_3} and that $\chi(0) < c_{i} < c_{i-1}$.
\end{enumerate}
We prove these claims together, via induction on $i$. We note, at the very outset, that $G'(x) = \sum_{i=1}^{\infty}i \chi(i) x^{i} > 0$ for all $x  > 0$, guaranteeing $G$ is strictly increasing on $[0,1]$ -- these facts are repeatedly utilized below.

Since $F'_{1}(x) = -G'(1-x) < 0$ for all $x \in [0,1)$ (as $\chi(0) < 1$), $F_{1}$ is strictly decreasing on $[0,1]$, proving \eqref{F_{i}_claim_1} for $i=1$. This also implies that $F_{1}(x) - x$ is strictly decreasing on $[0,1]$. Moreover, $F_{1}(0) = G(1) = 1 > 0$ (this verifies \eqref{F_{i}_claim_extra}), so that the curve $y=F_{1}(x)-x$ lies above the $x$-axis at $x=0$, and $F_{1}(1) - 1 = \chi(0) - 1 < 0$, so that the curve $y=F_{1}(x)-x$ lies below the $x$-axis at $x=1$. This fact, along with \eqref{F_{i}_claim_1} for $i=1$, guarantees that $F_{1}(x) - x$ has a unique root in $(0,1)$, which we denote by $c_{1}$. This proves \eqref{F_{i}_claim_2} for $i=1$. Finally, $F_{1}(\chi(0)) = G(1-\chi(0)) > G(0) = \chi(0)$ (as $\chi(0) < 1$). This, along with \eqref{F_{i}_claim_1} and \eqref{F_{i}_claim_2} for $i=1$, allows us to conclude that $c_{1} > \chi(0)$, thus proving \eqref{F_{i}_claim_3} for $i=1$. This completes verifying the base case for the induction.

Suppose we have shown that all of \eqref{F_{i}_claim_1}, \eqref{F_{i}_claim_extra}, \eqref{F_{i}_claim_2} and \eqref{F_{i}_claim_3} hold for some $i \in \mathbb{N}$. From \eqref{F_{i}_defn}, we have $F'_{i+1}(x) = G'(F_{i}(x)-x)(F'_{i}(x) - 1)$. By the induction hypothesis, $F'_{i}(x) - 1 < F'_{i}(x) < 0$ for all $x \in [0,c_{i}] \subset [0,c_{i-1}]$, whereas $G'(F_{i}(x)-x) > 0$ for $x \in [0,c_{i}]$, thus proving that $F'_{i+1}(x) < 0$ and hence $F_{i+1}$ is strictly decreasing on $[0,c_{i}]$. This proves \eqref{F_{i}_claim_1}. 

By the induction hypothesis, $F_{i+1}(0) = G(F_{i}(0) - 0) = G(1) = 1$, proving \eqref{F_{i}_claim_extra}. Note that this also implies that the curve $y=F_{i+1}(x)-x$ lies above the $x$-axis at $x=0$, whereas $F_{i+1}(c_{i}) - c_{i} = G(F_{i}(c_{i}) - c_{i}) - c_{i} = G(0) - c_{i} = \chi(0) - c_{i} < 0$ (using the induction hypothesis pertaining to \eqref{F_{i}_claim_3}), so that the curve $y=F_{i+1}(x)-x$ lies below the $x$-axis at $x=c_{i}$. This fact, along with \eqref{F_{i}_claim_1} that we have already proved, guarantees that $F_{i+1}(x) - x$ has a unique root, $c_{i+1}$, in $(0,c_{i})$ -- this proves \eqref{F_{i}_claim_2} and the second inequality of \eqref{F_{i}_claim_3}.

Due to \eqref{F_{i}_claim_1} and as $c_{i}$ is the fixed point of $F_{i}$ in $[0,c_{i-1}]$, we know that the curve $y=F_{i}(x)$ lies strictly above the line $y=x$ for $x \in [0,c_{i})$, and strictly below the line $y=x$ for $x \in (c_{i},c_{i-1}]$. Since $\chi(0) \in [0,c_{i})$ (because of the induction hypothesis pertaining to \eqref{F_{i}_claim_3}), we conclude that $F_{i}(\chi(0)) > \chi(0)$. This, along with the strictly increasing nature of $G$ on $[0,1]$, yields 
\begin{equation}
F_{i+1}(\chi(0)) = G(F_{i}(\chi(0)) - \chi(0)) > G(0) = \chi(0).\label{lem_claim_1} 
\end{equation}
Once again, by \eqref{F_{i}_claim_1} and as $c_{i+1}$ is the fixed point of $F_{i+1}$ in $[0,c_{i}]$, we know that the curve $y=F_{i+1}(x)$ lies strictly above the line $y=x$ for $x \in [0,c_{i+1})$, and strictly below the line $y=x$ for $x \in (c_{i+1},c_{i}]$. This, along with \eqref{lem_claim_1}, implies that $\chi(0) \in [0,c_{i+1})$, thus completing the proof of \eqref{F_{i}_claim_3}. 
\end{proof}

\begin{proof}[Proof of Lemma~\ref{lem:main_thm_1_3}]
\sloppy Recall that we wish to prove $p_{i,j,n} = g_{i+1}(F_{j-i-1}(\nl^{(n)}), F_{j-i}(\nl^{(n)}), F_{k-i+1}(\nl^{(n)}), \ldots,\\ F_{k}(\nl^{(n)}))$ for $0 \leqslant i < j \leqslant k$, where $F_{i}$s and $g_{i}$s are as defined in \eqref{F_{i}_defn} and \eqref{g_{i}_recursive_defn} respectively. We prove this via induction on $i$. 

First, we note that for $i=0$ and $j = 1$, the claim follows from \eqref{c_{0,1,n}_recursion}. Suppose the claim holds for $p_{0,\ell,n}$ for $1 \leqslant \ell \leqslant j-1$, for some $2 \leqslant j \leqslant k$. From \eqref{c_{0,j,n}_recursion}, the induction hypothesis and \eqref{F_{i}_defn}:
\begin{align}
p_{0,j,n} &= G\left(1 - \nl^{(n)} - \sum_{\ell=1}^{j-2}p_{0,\ell,n}\right) - G\left(1 - \nl^{(n)} - \sum_{\ell=1}^{j-1}p_{0,\ell,n}\right)\nonumber\\
&= G\left(1 - \nl^{(n)} - \sum_{\ell=1}^{j-2}(F_{\ell-1}(\nl^{(n)}) - F_{\ell}(\nl^{(n)}))\right) - G\left(1 - \nl^{(n)} - \sum_{\ell=1}^{j-1}(F_{\ell-1}(\nl^{(n)}) - F_{\ell}(\nl^{(n)}))\right)\nonumber\\
&= G\left(F_{j-2}(\nl^{(n)}) - \nl^{(n)}\right) - G\left(F_{j-1}(\nl^{(n)}) - \nl^{(n)}\right) \nonumber\\
&= F_{j-1}(\nl^{(n)}) - F_{j}(\nl^{(n)}) = g_{1}\left(F_{j-1}(\nl^{(n)}) - F_{j}(\nl^{(n)})\right),\nonumber
\end{align}
thus proving the claim for $i=0$ and \emph{all} $1 \leqslant j \leqslant k$.

Suppose the claim holds for all $p_{i-1,j,n}$ for all $i-1 < j \leqslant k$, for some $1 \leqslant i \leqslant k-1$. From \eqref{c_{i,j,n}_recursion} and the induction hypothesis, for any $j$ with $i < j \leqslant k$, we see that 
\begin{align}
p_{i,j,n} ={}& G\left(\sum_{\ell=j-1}^{k}p_{i-1,\ell,n}\right) - G\left(\sum_{\ell=j}^{k}p_{i-1,\ell,n}\right)\nonumber\\
\begin{split}
={}& G\left(\sum_{\ell=j-1}^{k}g_{i}\left(F_{\ell-i}(\nl^{(n)}), F_{\ell-i+1}(\nl^{(n)}), F_{k-i+2}(\nl^{(n)}), \ldots, F_{k}(\nl^{(n)})\right)\right)\\ &- G\left(\sum_{\ell=j}^{k}g_{i}\left(F_{\ell-i}(\nl^{(n)}), F_{\ell-i+1}(\nl^{(n)}), F_{k-i+2}(\nl^{(n)}), \ldots, F_{k}(\nl^{(n)})\right)\right).\label{claim_10}
\end{split}
\end{align}
We analyze the two terms in \eqref{claim_10} individually, to avoid cluttering. Using \eqref{g_{i}_recursive_defn}, we have
\begin{align}
\MoveEqLeft[3] G\left(\sum_{\ell=j-1}^{k}g_{i}\left(F_{\ell-i}(\nl^{(n)}), F_{\ell-i+1}(\nl^{(n)}), F_{k-i+2}(\nl^{(n)}), \ldots, F_{k}(\nl^{(n)})\right)\right)\nonumber\\
\begin{split}
={}& G\Bigg(\sum_{\ell=j-1}^{k}G\left(g_{i-1}\left(F_{\ell-i}(\nl^{(n)}), F_{k-i+2}(\nl^{(n)}), \ldots, F_{k}(\nl^{(n)})\right)\right) \nonumber\\&- G\left(g_{i-1}\left(F_{\ell-i+1}(\nl^{(n)}), F_{k-i+2}(\nl^{(n)}), \ldots, F_{k}(\nl^{(n)})\right)\right)\Bigg)
\end{split}\nonumber\\
\begin{split}
={}& G\Bigg(G\left(g_{i-1}\left(F_{j-1-i}(\nl^{(n)}), F_{k-i+2}(\nl^{(n)}), \ldots, F_{k}(\nl^{(n)})\right)\right) \nonumber\\&- G\left(g_{i-1}\left(F_{k-i+1}(\nl^{(n)}), F_{k-i+2}(\nl^{(n)}), \ldots, F_{k}(\nl^{(n)})\right)\right)\Bigg) 
\end{split}\nonumber\\
={}& G\left(g_{i}\left(F_{j-i-1}(\nl^{(n)}), F_{k-i+1}(\nl^{(n)}), F_{k-i+2}(\nl^{(n)}), \ldots, F_{k}(\nl^{(n)})\right)\right).\label{claim_11}
\end{align}
The second term of \eqref{claim_10} is analyzed as follows:
\begin{align}
\MoveEqLeft[3] G\left(\sum_{\ell=j}^{k}g_{i}\left(F_{\ell-i}(\nl^{(n)}), F_{\ell-i+1}(\nl^{(n)}), F_{k-i+2}(\nl^{(n)}), \ldots, F_{k}(\nl^{(n)})\right)\right)\nonumber\\
\begin{split}
={}& G\Bigg(\sum_{\ell=j}^{k}G\left(g_{i-1}\left(F_{\ell-i}(\nl^{(n)}), F_{k-i+2}(\nl^{(n)}), \ldots, F_{k}(\nl^{(n)})\right)\right) \nonumber\\&- G\left(g_{i-1}\left(F_{\ell-i+1}(\nl^{(n)}), F_{k-i+2}(\nl^{(n)}), \ldots, F_{k}(\nl^{(n)})\right)\right)\Bigg)
\end{split}\nonumber\\
\begin{split}
={}& G\Bigg(G\left(g_{i-1}\left(F_{j-i}(\nl^{(n)}), F_{k-i+2}(\nl^{(n)}), \ldots, F_{k}(\nl^{(n)})\right)\right) \nonumber\\&- G\left(g_{i-1}\left(F_{k-i+1}(\nl^{(n)}), F_{k-i+2}(\nl^{(n)}), \ldots, F_{k}(\nl^{(n)})\right)\right)\Bigg)
\end{split} \nonumber\\
={}& G\left(g_{i}\left(F_{j-i}(\nl^{(n)}), F_{k-i+1}(\nl^{(n)}), F_{k-i+2}(\nl^{(n)}), \ldots, F_{k}(\nl^{(n)})\right)\right).\label{claim_12}
\end{align}
Substituting from \eqref{claim_11} and \eqref{claim_12} in \eqref{claim_10} and using \eqref{g_{i}_recursive_defn}, we obtain
\begin{align}
\begin{split}
p_{i,j,n} ={}& G\left(g_{i}\left(F_{j-i-1}(\nl^{(n)}), F_{k-i+1}(\nl^{(n)}), F_{k-i+2}(\nl^{(n)}), \ldots, F_{k}(\nl^{(n)})\right)\right)\nonumber\\&- G\left(g_{i}\left(F_{j-i}(\nl^{(n)}), F_{k-i+1}(\nl^{(n)}), F_{k-i+2}(\nl^{(n)}), \ldots, F_{k}(\nl^{(n)})\right)\right)
\end{split}\nonumber\\
={}& g_{i+1}(F_{j-i-1}(\nl^{(n)}), F_{j-i}(\nl^{(n)}), F_{k-i+1}(\nl^{(n)}), F_{k-i+2}(\nl^{(n)}), \ldots, F_{k}(\nl^{(n)})),\nonumber
\end{align} 
thus completing the inductive step of the argument. This completes the proof of Lemma~\ref{lem:main_thm_1_3}.
\end{proof}

\begin{proof}[Proof of Equation~\eqref{belongs_to_D_{i}}]
Recall that we intend to show that for any $0 \leqslant i_{1} < i_{2} \leqslant k-j$ and all $x \in [0,c_{k-1}]$, 
\begin{align}
1 \geqslant g_{j}\big(F_{i_{1}}(x), F_{k-j+1}(x), F_{k-j+2}(x), \ldots, F_{k}(x)\big) \geqslant g_{j}\big(F_{i_{2}}(x), F_{k-j+1}(x), F_{k-j+2}(x), \ldots, F_{k}(x)\big) \geqslant 0.\nonumber
\end{align}
We prove this via induction on $j$. It is important to keep in mind that $G$ is strictly increasing on $[0,1]$ and that $G(x) \in [0,1]$ for all $x \in [0,1]$, as these fact is used repeatedly in what follows.

First, we prove, via induction on $i$, that $F_{i}(x) \geqslant F_{i+1}(x)$ for all $x \in [0,c_{i}]$, for each $i \in \mathbb{N}_{0}$. the base case of $i = 0$ is immediate as $F_{0}(x) = 1$ and $F_{1}(x) = G(1-x) \in [0,1]$. Suppose $F_{i}(x) \geqslant F_{i+1}(x)$ for all $x \in [0,c_{i}]$, for some $i \in \mathbb{N}_{0}$. By \eqref{F_{i}_defn} and the induction hypothesis, we have $F_{i+1}(x) = G(F_{i}(x)-x) \geqslant G(F_{i+1}(x)-x) = F_{i+2}(x)$ for all $x \in [0,c_{i+1}] \subset [0,c_{i}]$. When $0 \leqslant i_{1} < i_{2} < k$, we thus have 
\begin{equation}
g_{1}(F_{i_{1}}(x), F_{k}(x)) = F_{i_{1}}(x) - F_{k}(x) \geqslant F_{i_{2}}(x) - F_{k}(x) = g_{1}(F_{i_{2}}(x),F_{k}(x))\nonumber
\end{equation}
for all $x \in [0,c_{k-1}]$, thus establishing \eqref{belongs_to_D_{i}} for $j = 1$. 

We assume that \eqref{belongs_to_D_{i}} holds for some $j < k$, and now we prove it for $j+1$. For any $0 \leqslant i_{1} < i_{2} \leqslant k-(j+1)$ and $x \in [0,c_{k-1}]$, we obtain, using \eqref{g_{i}_recursive_defn},
\begin{align}
\MoveEqLeft[3] g_{j+1}\big(F_{i_{1}}(x), F_{k-j}(x), F_{k-j+1}(x), \ldots, F_{k}(x)\big) - g_{j+1}\big(F_{i_{2}}(x), F_{k-j}(x), F_{k-j+1}(x), \ldots, F_{k}(x)\big)\nonumber\\
\begin{split}
={}& \left[G\left(g_{j}\left(F_{i_{1}}(x), F_{k-j+1}(x), \ldots, F_{k}(x)\right)\right) - G\left(g_{j}\left(F_{k-j}(x), F_{k-j+1}(x), \ldots, F_{k}(x)\right)\right)\right]\nonumber\\& - \left[G\left(g_{j}\left(F_{i_{2}}(x), F_{k-j+1}(x), \ldots, F_{k}(x)\right)\right) - G\left(g_{j}\left(F_{k-j}(x), F_{k-j+1}(x), \ldots, F_{k}(x)\right)\right)\right]
\end{split}\nonumber\\
={}& G\left(g_{j}\left(F_{i_{1}}(x), F_{k-j+1}(x), \ldots, F_{k}(x)\right)\right) - G\left(g_{j}\left(F_{i_{2}}(x), F_{k-j+1}(x), \ldots, F_{k}(x)\right)\right)\nonumber
\end{align}
and this is non-negative by our induction hypothesis. For $0 \leqslant i \leqslant k-(j+1)$ and $x \in [0,c_{k-1}]$, using \eqref{g_{i}_recursive_defn},
\begin{align}
\MoveEqLeft[3] g_{j+1}\big(F_{i}(x), F_{k-j}(x), F_{k-j+1}(x), \ldots, F_{k}(x)\big)\nonumber\\
\begin{split}
={}& G\left(g_{j}\left(F_{i}(x), F_{k-j+1}(x), \ldots, F_{k}(x)\right)\right) - G\left(g_{j}\left(F_{k-j}(x), F_{k-j+1}(x), \ldots, F_{k}(x)\right)\right)
\end{split}\nonumber
\end{align}
and this is non-negative because our induction hypothesis guarantees that $g_{j}\left(F_{i}(x), F_{k-j+1}(x), \ldots, F_{k}(x)\right) \geqslant g_{j}\left(F_{k-j}(x), F_{k-j+1}(x), \ldots, F_{k}(x)\right)$ (by setting $i_{1}=i$ and $i_{2}=k-j$), and $G$ is increasing on $[0,1]$. It is also bounded above by $1$ since $G\left(g_{j}\left(F_{i}(x), F_{k-j+1}(x), \ldots, F_{k}(x)\right)\right) \leqslant 1$, since $G(x) \leqslant 1$ for all $x \in [0,1]$. This completes the proof of Equation~\eqref{belongs_to_D_{i}}.
\end{proof}

\subsection{Proofs of lemmas from \S\ref{sec:proof_of_main_4}}\label{appsec:proof_of_main_4}

\begin{proof}[Proof of Lemma~\ref{lem:draw_probab_limit_lemma_1}]
Recall that, given a sequence of functions $\{r_{i}\}_{0 \leqslant i \leqslant k}$ defined and differentiable on an interval $I$ of $\mathbb{R}$, with $(r_{i}(x), r_{k-j+1}(x), r_{k-j+2}(x), \ldots, r_{k}(x)) \in \mathcal{D}_{j}$ (defined as in \eqref{D_{i}_domain_defn}) for all $x \in I$ and all $0 \leqslant i < i+j \leqslant k$, Lemma~\ref{lem:draw_probab_limit_lemma_1} describes an expression for the derivative of $g_{k,\lambda}(r_{0}(x),r_{1}(x),\ldots,r_{k}(x))$ with respect to $x$, described via \eqref{g_{k}_derivative_general_Poisson}, \eqref{f_{k,0,lambda}_detailed} and \eqref{f_{k,i,lambda}_detailed}, and this is what we establish now. Here, we focus on the offspring distribution Poisson$(\lambda)$.

Recall that the pgf of Poisson$(\lambda)$ is $G_{\lambda}(x) = e^{\lambda(x-1)}$, so that $G'_{\lambda}(x) = \lambda G_{\lambda}(x)$. We prove the lemma using induction on $k$. When $k=2$, we have
\begin{align}
& \frac{d}{dx}g_{2,\lambda}(r_{0}(x),r_{1}(x),r_{2}(x)) = \frac{d}{dx}G_{\lambda}(r_{0}(x)-r_{2}(x)) - \frac{d}{dx}G_{\lambda}(r_{1}(x)-r_{2}(x))\nonumber\\
&= \lambda G_{\lambda}(r_{0}(x)-r_{2}(x))(r'_{0}(x)-r'_{2}(x)) - \lambda G_{\lambda}(r_{1}(x)-r_{2}(x))(r'_{1}(x)-r'_{2}(x)),\nonumber
\end{align}
which proves the base case for the induction. 

Suppose \eqref{g_{k}_derivative_general_Poisson}, \eqref{f_{k,0,lambda}_detailed} and \eqref{f_{k,i,lambda}_detailed} hold for some $k \geqslant 2$. Using the induction hypothesis and \eqref{g_{i}_recursive_defn}, we have
\begin{align}
\MoveEqLeft[3] \frac{d}{dx}g_{k+1,\lambda}(r_{0}(x),r_{1}(x),\ldots,r_{k+1}(x)) \nonumber\\
\begin{split}
={}& G'_{\lambda}\left(g_{k,\lambda}(r_{0}(x),r_{2}(x),\ldots,r_{k+1}(x))\right) \frac{d}{dx}g_{k,\lambda}\left(r_{0}(x),r_{2}(x),\ldots,r_{k+1}(x)\right) \nonumber\\&- G'_{\lambda}\left(g_{k,\lambda}(r_{1}(x),r_{2}(x),\ldots,r_{k+1}(x))\right) \frac{d}{dx}g_{k,\lambda}\left(r_{1}(x),r_{2}(x),\ldots,r_{k+1}(x)\right)
\end{split}\nonumber\\
\begin{split}
={}& \lambda^{k} G_{\lambda}\left(g_{k,\lambda}(r_{0}(x),r_{2}(x),\ldots,r_{k+1}(x))\right) f_{k,0,\lambda}\left(r_{0}(x),r_{2}(x),\ldots,r_{k+1}(x)\right) (r'_{0}(x) - r'_{k+1}(x))\nonumber\\& - \lambda^{k} G_{\lambda}\left(g_{k,\lambda}(r_{1}(x),r_{2}(x),\ldots,r_{k+1}(x))\right) f_{k,0,\lambda}\left(r_{1}(x),r_{2}(x),\ldots,r_{k+1}(x)\right) (r'_{1}(x) - r'_{k+1}(x))\nonumber\\& + \lambda^{k} \sum_{i=2}^{k} \Big[G_{\lambda}(g_{k,\lambda}(r_{0}(x),r_{2}(x),\ldots,r_{k+1}(x))) f_{k,i-1,\lambda}(r_{0}(x),r_{2}(x),\ldots,r_{k+1}(x)) \nonumber\\&- G_{\lambda}(g_{k,\lambda}(r_{1}(x),r_{2}(x),\ldots,r_{k+1}(x))) f_{k,i-1,\lambda}(r_{1}(x),r_{2}(x),\ldots,r_{k+1}(x))\Big](r'_{i}(x) - r'_{k+1}(x)) 
\end{split}\nonumber\\
\begin{split}
={}& \lambda^{k} G_{\lambda}\left(g_{k,\lambda}\left(r_{0}(x),r_{2}(x),\ldots,r_{k+1}(x)\right)\right) \prod_{t=1}^{k-1}G_{\lambda}\left(g_{t,\lambda}\left(r_{0}(x), r_{k-t+2}(x), r_{k-t+3}(x), \ldots, r_{k+1}(x)\right)\right) \nonumber\\& (r'_{0}(x) - r'_{k+1}(x)) \nonumber\\&- \lambda^{k} G_{\lambda}\left(g_{k,\lambda}\left(r_{1}(x),r_{2}(x),\ldots,r_{k+1}(x)\right)\right) \prod_{t=1}^{k-1}G_{\lambda}\left(g_{t,\lambda}\left(r_{1}(x),r_{k-t+2}(x),r_{k-t+3}(x),\ldots,r_{k+1}(x)\right)\right) \nonumber\\&(r'_{1}(x) - r'_{k+1}(x)) \nonumber\\& + \lambda^{k} \sum_{i=2}^{k} \Bigg[G_{\lambda}\left(g_{k,\lambda}\left(r_{0}(x),r_{2}(x),\ldots,r_{k+1}(x)\right)\right) \prod_{t=1}^{k-(i-1)}G_{\lambda}\Big(g_{t,\lambda}\big(r_{i}(x),r_{k-t+2}(x),r_{k-t+3}(x),\nonumber\\&\ldots,r_{k+1}(x)\big)\Big) \alpha_{k,i-1,\lambda}\left(r_{0}(x),r_{2}(x),\ldots,r_{k+1}(x)\right) \nonumber\\&- G_{\lambda}\left(g_{k,\lambda}\left(r_{1}(x),r_{2}(x),\ldots,r_{k+1}(x)\right)\right)\prod_{t=1}^{k-(i-1)}G_{\lambda}\left(g_{t,\lambda}\left(r_{i}(x),r_{k-t+2}(x),r_{k-t+3}(x),\ldots,r_{k+1}(x)\right)\right) \nonumber\\&\alpha_{k,i-1,\lambda}\left(r_{1}(x),r_{2}(x),\ldots,r_{k+1}(x)\right)\Bigg](r'_{i}(x) - r'_{k+1}(x)) 
\end{split}\nonumber\\
\begin{split}
={}& \lambda^{k} \prod_{t=1}^{(k+1)-1}G_{\lambda}\left(g_{t,\lambda}\left(r_{0}(x), r_{(k+1)-t+1}(x), r_{(k+1)-t+2}(x), \ldots, r_{k+1}(x)\right)\right) (r'_{0}(x) - r'_{k+1}(x))\nonumber\\& - \lambda^{k} \prod_{t=1}^{(k+1)-1}G_{\lambda}\left(g_{t,\lambda}\left(r_{1}(x),r_{(k+1)-t+1}(x),r_{(k+1)-t+2}(x),\ldots,r_{k+1}(x)\right)\right)(r'_{1}(x) - r'_{k+1}(x))\nonumber\\& + \lambda^{k} \sum_{i=2}^{k} \prod_{t=1}^{(k+1)-i}G_{\lambda}\left(g_{t,\lambda}\left(r_{i}(x),r_{(k+1)-t+1}(x),r_{(k+1)-t+2}(x),\ldots,r_{k+1}(x)\right)\right)\nonumber\\& \alpha_{k+1,i,\lambda}(r_{0}(x),\ldots,r_{k+1}(x)) (r'_{i}(x) - r'_{k+1}(x)),
\end{split} \nonumber
\end{align} 
where
\sloppy 
\begin{align}
\alpha_{k+1,i,\lambda}(r_{0}(x),\ldots,r_{k+1}(x)) &= G_{\lambda}(g_{k,\lambda}(r_{0}(x),r_{2}(x),\ldots,r_{k+1}(x))) \alpha_{k,i-1,\lambda}(r_{0}(x),r_{2}(x), \ldots,r_{k+1}(x)) \nonumber\\&- G_{\lambda}(g_{k,\lambda}(r_{1}(x),r_{2}(x),\ldots,r_{k+1}(x))) \alpha_{k,i-1,\lambda}(r_{1}(x),r_{2}(x),\ldots,r_{k+1}(x))\nonumber
\end{align}
is bounded by $a_{k+1,i} = 2a_{k,i-1}$ (using the induction hypothesis).
\end{proof}

\begin{proof}[Proof of Lemma~\ref{lem:F_{i}_derivatives}]
Recall that our objective is to prove that $F'_{i,\lambda}(x) = -\lambda\left[\sum_{t=1}^{i-1}\lambda^{i-t}\prod_{j=t}^{i-1}F_{j,\lambda}(x)+1\right]F_{i,\lambda}(x)$ for all $x \in (0,c_{i-1})$ and $i \in \mathbb{N}$. We prove this by induction on $i$. When $i=1$, we have $F'_{1,\lambda}(x) = -\lambda e^{-\lambda x} = -\lambda F_{1,\lambda}(x)$, proving the base case for the induction. 

Suppose the lemma holds for some $i \geqslant 1$. Then, from \eqref{F_{i}_defn} and the induction hypothesis, we have
\begin{align}
F'_{i+1,\lambda}(x) ={}& G'_{\lambda}(F_{i,\lambda}(x)-x) [F'_{i,\lambda}(x) - 1] \nonumber\\
={}& -\lambda G_{\lambda}(F_{i,\lambda}(x)-x)\left[\lambda\{\sum_{t=1}^{i-1}\lambda^{i-t}\prod_{j=t}^{i-1}F_{j,\lambda}(x)+1\}F_{i,\lambda}(x) + 1\right] \nonumber\\
={}&-\lambda F_{i+1,\lambda}(x) \left[\sum_{t=1}^{i-1}\lambda^{i-t+1} F_{i,\lambda}(x)\prod_{j=t}^{i-1}F_{j,\lambda}(x) + \lambda F_{i,\lambda}(x) + 1\right]\nonumber\\
={}& -\lambda \left[\sum_{t=1}^{i}\lambda^{(i+1)-t}\prod_{j=t}^{i}F_{j,\lambda}(x) + 1\right]F_{i+1,\lambda}(x),\nonumber 
\end{align}
thus completing the inductive proof.
\end{proof}

\subsection{Proofs of lemmas from \S\ref{sec:Poisson_k=2}}\label{appsec:Poisson_k=2}
\begin{proof}[Proof of Lemma~\ref{lem:lambda c_{2,lambda} behaviour}]
Recall that our objective here is to prove that $\eta_{\lambda} = \lambda c_{2,\lambda}$ is is strictly increasing for $\lambda \in (0,\lambda_{0})$ and strictly decreasing for $\lambda \in (\lambda_{0}, \infty)$, where $\lambda_{0} \approx 2.43634$. We accomplish this by showing that the derivative of $\eta_{\lambda}$ with respect to $\lambda$ is strictly positive, and we justify the differentiability of $\eta_{\lambda}$ with respect to $\lambda$ by recalling that $c_{2,\lambda}$ is differentiable with respect to $\lambda$, as proved in detail in \S\ref{subsec:proof_of_main_4_part_1}. 

To find an expression for the derivative of $\eta_{\lambda}$, we first find the derivative $c'_{2,\lambda}$ of $c_{2,\lambda}$ using the definition of $c_{2,\lambda}$ (i.e.\ $c_{2,\lambda} = F_{2,\lambda}(c_{2,\lambda})$). This yields 
\begin{align}
&c'_{2,\lambda} = \frac{d}{d\lambda}F_{2,\lambda}(c_{2,\lambda}) = \frac{d}{d\lambda}\left[\exp\left\{\lambda e^{-\lambda c_{2,\lambda}} - \lambda c_{2,\lambda} - \lambda\right\}\right]\nonumber\\
\implies & c'_{2,\lambda} = \left[e^{-\lambda c_{2,\lambda}} + \lambda\left\{-c_{2,\lambda} - \lambda c'_{2,\lambda}\right\}e^{-\lambda c_{2,\lambda}} - \left\{c_{2,\lambda} + \lambda c'_{2,\lambda}\right\} - 1\right]c_{2,\lambda}\nonumber\\
\implies & c'_{2,\lambda} = \frac{\left[e^{-\lambda c_{2,\lambda}} - c_{2,\lambda} - 1 - \lambda c_{2,\lambda} e^{-\lambda c_{2,\lambda}}\right] c_{2,\lambda}}{1 + \lambda^{2} c_{2,\lambda} e^{-\lambda c_{2,\lambda}} + \lambda c_{2,\lambda}},\label{c_{2,lambda}_derivative}
\end{align}

Next, we examine the behaviour of $c_{2,\lambda}e^{\lambda c_{2,\lambda}+1}$ as a function of $\lambda$. For all $\lambda > 0$, substituting from \eqref{c_{2,lambda}_derivative} and noting that $e^{-\lambda c_{2,\lambda}} < 1$, we have
\begin{align}
&\frac{d}{d\lambda}\left[c_{2,\lambda}e^{\lambda c_{2,\lambda}}\right] = \left[c'_{2,\lambda}\left(1 + \lambda c_{2,\lambda}\right) + c_{2,\lambda}^{2}\right]e^{\lambda c_{2,\lambda}}
= \frac{(e^{-\lambda c_{2,\lambda}} - 1 - \lambda c_{2,\lambda}) c_{2,\lambda} e^{\lambda c_{2,\lambda}}}{1 + \lambda^{2} c_{2,\lambda} e^{-\lambda c_{2,\lambda}} + \lambda c_{2,\lambda}} < 0.\nonumber
\end{align}
Hence $c_{2,\lambda}e^{\lambda c_{2,\lambda}}$, and consequently, $c_{2,\lambda} e^{\lambda c_{2,\lambda}+1}$, is strictly decreasing in $\lambda$. Recall that \eqref{lambda^{i}c_{k,lambda}_limit_various_i} applied to $k=2$ guarantees that $\lambda c_{2,\lambda} \rightarrow 0$ (and hence $c_{2,\lambda} \rightarrow 0$ as well) as $\lambda \rightarrow \infty$. Consequently, 
\begin{equation}
\lim_{\lambda \rightarrow \infty} c_{2,\lambda} e^{\lambda c_{2,\lambda}+1} = 0.\nonumber
\end{equation}
On the other hand,  
\begin{equation}
\lim_{\lambda \rightarrow 0} c_{2,\lambda} = \lim_{\lambda \rightarrow 0} F_{2,\lambda}(c_{2,\lambda}) = \lim_{\lambda \rightarrow 0}\exp\left\{\lambda\left(e^{-\lambda c_{2,\lambda}} - c_{2,\lambda} - 1\right)\right\} = 1 \implies \lim_{\lambda \rightarrow 0}c_{2,\lambda} e^{\lambda c_{2,\lambda}+1} = e > 1.\nonumber
\end{equation}
From the three observations made above, we conclude that there exists a unique $\lambda_{0}$ with $c_{2,\lambda_{0}} e^{\lambda_{0} c_{2,\lambda_{0}} + 1} = 1$. 

We now explore some useful characteristics of this $\lambda_{0}$. Using, again, $c_{2,\lambda} = F_{2,\lambda}(c_{2,\lambda})$, we have
\begin{align}
&c_{2,\lambda_{0}} = \exp\left\{\lambda_{0}\left(e^{-\lambda_{0} c_{2,\lambda_{0}}} - c_{2,\lambda_{0}} - 1\right)\right\}\nonumber\\
\implies & c_{2,\lambda_{0}} e^{\lambda_{0} c_{2,\lambda_{0}}+1} = \exp\left\{\lambda_{0} e^{-\lambda_{0} c_{2,\lambda_{0}}} - \lambda_{0} + 1\right\} = 1 \implies c_{2,\lambda_{0}} = \frac{1}{\lambda_{0}} \ln\left(\frac{\lambda_{0}}{\lambda_{0}-1}\right)\nonumber
\end{align}
where the second step follows from how $\lambda_{0}$ has been defined above. Using the above derivation, we obtain 
\begin{align}
c_{2,\lambda_{0}} e^{\lambda_{0} c_{2,\lambda_{0}} + 1} = \frac{1}{\lambda_{0}} \ln\left(\frac{\lambda_{0}}{\lambda_{0}-1}\right) e \cdot \frac{\lambda_{0}}{\lambda_{0}-1} = 1 \implies \lambda_{0} \approx 2.43634.\nonumber
\end{align}

We now come to the derivative of $\eta_{\lambda}$ with respect to $\lambda$, and we make use of the observations made above:
\begin{align}\label{lambda c_{2,lambda} derivative}
&\eta'_{\lambda} = c_{2,\lambda} + \lambda c'_{2,\lambda} = \frac{c_{2,\lambda}\left(1 + \lambda e^{-\lambda c_{2,\lambda}} - \lambda\right)}{1 + \lambda^{2} c_{2,\lambda} e^{-\lambda c_{2,\lambda}} + \lambda c_{2,\lambda}} = \frac{c_{2,\lambda} \ln\left(c_{2,\lambda} e^{\lambda c_{2,\lambda} + 1}\right)}{1 + \lambda^{2} c_{2,\lambda} e^{-\lambda c_{2,\lambda}} + \lambda c_{2,\lambda}}.
\end{align}
Since $c_{2,\lambda} e^{\lambda c_{2,\lambda}+1}$ is strictly decreasing and takes the value $1$ at $\lambda = \lambda_{0}$, we have $\ln\left(c_{2,\lambda} e^{\lambda c_{2,\lambda} + 1}\right) > 0$ for $\lambda < \lambda_{0}$ and $\ln\left(c_{2,\lambda} e^{\lambda c_{2,\lambda} + 1}\right) < 0$ for $\lambda > \lambda_{0}$. Hence $\eta'_{\lambda} > 0$ for $0 < \lambda < \lambda_{0}$ and $\eta'_{\lambda} < 0$ for $\lambda > \lambda_{0}$. Thus $\eta_{\lambda}$ is strictly increasing for $\lambda \in (0,\lambda_{0})$, strictly decreasing for $\lambda \in (\lambda_{0}, \infty)$, and its maximum value, attained at $\lambda_{0}$, is $\approx 0.52839925$. 
\end{proof}

Before we begin the proofs of Lemmas~\ref{convexity_lem_1} through \ref{convexity_lem_5}, we urge the reader to recall the definitions of $\alpha(x)$, $\beta(x)$, and $A_{i}$ for $1 \leqslant i \leqslant 8$, from \S\ref{sec:Poisson_k=2}.

\begin{proof}[Proof of Lemma~\ref{convexity_lem_1}]
Recall that our objective here is to prove, after taking out the common factor of $\{\alpha(x) - \beta(x)\}F_{2,\lambda}(x)(\lambda F_{1,\lambda}(x)+1)^{2}$, that 
\begin{equation}
\lambda^{2}F_{2,\lambda}(x)\{\alpha(x)-\beta(x)\} + \lambda F_{2,\lambda}(x) > 1\label{claim_16}
\end{equation}
for all $x \in (\gamma_{\lambda},c_{2,\lambda}]$ and all $\lambda \geqslant 2.5$. Differentiating the left side of this inequality with respect to $x$ gives
\begin{align}
\MoveEqLeft[3] \lambda^{2} F'_{2,\lambda}(x)\{\alpha(x) - \beta(x)\} + \lambda^{2}F_{2,\lambda}(x)\{\alpha'(x) - \beta'(x)\} + \lambda F'_{2,\lambda}(x)\nonumber\\
\begin{split}
={}& -\lambda^{3}\left(\lambda F_{1,\lambda}(x) + 1\right) F_{2,\lambda}(x)\{\alpha(x) - \beta(x)\} + \lambda^{2}F_{2,\lambda}(x)[\lambda^{2}\alpha(x)(\lambda F_{1,\lambda}(x) + 1) F_{2,\lambda}(x) \nonumber\\&- \lambda^{2}\beta(x)\{-F_{1,\lambda}(x) + (\lambda F_{1,\lambda}(x) + 1) F_{2,\lambda}(x)\}] - \lambda^{2}(\lambda F_{1,\lambda}(x) + 1) F_{2,\lambda}(x)
\end{split}\nonumber\\
={}& \lambda^{3}(\lambda F_{1,\lambda}(x) + 1) F_{2,\lambda}(x)\{\alpha(x) - \beta(x)\}[\lambda F_{2,\lambda}(x)-1] + \lambda^{3}F_{1,\lambda}(x)F_{2,\lambda}(x)\{\lambda\beta(x) - 1\} - \lambda^{2}F_{2,\lambda}(x).\label{claim_15}
\end{align}
In what follows, we show that this derivative is strictly negative for $x \in (\gamma_{\lambda},c_{2,\lambda}]$, for $\lambda > 0$.

Since $F_{2,\lambda}$ is strictly decreasing on $[0,c_{1,\lambda}]$ by Lemma~\ref{lem:main_thm_1_1}, and as $c_{2,\lambda}$ is the unique fixed point of $F_{2,\lambda}$, we have $F_{2,\lambda}(x) \geqslant x$ for all $0 \leqslant x \leqslant c_{2,\lambda}$. Thus 
\begin{equation}
\beta(x) = G_{\lambda}(F_{1,\lambda}(x) - F_{2,\lambda}(x)) \leqslant G_{\lambda}(F_{1,\lambda}(x) - x) = F_{2,\lambda}(x).\label{claim_13}
\end{equation}
Due to the strictly decreasing nature of $F_{2,\lambda}$, we also have 
\begin{equation}
F_{2,\lambda}(x) < F_{2,\lambda}(\gamma_{\lambda}) = \frac{1}{\lambda} \text{ for } x > \gamma_{\lambda},\label{claim_14}
\end{equation}
so that from \eqref{claim_13} and \eqref{claim_14}, we obtain
\begin{equation}
\lambda F_{2,\lambda}(x)-1 < 0 \text{ and } \lambda\beta(x) - 1 < 0 \text{ for all } x \in (\gamma_{\lambda},c_{2,\lambda}].\label{claim_27}
\end{equation}
Thus the derivative in \eqref{claim_15} strictly negative for $x \in (\gamma_{\lambda},c_{2,\lambda}]$, implying that the function on the left side of \eqref{claim_16} is strictly decreasing for $x \in (\gamma_{\lambda},c_{2,\lambda}]$. 

Utilizing the above finding, for all $\lambda > 0$ and $\gamma_{\lambda} < x \leqslant c_{2,\lambda}$, we have
\begin{align}
\lambda^{2}F_{2,\lambda}(x)\{\alpha(x)-\beta(x)\} + \lambda F_{2,\lambda}(x) &\geqslant \lambda^{2}F_{2,\lambda}(c_{2,\lambda})\{\alpha(c_{2,\lambda})-\beta(c_{2,\lambda})\} + \lambda F_{2,\lambda}(c_{2,\lambda}) \nonumber\\
&= \lambda^{2}c_{2,\lambda}\{e^{-\lambda c_{2,\lambda}} - c_{2,\lambda}\} + \lambda c_{2,\lambda} = \lambda \eta_{\lambda} e^{-\eta_{\lambda}} - \eta_{\lambda}^{2} + \eta_{\lambda}.\label{claim_17}
\end{align}
Differentiating the right side of \eqref{claim_17} with respect to $\lambda$ and substituting from \eqref{lambda c_{2,lambda} derivative}, we have
\begin{align}
&\eta_{\lambda} e^{-\eta_{\lambda}} + \lambda \eta'_{\lambda} e^{-\eta_{\lambda}} - \lambda \eta_{\lambda} \eta'_{\lambda} e^{-\eta_{\lambda}} - 2\eta_{\lambda}\eta'_{\lambda} + \eta'_{\lambda}\nonumber\\
&=\eta_{\lambda} e^{-\eta_{\lambda}} + \frac{\eta_{\lambda}\left(1 + \lambda e^{-\eta_{\lambda}} - \lambda\right)}{\lambda\left(1 + \lambda \eta_{\lambda} e^{-\eta_{\lambda}} + \eta_{\lambda}\right)}\left[\lambda e^{-\eta_{\lambda}} - \lambda \eta_{\lambda} e^{-\eta_{\lambda}} - 2\eta_{\lambda} + 1\right]\nonumber\\
&= \frac{\eta_{\lambda}[\lambda^{2} e^{-\eta_{\lambda}}\{e^{-\eta_{\lambda}} - 1 + \eta_{\lambda}\} + 2\lambda \eta_{\lambda}\{1 - e^{-\eta_{\lambda}}\} + \{3\lambda e^{-\eta_{\lambda}} - 2\eta_{\lambda} + 1 - \lambda\}]}{\lambda\left(1 + \lambda \eta_{\lambda} e^{-\eta_{\lambda}} + \eta_{\lambda}\right)}\nonumber\\
&\geqslant \frac{\eta_{\lambda}\{3 \lambda e^{-0.5284} - 2 \cdot 0.5284 + 1 - \lambda\}}{\lambda\left(1 + \lambda \eta_{\lambda} e^{-\eta_{\lambda}} + \eta_{\lambda}\right)} = \frac{\eta_{\lambda}\{0.7686 \lambda - 0.0568\}}{\lambda\left(1 + \lambda \eta_{\lambda} e^{-\eta_{\lambda}} + \eta_{\lambda}\right)} > 0 \text{ for all } \lambda \geqslant 2,\nonumber
\end{align}
where we use $e^{-x} - 1 \geqslant -x$ for all $x \geqslant 0$, and that $\eta_{\lambda} < 0.5284$ for all $\lambda > 0$ from Lemma~\ref{lem:lambda c_{2,lambda} behaviour}. Thus $\lambda \eta_{\lambda} e^{-\eta_{\lambda}} - \eta_{\lambda}^{2} + \eta_{\lambda}$ is strictly increasing for $\lambda \geqslant 2$. Since its value at $\lambda = 2.5$ is $\approx 1.0279 > 1$, hence we conclude that 
\begin{equation}
\lambda \eta_{\lambda} e^{-\eta_{\lambda}} - \eta_{\lambda}^{2} + \eta_{\lambda} > 1 \text{ for all } \lambda \geqslant 2.5.\label{claim_18}
\end{equation}
Combining \eqref{claim_17} and \eqref{claim_18}, we conclude that \eqref{claim_16} does hold for all $x \in (\gamma_{\lambda},c_{2,\lambda}]$ and $\lambda \geqslant 2.5$, as desired.
\end{proof}

\begin{proof}[Proof of Lemma~\ref{convexity_lem_2}]
Recall that our objective here is to show, after taking out the common factor of $\lambda \{\alpha(x) - \beta(x)\}F_{1,\lambda}(x) F_{2,\lambda}(x)$, that
\begin{equation}
\lambda \beta(x) (\lambda F_{1,\lambda}(x)+1) > 1\label{claim_19}
\end{equation}
for all $x \in (\gamma_{\lambda},c_{2,\lambda}]$ and $\lambda \geqslant 2$. Differentiating the left side of \eqref{claim_19} with respect to $x$ yields 
\begin{equation}\label{intermediate_1}
\lambda \beta'(x) (\lambda F_{1,\lambda}(x)+1) + \lambda^{2}\beta(x)F'_{1,\lambda}(x) = \lambda^{3}\beta(x)[F_{2,\lambda}(x)(\lambda F_{1,\lambda}(x)+1)^{2}-F_{1,\lambda}(x)(\lambda F_{1,\lambda}(x)+2)].
\end{equation}
For $x < c_{2,\lambda}$, we have 
\begin{align}\label{intermediate_2}
&\frac{d}{dx}[F_{2,\lambda}(x)(\lambda F_{1,\lambda}(x)+1)^{2}-F_{1,\lambda}(x)(\lambda F_{1,\lambda}(x)+2)]\nonumber\\
&=\lambda (\lambda F_{1,\lambda}(x) + 1)[F_{1,\lambda}(x)\{2 - \lambda^{2}F_{1,\lambda}(x)F_{2,\lambda}(x) - 4\lambda F_{2,\lambda}(x)\} - F_{2,\lambda}(x)]\nonumber\\
&\leqslant \lambda (\lambda F_{1,\lambda}(x) + 1)[F_{1,\lambda}(x)\{2 - \lambda^{2}F_{1,\lambda}(c_{2,\lambda})F_{2,\lambda}(c_{2,\lambda}) - 4\lambda F_{2,\lambda}(c_{2,\lambda})\} - F_{2,\lambda}(x)] \text{ (by Lemma~\ref{lem:main_thm_1_1})}\nonumber\\ 
&= \lambda (\lambda F_{1,\lambda}(x) + 1)[F_{1,\lambda}(x)\{2 - \lambda \eta_{\lambda} e^{-\eta_{\lambda}} - 4 \eta_{\lambda}\} - F_{2,\lambda}(x)].
\end{align}

We need to examine the behaviour of $2 - \lambda \eta_{\lambda} e^{-\eta_{\lambda}} - 4 \eta_{\lambda}$ as a function of $\lambda$. Substituting from \eqref{lambda c_{2,lambda} derivative},  
\begin{align}
& \frac{d}{d\lambda}[\lambda \eta_{\lambda} e^{-\eta_{\lambda}}] 
= \eta_{\lambda} e^{-\eta_{\lambda}}[1 + (1-\eta_{\lambda}) \frac{1 + \lambda e^{-\eta_{\lambda}} - \lambda}{1 + \lambda \eta_{\lambda} e^{-\eta_{\lambda}} + \eta_{\lambda}}] 
= \eta_{\lambda} e^{-\eta_{\lambda}} \cdot \frac{2 + \lambda e^{-\eta_{\lambda}} - \lambda + \lambda \eta_{\lambda}}{1 + \lambda \eta_{\lambda} e^{-\eta_{\lambda}} + \eta_{\lambda}} > 0\nonumber
\end{align}
since $e^{-\eta_{\lambda}} - 1 + \eta_{\lambda} \geqslant 0$. Thus $\lambda \eta_{\lambda} e^{-\eta_{\lambda}}$ is strictly increasing in $\lambda$. Since its value is $\approx 2.25080$ at $\lambda = 8.644$, we conclude that $\lambda \eta_{\lambda} e^{-\eta_{\lambda}} > 2$ for all $\lambda \geqslant 8.644$. This shows that $2 - \lambda \eta_{\lambda} e^{-\eta_{\lambda}} - 4 \eta_{\lambda} < 0$ for all $\lambda \geqslant 8.644$. 

On the other hand, 
\begin{align}
& \frac{d}{d\lambda}[\lambda \eta_{\lambda} e^{-\eta_{\lambda}} + 4\eta_{\lambda}] = \eta_{\lambda} e^{-\eta_{\lambda}} \cdot \frac{2 + \lambda e^{-\eta_{\lambda}} - \lambda + \lambda \eta_{\lambda}}{1 + \lambda \eta_{\lambda} e^{-\eta_{\lambda}} + \eta_{\lambda}} + \frac{4\eta_{\lambda}\left(1 + \lambda e^{-\eta_{\lambda}} - \lambda\right)}{\lambda\left(1 + \lambda \eta_{\lambda} e^{-\eta_{\lambda}} + \eta_{\lambda}\right)}\nonumber\\
&= \frac{\eta_{\lambda}}{\lambda\left(1 + \lambda \eta_{\lambda} e^{-\eta_{\lambda}} + \eta_{\lambda}\right)}[\{6\lambda e^{-\eta_{\lambda}} + 4 - 4\lambda\} + \lambda^{2}e^{-\eta_{\lambda}}\{e^{-\eta_{\lambda}} - 1 + \eta_{\lambda}\}]\nonumber\\
&\geqslant \frac{\eta_{\lambda}[\{6\lambda e^{-0.5284} + 4 - 4\lambda\} + \lambda^{2}e^{-\eta_{\lambda}}\{e^{-\eta_{\lambda}} - 1 + \eta_{\lambda}\}]}{\lambda\left(1 + \lambda \eta_{\lambda} e^{-\eta_{\lambda}} + \eta_{\lambda}\right)} \text{ (by Lemma~\ref{lem:lambda c_{2,lambda} behaviour})}\nonumber\\
&= \frac{\eta_{\lambda}[\{4 - 0.462715\lambda\} + \lambda^{2}e^{-\eta_{\lambda}}\{e^{-\eta_{\lambda}} - 1 + \eta_{\lambda}\}]}{\lambda\left(1 + \lambda \eta_{\lambda} e^{-\eta_{\lambda}} + \eta_{\lambda}\right)}\nonumber
\end{align}
which is non-negative for all $\lambda \leqslant 4/0.462715 \approx 8.644$. This shows that $\lambda \eta_{\lambda} e^{-\eta_{\lambda}} + 4\eta_{\lambda}$ is increasing for $\lambda \leqslant 8.644$, and at $\lambda = 2$ we have $4\eta_{\lambda} = 2.0957 > 2$. Therefore $\lambda \eta_{\lambda} e^{-\eta_{\lambda}} + 4\eta_{\lambda} > 2$, and hence $2 - \lambda \eta_{\lambda} e^{-\eta_{\lambda}} - 4 \eta_{\lambda} < 0$ for all $\lambda \leqslant 8.644$.

The above findings show that $2 - \lambda \eta_{\lambda} e^{-\eta_{\lambda}} - 4 \eta_{\lambda} < 0$ for all $\lambda \geqslant 2$, so that the expression on the right side of \eqref{intermediate_2} is strictly negative for all $\lambda \geqslant 2$. This implies that, via \eqref{intermediate_2}, that $F_{2,\lambda}(x)(\lambda F_{1,\lambda}(x)+1)^{2}-F_{1,\lambda}(x)(\lambda F_{1,\lambda}(x)+2)$ is strictly decreasing for $x \leqslant c_{2,\lambda}$ and $\lambda \geqslant 2$. In the next paragraph, we focus on the value of this function at $c_{2,\lambda}$, as a function of $\lambda$. 

When $x = c_{2,\lambda}$ and $\lambda \geqslant 2$, we have
\begin{align}
& F_{2,\lambda}(c_{2,\lambda})(\lambda F_{1,\lambda}(c_{2,\lambda})+1)^{2}-F_{1,\lambda}(c_{2,\lambda})(\lambda F_{1,\lambda}(c_{2,\lambda})+2)\nonumber\\
&< F_{2,\lambda}(c_{2,\lambda})(\lambda F_{1,\lambda}(c_{2,\lambda})+1)^{2}-F_{1,\lambda}(c_{2,\lambda})(\lambda F_{1,\lambda}(c_{2,\lambda})+1)\nonumber\\ 
&=(\lambda F_{1,\lambda}(c_{2,\lambda})+1)\{c_{2,\lambda}(\lambda F_{1,\lambda}(c_{2,\lambda})+1) - F_{1,\lambda}(c_{2,\lambda})\}, \text{ since } F_{2,\lambda}(c_{2,\lambda}) = c_{2,\lambda}; \nonumber\\
&\leqslant (\lambda F_{1,\lambda}(c_{2,\lambda})+1)\{0.5284 F_{1,\lambda}(c_{2,\lambda})+ c_{2,\lambda} - F_{1,\lambda}(c_{2,\lambda})\}, \text{ since } \eta_{\lambda} \leqslant 0.5284 \text{ by Lemma~\ref{lem:lambda c_{2,lambda} behaviour}};\nonumber\\
&= (\lambda F_{1,\lambda}(c_{2,\lambda})+1)\{-0.4716 F_{1,\lambda}(c_{2,\lambda}) + c_{2,\lambda}\}\nonumber\\
&= (\lambda F_{1,\lambda}(c_{2,\lambda})+1)\{-0.4716 e^{-\lambda c_{2,\lambda}} + c_{2,\lambda}\} \nonumber\\
&\leqslant (\lambda F_{1,\lambda}(c_{2,\lambda})+1)\{-0.4716 e^{-0.5284} + 0.2619\} < 0,\label{claim_21}
\end{align}
where in the last step we use 
\begin{enumerate}
\item the fact (proved in \S\ref{subsec:proof_of_main_4_part_1}) that $c_{2,\lambda}$ is strictly decreasing in $\lambda$, and hence $c_{2,\lambda} \leqslant c_{2,2} \approx 0.2619$ for $\lambda \geqslant 2$,
\item and, once again, that $\eta_{\lambda} \leqslant 0.5284$, by Lemma~\ref{lem:lambda c_{2,lambda} behaviour}.
\end{enumerate}

Next, we focus on the behaviour of $F_{2,\lambda}(x)(\lambda F_{1,\lambda}(x)+1)^{2}-F_{1,\lambda}(x)(\lambda F_{1,\lambda}(x)+2)$ at $x = \gamma_{\lambda}$, as a function of $\lambda$. Since $F_{2,\lambda}(\gamma_{\lambda}) = \frac{1}{\lambda}$ (by definition of $\gamma_{\lambda}$ in \eqref{gamma_delta_defn}), we have 
\begin{equation}
-F_{1,\lambda}(\gamma_{\lambda})(\lambda F_{1,\lambda}(\gamma_{\lambda})+2) + F_{2,\lambda}(\gamma_{\lambda})(\lambda F_{1,\lambda}(\gamma_{\lambda})+1)^{2} = \frac{1}{\lambda} > 0.\label{claim_22}
\end{equation}

For each $\lambda \geqslant 2$, since we have proved above that $F_{2,\lambda}(x)(\lambda F_{1,\lambda}(x)+1)^{2}-F_{1,\lambda}(x)(\lambda F_{1,\lambda}(x)+2)$ is strictly decreasing when $x \in [0,c_{2,\lambda}]$, from \eqref{claim_21} and \eqref{claim_22} we conclude that $F_{2,\lambda}(x)(\lambda F_{1,\lambda}(x)+1)^{2}-F_{1,\lambda}(x)(\lambda F_{1,\lambda}(x)+2)$ is initially strictly positive, and then strictly negative, on the interval $(\gamma_{\lambda},c_{2,\lambda}]$. This, along with \eqref{intermediate_1}, implies that $\lambda \beta(x) (\lambda F_{1,\lambda}(x)+1)$ is initially strictly increasing, and then strictly decreasing, on $(\gamma_{\lambda},c_{2,\lambda}]$. Therefore, for each $\lambda \geqslant 2$,
\begin{multline}
\lambda \beta(x) (\lambda F_{1,\lambda}(x)+1) \geqslant \min\left\{\lambda \beta(c_{2,\lambda}) (\lambda F_{1,\lambda}(c_{2,\lambda})+1), \lambda \beta(\gamma_{\lambda}) (\lambda F_{1,\lambda}(\gamma_{\lambda})+1)\right\} \text{ for all } x \in (\gamma_{\lambda},c_{2,\lambda}].\label{claim_25}
\end{multline}

The idea, now, is to consider the values of $\lambda \beta(x) (\lambda F_{1,\lambda}(x)+1)$ at both $x = c_{2,\lambda}$ and $x = \gamma_{\lambda}$ to see which of them is the minimum, and then use this global minima of $\lambda \beta(x) (\lambda F_{1,\lambda}(x)+1)$ on $(\gamma_{\lambda},c_{2,\lambda}]$ to deduce that \eqref{claim_19} holds.

At the very outset of \S\ref{subsec:Poisson_k=2_part_2}, we have mentioned that $\beta(c_{2,\lambda}) = c_{2,\lambda}$. Substituting from \eqref{lambda c_{2,lambda} derivative}, the derivative of the value of $\lambda \beta(x) (\lambda F_{1,\lambda}(x)+1)$ at $x=c_{2,\lambda}$, with respect to $\lambda$, becomes
\begin{align}
\frac{d}{d\lambda}[\lambda \beta(c_{2,\lambda}) (\lambda F_{1,\lambda}(c_{2,\lambda})+1)] &= \frac{d}{d\lambda}\left[\lambda c_{2,\lambda} \left(\lambda e^{-\lambda c_{2,\lambda}} + 1\right)\right] = \frac{d}{d\lambda}\left[\eta_{\lambda}\left(\lambda e^{-\eta_{\lambda}}+1\right)\right]\nonumber\\
&= \eta_{\lambda}e^{-\eta_{\lambda}} +  \frac{\eta_{\lambda}\left(1 + \lambda e^{-\eta_{\lambda}} - \lambda\right)}{\lambda\left(1 + \lambda \eta_{\lambda} e^{-\eta_{\lambda}} + \eta_{\lambda}\right)}[\lambda e^{-\eta_{\lambda}} - \lambda \eta_{\lambda} e^{-\eta_{\lambda}} + 1]\nonumber\\
&= \frac{\eta_{\lambda}}{\lambda\left(1 + \lambda \eta_{\lambda} e^{-\eta_{\lambda}} + \eta_{\lambda}\right)}[\{3\lambda e^{-\eta_{\lambda}} + 1 - \lambda\} + \lambda^{2} e^{-\eta_{\lambda}}\{e^{-\eta_{\lambda}} - 1 + \eta_{\lambda}\}], \nonumber
\end{align}
and this is strictly positive since $3 e^{-\eta_{\lambda}} \geqslant 3 e^{-0.5284} \approx 1.7686 > 1$ for all $\lambda > 0$, by Lemma~\ref{lem:lambda c_{2,lambda} behaviour}. Thus $\lambda \beta(c_{2,\lambda}) (\lambda F_{1,\lambda}(c_{2,\lambda})+1)$ is strictly increasing in $\lambda$, and its value at $\lambda = 2$ is $\approx 1.14446 > 1$. Thus
\begin{equation}
\lambda \beta(c_{2,\lambda}) (\lambda F_{1,\lambda}(c_{2,\lambda})+1) \geqslant 1.14446 > 1 \text{ for all } \lambda \geqslant 2.\label{claim_23}
\end{equation}

We now need to examine the behaviour of the value of $\lambda \beta(x) (\lambda F_{1,\lambda}(x)+1)$ at $x=\gamma_{\lambda}$, as a function of $\lambda$. Via a similar application of the implicit function theorem as that used to justify the differentiability of $c_{k,\lambda}$ with respect to $\lambda$ in \S\ref{subsec:proof_of_main_4_part_1}, we conclude that $\gamma_{\lambda}$ is differentiable with respect to $\lambda$ as well. Writing $\theta_{\lambda} = \lambda \gamma_{\lambda}$, we then have
\begin{align}
F_{2,\lambda}(\gamma_{\lambda}) = \exp\{\lambda e^{-\theta_{\lambda}} - \theta_{\lambda} - \lambda\} = \frac{1}{\lambda} 
\implies & \theta'_{\lambda} = \frac{e^{-\theta_{\lambda}} - 1 + \frac{1}{\lambda}}{\lambda e^{-\theta_{\lambda}} + 1}.\nonumber
\end{align}
Differentiating the value of $\lambda \beta(x) (\lambda F_{1,\lambda}(x)+1)$ at $x=\gamma_{\lambda}$ with respect to $\lambda$ and substituting from above, we have
\begin{align}
\frac{d}{d\lambda}[\lambda^{2} \beta(\gamma_{\lambda}) F_{1,\lambda}(\gamma_{\lambda}) + \lambda \beta(\gamma_{\lambda})] 
={}& \frac{d}{d\lambda}[\lambda^{2} e^{\lambda e^{-\theta_{\lambda}} - 1 - \lambda - \theta_{\lambda}} + \lambda e^{\lambda e^{-\theta_{\lambda}} - 1 - \lambda}]\nonumber\\
\begin{split}
={}& 2\lambda \beta(\gamma_{\lambda})F_{1,\lambda}(\gamma_{\lambda}) + \lambda^{2}\{e^{-\theta_{\lambda}} - \lambda \theta'_{\lambda} e^{-\theta_{\lambda}} - 1 - \theta'_{\lambda}\}\beta(\gamma_{\lambda})F_{1,\lambda}(\gamma_{\lambda}) \nonumber\\&+  \beta(\gamma_{\lambda}) + \lambda\{e^{-\theta_{\lambda}} - \lambda \theta'_{\lambda} e^{-\theta_{\lambda}} - 1\}\beta(\gamma_{\lambda})
\end{split}\nonumber\\
={}& \lambda \beta(\gamma_{\lambda})F_{1,\lambda}(\gamma_{\lambda}) + \beta(\gamma_{\lambda}) - \frac{\lambda \beta(\gamma_{\lambda})}{\lambda e^{-\theta_{\lambda}} + 1}\nonumber\\
={}& \lambda \beta(\gamma_{\lambda}) \cdot \frac{\lambda e^{-2\theta_{\lambda}} + e^{-\theta_{\lambda}} - 1}{\lambda e^{-\theta_{\lambda}} + 1} + \beta(\gamma_{\lambda})\nonumber\\
>{}& \lambda \beta(\gamma_{\lambda}) \cdot \frac{\lambda e^{-2\eta_{\lambda}} + e^{-\eta_{\lambda}} - 1}{\lambda e^{-\theta_{\lambda}} + 1} + \beta(\gamma_{\lambda}) \nonumber\\
\geqslant{}& \lambda \beta(\gamma_{\lambda}) \cdot \frac{2 e^{-2 \cdot 0.5284} + e^{-0.5284} - 1}{\lambda e^{-\theta_{\lambda}} + 1} + \beta(\gamma_{\lambda})> 0,\nonumber
\end{align}
for $\lambda \geqslant 2$, where we use $\gamma_{\lambda} < c_{2,\lambda} \implies \theta_{\lambda} < \eta_{\lambda} < 0.5284$ (from Lemma~\ref{lem:lambda c_{2,lambda} behaviour}). This tells us that $\lambda \beta(\gamma_{\lambda}) (\lambda F_{1,\lambda}(\gamma_{\lambda}) + 1)$ is strictly increasing for $\lambda \geqslant 2$, and its value at $\lambda = 2$ is $\approx 1.20824 > 1$. Thus
\begin{equation}
\lambda \beta(\gamma_{\lambda}) (\lambda F_{1,\lambda}(\gamma_{\lambda}) + 1) \geqslant 1.20824 > 1 \text{ for all } \lambda \geqslant 2.\label{claim_24}
\end{equation}

From \eqref{claim_25}, \eqref{claim_23} and \eqref{claim_24}, we conclude that \eqref{claim_19} holds for all $x \in (\gamma_{\lambda},c_{2,\lambda}]$, for each $\lambda \geqslant 2$, completing our proof.
\end{proof}

\begin{proof}[Proof of Lemma~\ref{convexity_lem_3}]
Recall that, after setting aside the common factor of $\beta(x) F_{1,\lambda}(x) (\lambda F_{1,\lambda}(x) + 1)$, our objective here is to prove that 
\begin{equation}\label{claim_26}
\lambda^{2} F_{2,\lambda}(x)\{\alpha(x) - \beta(x)\} + 2\lambda F_{2,\lambda}(x) > 1
\end{equation}
for $x \in (\gamma_{\lambda},c_{2,\lambda}]$ and $\lambda \geqslant 2$. Differentiating the left side with respect to $x$, we have
\begin{align}
&\lambda^{3}F_{2,\lambda}(x)(\lambda F_{1,\lambda}(x)+1)\{\alpha(x) - \beta(x)\}[\lambda F_{2,\lambda}(x) - 1] + \lambda^{2}F_{2,\lambda}(x)[\lambda F_{1,\lambda}(x)\{\lambda \beta(x) - 2\} - 2],\nonumber
\end{align}
and using \eqref{claim_27}, we conclude that the above is strictly negative whenever $x > \gamma_{\lambda}$. Thus $\lambda^{2} F_{2,\lambda}(x)\{\alpha(x) - \beta(x)\} + 2\lambda F_{2,\lambda}(x)$ is strictly decreasing for $x \in [\gamma_{\lambda},c_{2,\lambda})$, and its minimum is attained at $c_{2,\lambda}$. In other words,
\begin{equation}\label{claim_28}
\lambda^{2} F_{2,\lambda}(x)\{\alpha(x) - \beta(x)\} + 2\lambda F_{2,\lambda}(x) \geqslant \lambda^{2}c_{2,\lambda}\{e^{-\lambda c_{2,\lambda}} - c_{2,\lambda}\} + 2\lambda c_{2,\lambda} \text{ for all } x \in [\gamma_{\lambda},c_{2,\lambda}).
\end{equation}

Differentiating this minima with respect to $\lambda$ and substituting from \eqref{lambda c_{2,lambda} derivative} yield
\begin{align}
& \frac{d}{d\lambda}[\lambda^{2}c_{2,\lambda}\{e^{-\lambda c_{2,\lambda}} - c_{2,\lambda}\} + 2\lambda c_{2,\lambda}] = \frac{d}{d\lambda}[\lambda \eta_{\lambda} e^{-\eta_{\lambda}} - \eta_{\lambda}^{2} + 2\eta_{\lambda}]\nonumber\\
&= \frac{\eta_{\lambda}[2\lambda\{2 e^{-\eta_{\lambda}} - 1\} + 2\{1-\eta_{\lambda}\} + 2\lambda \eta_{\lambda}\{1 - e^{-\eta_{\lambda}}\} + \lambda^{2} e^{-\eta_{\lambda}}\{e^{-\eta_{\lambda}} - 1 + \eta_{\lambda}\}]}{\lambda\left(1 + \lambda \eta_{\lambda} e^{-\eta_{\lambda}} + \eta_{\lambda}\right)},\nonumber
\end{align}
and to show that this is strictly positive for $\lambda \geqslant 2$, it suffices to show that the first two summands in the numerator are both strictly positive. This follows since $\eta_{\lambda} < 0.5284$, from Lemma~\ref{lem:lambda c_{2,lambda} behaviour}. Thus $\lambda^{2}c_{2,\lambda}\{e^{-\lambda c_{2,\lambda}} - c_{2,\lambda}\} + 2\lambda c_{2,\lambda}$ is strictly increasing in $\lambda$. Its value at $\lambda = 2$ is $\approx 1.3939 > 1$, so that we conclude that
\begin{equation}
\lambda^{2}c_{2,\lambda}\{e^{-\lambda c_{2,\lambda}} - c_{2,\lambda}\} + 2\lambda c_{2,\lambda} > 1 \text{ for all } \lambda \geqslant 2.\label{claim_29}
\end{equation}
From \eqref{claim_28} and \eqref{claim_29}, we conclude that \eqref{claim_26} holds, and the proof is complete.
\end{proof}

\begin{proof}[Proof of Lemma~\ref{convexity_lem_4}]
It is important, in order to understand the inequalities we need to establish in this proof, to revisit the proofs of Lemmas~\ref{convexity_lem_1}, \ref{convexity_lem_2} and \ref{convexity_lem_3}, for all $\lambda \geqslant 2$ and $x \in (\gamma_{\lambda},c_{2,\lambda}]$. Note that the entire argument outlined in the proof of Lemma~\ref{convexity_lem_1} extends to $\lambda \geqslant 2$, with the only difference being that we now consider the value of $\lambda \eta_{\lambda} e^{-\eta_{\lambda}} - \eta_{\lambda}^{2} + \eta_{\lambda}$ at $\lambda = 2$ instead of at $\lambda = 2.5$, and this value is $\approx 0.869957$. Consequently, from \eqref{claim_17} and the fact (shown in the proof of Lemma~\ref{convexity_lem_1}) that $\lambda \eta_{\lambda} e^{-\eta_{\lambda}} - \eta_{\lambda}^{2} + \eta_{\lambda}$ is strictly increasing for $\lambda \geqslant 2$, we conclude that
\begin{equation}
\lambda^{2}F_{2,\lambda}(x)\{\alpha(x)-\beta(x)\} + \lambda F_{2,\lambda}(x) \geqslant 0.869957 \text{ for all } x \in (\gamma_{\lambda},c_{2,\lambda}], \text{ for all } \lambda \geqslant 2.\label{claim_30}
\end{equation}
This leads to
\begin{align}\label{claim_31}
\begin{split}
&\lambda^{2}\{\alpha(x) - \beta(x)\}^{2}(\lambda F_{1,\lambda}(x) + 1)^{2}(F_{2,\lambda}(x))^{2} + \lambda \{\alpha(x) - \beta(x)\}(\lambda F_{1,\lambda}(x) + 1)^{2}(F_{2,\lambda}(x))^{2} \nonumber\\&- \{\alpha(x) - \beta(x)\}(\lambda F_{1,\lambda}(x)+1)^{2}F_{2,\lambda}(x)
\end{split}\nonumber\\
={}& \{\alpha(x) - \beta(x)\}^{2}(\lambda F_{1,\lambda}(x) + 1)^{2}F_{2,\lambda}(x) [\lambda^{2}F_{2,\lambda}(x)\{\alpha(x)-\beta(x)\} + \lambda F_{2,\lambda}(x) - 1]\nonumber\\
\geqslant {}& -0.130043\{\alpha(x) - \beta(x)\}(\lambda F_{1,\lambda}(x) + 1)^{2}F_{2,\lambda}(x) \text{ for all } x \in (\gamma_{\lambda},c_{2,\lambda}], \text{ for } \lambda \geqslant 2.
\end{align}

From \eqref{claim_25}, \eqref{claim_23} and \eqref{claim_24} in the proof of Lemma~\ref{convexity_lem_2}, we have $\lambda \beta(x) (\lambda F_{1,\lambda}(x)+1) \geqslant 1.14446$ for all $\lambda \geqslant 2$, so that
\begin{align}
& \lambda^{2}\beta(x)\{\alpha(x)-\beta(x)\}(\lambda F_{1,\lambda}(x)+1) F_{1,\lambda}(x) F_{2,\lambda}(x) - \lambda \{\alpha(x) - \beta(x)\} F_{1,\lambda}(x) F_{2,\lambda}(x)\nonumber\\
&= \lambda \{\alpha(x) - \beta(x)\} F_{1,\lambda}(x) F_{2,\lambda}(x)[\lambda \beta(x) (\lambda F_{1,\lambda}(x)+1) - 1]\nonumber\\
&\geqslant 0.14446 \lambda \{\alpha(x) - \beta(x)\} F_{1,\lambda}(x) F_{2,\lambda}(x) \text{ for all } x \in (\gamma_{\lambda},c_{2,\lambda}], \text{ for } \lambda \geqslant 2.\label{claim_32}
\end{align}

Finally, from the proof of Lemma~\ref{convexity_lem_3}, we obtain $\lambda^{2} F_{2,\lambda}(x)\{\alpha(x) - \beta(x)\} + 2\lambda F_{2,\lambda}(x) \geqslant 1.3939$, so that
\begin{align}
\begin{split}
&\lambda^{2} \beta(x) \{\alpha(x) - \beta(x)\}\left(\lambda F_{1,\lambda}(x) + 1\right) F_{1,\lambda}(x) F_{2,\lambda}(x) + 2\lambda \beta(x) F_{1,\lambda}(x) F_{2,\lambda}(x) \left(\lambda F_{1,\lambda}(x) + 1\right) \nonumber\\&- \beta(x)F_{1,\lambda}(x)(\lambda F_{1,\lambda}(x) + 1)
\end{split}\nonumber\\
={}& \beta(x)F_{1,\lambda}(x)(\lambda F_{1,\lambda}(x) + 1)[\lambda^{2} \{\alpha(x) - \beta(x)\} F_{2,\lambda}(x) + 2 \lambda F_{2,\lambda}(x) - 1]\nonumber\\
\geqslant{}& 0.3939 \beta(x)F_{1,\lambda}(x)(\lambda F_{1,\lambda}(x) + 1) \text{ for all } x \in (\gamma_{\lambda},c_{2,\lambda}], \text{ for } \lambda \geqslant 2.\label{claim_33}
\end{align}

To prove Lemma~\ref{convexity_lem_4}, it thus suffices to show that for $\lambda \geqslant 2$ and $x \in (\gamma_{\lambda},c_{2,\lambda}]$, the expressions on the right sides of \eqref{claim_31}, \eqref{claim_32} and \eqref{claim_33} add up to a strictly positive quantity. In other words, we need to show, for all $x \in (\gamma_{\lambda},c_{2,\lambda}]$ and each $\lambda \in [2,2.5)$, that
\begin{multline}
0.14446 \lambda \{\alpha(x) - \beta(x)\} F_{1,\lambda}(x) F_{2,\lambda}(x) + 0.3939\beta(x)F_{1,\lambda}(x)(\lambda F_{1,\lambda}(x) + 1) \\> 0.130043\{\alpha(x) - \beta(x)\} (\lambda F_{1,\lambda}(x) + 1)^{2}F_{2,\lambda}(x).\nonumber  
\end{multline}
Separating the terms involving the factor $\alpha(x)$ from those that involve the factor $\beta(x)$, we find that this is equivalent to proving, for all $x \in (\gamma_{\lambda},c_{2,\lambda}]$, for all $\lambda \in [2,2.5)$,
\begin{align}\label{intermediate_3}
&\beta(x)[0.3939\lambda (F_{1,\lambda}(x))^{2} + 0.3939 F_{1,\lambda}(x) + 0.130043 \lambda^{2} (F_{1,\lambda}(x))^{2}F_{2,\lambda}(x) \nonumber\\&+ 0.115626\lambda F_{1,\lambda}(x)F_{2,\lambda}(x) + 0.130043 F_{2,\lambda}(x)] - \alpha(x)[0.130043 \lambda^{2} (F_{1,\lambda}(x))^{2}F_{2,\lambda}(x) \nonumber\\&+ 0.115626\lambda F_{1,\lambda}(x)F_{2,\lambda}(x) + 0.130043 F_{2,\lambda}(x)] > 0.
\end{align}

The idea now is to demarcate the terms in \eqref{intermediate_3} into a few different groups, and then show that the sum of the terms in each such group is strictly positive for $\lambda \geqslant 2$. In this endeavour, we make use of certain facts that we enumerate below:
\begin{enumerate}
\item Recall from Lemma~\ref{lem:lambda c_{2,lambda} behaviour} that $\eta_{\lambda}$ increases for $\lambda \in [2,2.43634]$ and decreases for $\lambda \in [2.43634,2.5]$, and $\eta_{2} \approx 0.523928 < 0.528322 \approx \eta_{2.5}$. Therefore, $\min\{\eta_{\lambda}: \lambda \in [2, 2.5)\} = \eta_{2} \approx 0.523928$.
\item Since, substituting from \eqref{lambda c_{2,lambda} derivative}, we have
\begin{equation}
\frac{d}{d\lambda}[\lambda e^{-\eta_{\lambda}} - \lambda] = \frac{e^{-\eta_{\lambda}} - 1 - \eta_{\lambda}}{1 + \lambda \eta_{\lambda} e^{-\eta_{\lambda}} + \eta_{\lambda}} < 0,\nonumber
\end{equation} 
we conclude that $e^{\lambda e^{-\eta_{\lambda}} - \lambda}$ is strictly decreasing as a function of $\lambda$ for all $\lambda > 0$. Its minimum value for $\lambda \in [2,2.5]$, attained at $\lambda = 2.5$, is $\approx 0.358432$.
\item Since, substituting from \eqref{lambda c_{2,lambda} derivative}, we have
\begin{equation} 
\frac{d}{d\lambda}[\lambda e^{-\eta_{\lambda}}] = \frac{e^{-\eta_{\lambda}}[1 + \lambda \eta_{\lambda}]}{1 + \lambda \eta_{\lambda} e^{-\eta_{\lambda}} + \eta_{\lambda}} > 0,\nonumber
\end{equation}
the function $\lambda e^{-\eta_{\lambda}}$ is strictly increasing in $\lambda$ for all $\lambda > 0$, so that its minimum value for $\lambda \in [2,2.5]$, attained at $\lambda = 2$, is $\approx 1.18438$.
\item We make use of the definitions of $\alpha(x)$ and $\beta(x)$ from \S\ref{sec:Poisson_k=2}, namely that 
\begin{equation}
\alpha(x) = e^{-\lambda F_{2,\lambda}(x)} \text{ and } \beta(x) = e^{\lambda F_{1,\lambda}(x) - \lambda F_{2,\lambda}(x) - \lambda}.\nonumber
\end{equation}
\item We also make use of the decreasing nature of $F_{2,\lambda}$ and $F_{1,\lambda}$ from Lemma~\ref{lem:main_thm_1_1}, that $c_{2,\lambda}$ is the fixed point of $F_{2,\lambda}$, and the definition of $\gamma_{\lambda}$ from \eqref{gamma_delta_defn}. In particular, we make use of the inequalities 
\begin{equation}
F_{2,\lambda}(\gamma_{\lambda}) > F_{2,\lambda}(x) \geqslant F_{2,\lambda}(c_{2,\lambda}) \text{ and } F_{1,\lambda}(x) = e^{-\lambda x} \geqslant e^{-\lambda c_{2,\lambda}} = e^{-\eta_{\lambda}}\nonumber
\end{equation}
for all $x \in (\gamma_{\lambda},c_{2,\lambda}]$.
\end{enumerate}

The first such group of terms from \eqref{intermediate_3} that we consider is as follows: for $x \in (\gamma_{\lambda},c_{2,\lambda}]$ and $\lambda \in [2,2.5)$,
\begin{align}\label{intermediate_4}
& \beta(x)[0.3939\lambda (F_{1,\lambda}(x))^{2} + 0.130043 \lambda^{2} (F_{1,\lambda}(x))^{2}F_{2,\lambda}(x)] - 0.130043 \lambda^{2} \alpha(x) (F_{1,\lambda}(x))^{2}F_{2,\lambda}(x)\nonumber\\
&\geqslant \lambda (F_{1,\lambda}(x))^{2}[0.3939 \beta(x) + 0.130043 \lambda \beta(x) F_{2,\lambda}(c_{2,\lambda}) - 0.130043 \lambda \alpha(x)F_{2,\lambda}(\gamma_{\lambda})]\nonumber\\
&= \lambda (F_{1,\lambda}(x))^{2}\left[0.3939 \beta(x) + 0.130043 \beta(x) \lambda c_{2,\lambda} - 0.130043 \lambda \alpha(x) \cdot \frac{1}{\lambda}\right]\nonumber\\
&= \lambda (F_{1,\lambda}(x))^{2}[0.3939 \beta(x) + 0.130043 \beta(x) \eta_{\lambda} - 0.130043 \alpha(x)]\nonumber\\
&\geqslant \lambda (F_{1,\lambda}(x))^{2}[0.3939 \beta(x) + 0.130043 \beta(x) \cdot 0.523928 - 0.130043 \alpha(x)]\nonumber\\
&= \lambda (F_{1,\lambda}(x))^{2} e^{-\lambda F_{2,\lambda}(x)}[0.462033 e^{\lambda F_{1,\lambda}(x) - \lambda} - 0.130043]\nonumber\\
&\geqslant \lambda (F_{1,\lambda}(x))^{2} e^{-\lambda F_{2,\lambda}(x)}[0.462033 e^{\lambda e^{-\eta_{\lambda}} - \lambda} - 0.130043] \nonumber\\
&\geqslant \lambda (F_{1,\lambda}(x))^{2} e^{-\lambda F_{2,\lambda}(x)}[0.462033 \cdot 0.358432 - 0.130043] = 0.035564 \lambda (F_{1,\lambda}(x))^{2} \alpha(x) > 0.
\end{align}

The second group of terms from \eqref{intermediate_3} that we consider, for $x \in (\gamma_{\lambda},c_{2,\lambda}]$ and $\lambda \in [2,2.5)$, is as follows:
\begin{align}\label{intermediate_5}
& \beta(x)[0.3939 F_{1,\lambda}(x) + 0.115626\lambda F_{1,\lambda}(x)F_{2,\lambda}(x)] - 0.115626\lambda \alpha(x) F_{1,\lambda}(x)F_{2,\lambda}(x)\nonumber\\
&\geqslant \beta(x)[0.3939 F_{1,\lambda}(x) + 0.115626\lambda F_{1,\lambda}(x)F_{2,\lambda}(c_{2,\lambda})] - 0.115626\lambda \alpha(x) F_{1,\lambda}(x)F_{2,\lambda}(\gamma_{\lambda})\nonumber\\
&= \beta(x)[0.3939 F_{1,\lambda}(x) + 0.115626\lambda c_{2,\lambda} F_{1,\lambda}(x)] - 0.115626\lambda \alpha(x) F_{1,\lambda}(x) \cdot \frac{1}{\lambda}\nonumber\\
&= \beta(x)[0.3939 F_{1,\lambda}(x) + 0.115626 \eta_{\lambda} F_{1,\lambda}(x)] - 0.115626 \alpha(x) F_{1,\lambda}(x)\nonumber\\
&\geqslant \beta(x)[0.3939 F_{1,\lambda}(x) + 0.115626 \cdot 0.523928 F_{1,\lambda}(x)] - 0.115626\alpha(x) F_{1,\lambda}(x)\nonumber\\
&= F_{1,\lambda}(x)e^{-\lambda F_{2,\lambda}(x)}[0.454479 e^{\lambda F_{1,\lambda}(x) - \lambda} - 0.115626]\nonumber\\
&\geqslant F_{1,\lambda}(x)e^{-\lambda F_{2,\lambda}(x)}[0.454479 e^{\lambda e^{-\eta_{\lambda}} - \lambda} - 0.115626]\nonumber\\
&\geqslant F_{1,\lambda}(x)e^{-\lambda F_{2,\lambda}(x)}[0.454479 \cdot 0.358432 - 0.115626] = 0.047274 F_{1,\lambda}(x) \alpha(x) > 0.
\end{align}

We now add the remaining terms of \eqref{intermediate_3} (i.e.\ the ones that have not been taken into account in the expressions on the left sides of \eqref{intermediate_4} and \eqref{intermediate_5}) with the expressions on the right sides of \eqref{intermediate_4} and \eqref{intermediate_5} to get
\begin{align}
&0.130043 \beta(x) F_{2,\lambda}(x) - 0.130043 \alpha(x) F_{2,\lambda}(x) + 0.035564 \lambda (F_{1,\lambda}(x))^{2}\alpha(x) + 0.047274 F_{1,\lambda}(x) \alpha(x) \nonumber\\
&= 0.130043 F_{2,\lambda}(x) \{\beta(x) - \alpha(x)\} + 0.035564 \lambda (F_{1,\lambda}(x))^{2}\alpha(x) + 0.047274 F_{1,\lambda}(x) \alpha(x) \nonumber\\
&= 0.130043 F_{2,\lambda}(x) e^{-\lambda F_{2,\lambda}(x)}[e^{\lambda F_{1,\lambda}(x) - \lambda} - 1] + 0.035564 \lambda (F_{1,\lambda}(x))^{2}\alpha(x) + 0.047274 F_{1,\lambda}(x) \alpha(x) \nonumber\\
&\geqslant 0.130043 F_{2,\lambda}(x) \alpha(x)[e^{\lambda e^{-\eta_{\lambda}} - \lambda} - 1] + 0.035564 \lambda (F_{1,\lambda}(x))^{2}\alpha(x) + 0.047274 F_{1,\lambda}(x) \alpha(x) \nonumber\\
&\geqslant 0.130043 F_{2,\lambda}(x) \alpha(x)[0.358432 - 1] + 0.035564 \lambda (F_{1,\lambda}(x))^{2}\alpha(x) + 0.047274 F_{1,\lambda}(x) \alpha(x) \nonumber\\
&= -0.083431 F_{2,\lambda}(x) \alpha(x) + 0.035564 \lambda (F_{1,\lambda}(x))^{2}\alpha(x) + 0.047274 F_{1,\lambda}(x) \alpha(x) \nonumber\\
&\geqslant -0.083431 F_{1,\lambda}(x) \alpha(x) + 0.035564 \lambda (F_{1,\lambda}(x))^{2}\alpha(x) + 0.047274 F_{1,\lambda}(x) \alpha(x) \nonumber\\
&\geqslant F_{1,\lambda}(x) \alpha(x)[0.035564 \lambda F_{1,\lambda}(x) - 0.036157] \nonumber\\
&\geqslant F_{1,\lambda}(x) \alpha(x)[0.035564 \lambda e^{-\eta_{\lambda}} - 0.036157] \nonumber\\
&\geqslant F_{1,\lambda}(x) \alpha(x)[0.035564 \cdot 1.18438 - 0.036157] = 0.005964 F_{1,\lambda}(x) \alpha(x) > 0,\label{intermediate_6}
\end{align}
where make use of the inequality $F_{2,\lambda}(x) \leqslant F_{1,\lambda}(x)$ for $x \in [0,c_{1,\lambda}]$ (as shown in the proof of \eqref{belongs_to_D_{i}}). From \eqref{intermediate_4}, \eqref{intermediate_5} and \eqref{intermediate_6}, we conclude that the inequality in \eqref{intermediate_3} does hold, and this brings us to the end of the proof. 
\end{proof}

\begin{proof}[Proof of Lemma~\ref{convexity_lem_5}]
Recall that, after taking out the common factor of $\lambda \{\alpha(x)-\beta(x)\} F_{1,\lambda}(x) F_{2,\lambda}(x)$, our objective here is to show that 
\begin{equation}
2\lambda \beta(x) (\lambda F_{1,\lambda}(x)+1) > 1 \text{ for } x \in [\delta_{\lambda}, \gamma_{\lambda}], \text{ for all } \lambda \geqslant 2.\label{claim_34}
\end{equation}
The idea now is to come up with suitable lower bounds on each of $\beta(x)$ and $\lambda F_{1,\lambda}(x)$, the first of which we accomplish by examining whether $\beta(x)$ exhibits monotonicity as a function of $x$. For the second, we recall from Lemma~\ref{lem:main_thm_1_1} that $F_{1,\lambda}$ is strictly decreasing on $[0,1]$, so that 
\begin{equation}\label{claim_35}
F_{1,\lambda}(x) \geqslant F_{1,\lambda}(\gamma_{\lambda}) \text{ for } x \in [\delta_{\lambda}, \gamma_{\lambda}].
\end{equation}

Recalling from \S\ref{sec:Poisson_k=2} the expression for the derivative of $\beta(x)$ with respect to $x$, we have, for $x \in [0,\gamma_{\lambda}]$:
\begin{align}
\beta'(x) &= \lambda^{2}\beta(x)\left\{-F_{1,\lambda}(x) + \left(\lambda F_{1,\lambda}(x) + 1\right) F_{2,\lambda}(x)\right\}\nonumber\\
&\geqslant \lambda^{2}\beta(x)[-F_{1,\lambda}(x) + (\lambda F_{1,\lambda}(x)+1)F_{2,\lambda}(\gamma_{\lambda})] = \lambda \beta(x) > 0,\nonumber
\end{align}
where we obtain the inequality since $F_{2,\lambda}$ is strictly decreasing on $[0,c_{1,\lambda}]$ (again from Lemma~\ref{lem:main_thm_1_1}). This shows that $\beta(x)$ is a strictly increasing function of $x$ for $x \in [0,\gamma_{\lambda}]$, so that 
\begin{equation}\label{claim_36}
\beta(x) \geqslant \beta(\delta_{\lambda}) \text{ for } x \in [\delta_{\lambda}, \gamma_{\lambda}].
\end{equation}

Pur next aim is to obtain suitable expressions for $F_{1,\lambda}(\gamma_{\lambda})$ and $\beta(\delta_{\lambda})$. Recalling the definitions of $\alpha(x)$ and $\beta(x)$ from \S\ref{sec:Poisson_k=2}, and the definitions of $\delta_{\lambda}$ and $\gamma_{\lambda}$ from \eqref{gamma_delta_defn}, we have: 
\begin{align}
&F_{2,\lambda}(\delta_{\lambda}) = \exp\{\lambda e^{-\lambda \delta_{\lambda}} - \lambda \delta_{\lambda} - \lambda\} = \frac{5}{4\lambda} \nonumber\\
\implies & \beta(\delta_{\lambda}) = e^{\lambda F_{1,\lambda}(\delta_{\lambda}) - \lambda F_{2,\lambda}(\delta_{\lambda}) - \lambda} = \exp\{\lambda e^{-\lambda \delta_{\lambda}} - \frac{5}{4} - \lambda\} = \frac{5}{4\lambda}e^{\lambda \delta_{\lambda} - \frac{5}{4}};\label{claim_37}
\end{align}
and
\begin{align}
&F_{2,\lambda}(\gamma_{\lambda}) = \exp\{\lambda F_{1,\lambda}(\gamma_{\lambda}) - \lambda \gamma_{\lambda} - \lambda\} = \frac{1}{\lambda} \implies \lambda F_{1,\lambda}(\gamma_{\lambda}) = \lambda \gamma_{\lambda} + \lambda - \ln \lambda.\label{claim_38}
\end{align}

Combining our findings from \eqref{claim_35}, \eqref{claim_36}, \eqref{claim_37} and \eqref{claim_38}, for $x \in (\delta_{\lambda},\gamma_{\lambda}]$ and $\lambda \geqslant 2$, we have
\begin{align}
2\lambda \beta(x) (\lambda F_{1,\lambda}(x)+1) &\geqslant 2\lambda \beta(\delta_{\lambda}) (\lambda F_{1,\lambda}(\gamma_{\lambda})+1) \nonumber\\
&= \frac{5}{2}e^{\lambda \delta_{\lambda} - \frac{5}{4}}(\lambda \gamma_{\lambda} + \lambda - \ln \lambda + 1) > \frac{5}{2}e^{- \frac{5}{4}}(\lambda - \ln \lambda + 1).\nonumber
\end{align}
Since $\lambda - \ln \lambda$ is strictly increasing for $\lambda > 1$, we have $\frac{5}{2}e^{- \frac{5}{4}}(\lambda - \ln \lambda + 1) \geqslant \frac{5}{2}e^{-\frac{5}{4}}(2 - \ln 2 + 1) \approx 1.6523 > 1$ for $\lambda \geqslant 2$. This completes the proof of \eqref{claim_34} and brings us to the end of the proof of Lemma~\ref{convexity_lem_5}.
\end{proof}

\subsection{Proof of Lemma~\ref{lem:decay_rate_nl_{2,lambda}_precursor}}\label{appsec:misere_normal_comparison}
Recall that our objective here is to prove, for given $c > 0$ and $i \in \mathbb{N}$ with $i \geqslant 2$, the inequality $H_{2,\lambda}(c \lambda^{-i}) < c \lambda^{-i}$ for all $\lambda$ sufficiently large.

To begin with, we note that $c_{2,\lambda} > c \lambda^{-i}$ for all $\lambda$ sufficiently large, since \eqref{lambda^{i}c_{k,lambda}_limit_various_i} tells us that $\lambda^{2} c_{2,\lambda}$, and hence $\lambda^{i} c_{2,\lambda}$ as well, approaches $\infty$ as $\lambda \rightarrow \infty$. 

From \eqref{F_{i,lambda}_extended}, and applying Taylor expansion, we have 
\begin{equation}
F_{1,\lambda}(c \lambda^{-i}) = e^{-c \lambda^{-i+1}} \text{ and } F_{2,\lambda}(c \lambda^{-i}) = \exp\{\lambda e^{-c \lambda^{-i+1}} - c \lambda^{-i+1} - \lambda\} = \exp\{-c \lambda^{-i+2} + O(\lambda^{-i+1})\}.\nonumber
\end{equation}
Note that the second expression yields
\begin{equation}
\lim_{\lambda \rightarrow \infty} F_{2,\lambda}(c \lambda^{-2}) = e^{-c} \text{ and } \lim_{\lambda \rightarrow \infty} F_{2,\lambda}(c \lambda^{-i}) = 1 \text{ for each } i \geqslant 3.\nonumber
\end{equation}
Thus, given $0 < \epsilon < e^{-c} < 1$, we have, for all $\lambda$ sufficiently large, 
\begin{equation}
F_{2,\lambda}(c \lambda^{-i}) > M > 0, \text{ with } M = e^{-c}-\epsilon \text{ for } i=2 \text{ and } M = 1-\epsilon \text{ for } i \geqslant 3.\label{claim_39}
\end{equation}

Next, using the inequality $1 - e^{-x} \leqslant x$ for $x \geqslant 0$, we have 
\begin{equation}
1 - e^{\lambda F_{1,\lambda}(c \lambda^{-i}) - \lambda} \leqslant \lambda\left\{1-F_{1,\lambda}(c \lambda^{-i})\right\} = \lambda(1 - e^{-c \lambda^{-i+1}}) \leqslant c \lambda^{-i+2}.\label{claim_40}
\end{equation}
Thus, for all $\lambda$ sufficiently large, using \eqref{claim_39}, \eqref{claim_40} and the fact that $G_{\lambda}$ is increasing, we have
\begin{align}\label{one_last_eq}
H_{2,\lambda}\left(\frac{c}{\lambda^{i}}\right) &= G_{\lambda}\left[e^{-\lambda F_{2,\lambda}\left(\frac{c}{\lambda^{i}}\right)}\left\{1 - e^{\lambda F_{1,\lambda}\left(\frac{c}{\lambda^{i}}\right) - \lambda}\right\}\right] \nonumber\\
&< G_{\lambda}\left[e^{-\lambda M} \cdot \frac{c}{\lambda^{i-2}}\right] = \exp\left\{\frac{c}{\lambda^{i-3}} e^{-\lambda M} - \lambda\right\}.
\end{align} 
Since $i \geqslant 2$, we have $c \lambda^{-i+3} e^{-\lambda M} \leqslant c \lambda e^{-\lambda M} \rightarrow 0$, so that \eqref{one_last_eq} is bounded above by $e^{-\lambda/2}$ for all $\lambda$ sufficiently large, and this in turn is bounded above by $c \lambda^{-i}$ for all $\lambda$ sufficiently large (since $\lambda^{i} e^{-\lambda/2} \rightarrow 0$ as $\lambda \rightarrow \infty$). This establishes the desired inequality $H_{2,\lambda}(c \lambda^{-i}) < c \lambda^{-i}$ for all $\lambda$ sufficiently large.

\bibliography{Mult_GW_games}

\begin{thebibliography}{45}
\providecommand{\natexlab}[1]{#1}
\providecommand{\url}[1]{\texttt{#1}}
\expandafter\ifx\csname urlstyle\endcsname\relax
  \providecommand{\doi}[1]{doi: #1}\else
  \providecommand{\doi}{doi: \begingroup \urlstyle{rm}\Url}\fi

\bibitem[Aldous and Bandyopadhyay(2005)]{aldous2005survey}
David~J Aldous and Antar Bandyopadhyay.
\newblock A survey of max-type recursive distributional equations.
\newblock \emph{Annals of Applied Probability}, 15\penalty0 (2):\penalty0
  1047--1110, 2005.

\bibitem[Anderson~Jr(1974)]{slither}
William~N Anderson~Jr.
\newblock Maximum matching and the game of slither.
\newblock \emph{Journal of Combinatorial Theory, Series B}, 17\penalty0
  (3):\penalty0 234--239, 1974.

\bibitem[Athreya and Jagers(2012)]{athreya_jagers}
Krishna~B Athreya and Peter Jagers.
\newblock \emph{Classical and modern branching processes}, volume~84.
\newblock Springer Science \& Business Media, 2012.

\bibitem[Athreya and Ney(1972)]{athreya_ney}
Krishna~B Athreya and Peter~E Ney.
\newblock \emph{Branching processes}.
\newblock Springer, Berlin, Heidelberg, 1972.
\newblock \doi{https://doi.org/10.1007/978-3-642-65371-1}.

\bibitem[Athreya and Vidyashankar(2001)]{athreya_vidyashankar}
Krishna~B Athreya and AN~Vidyashankar.
\newblock Branching processes.
\newblock \emph{Stochastic processes: theory and methods}, 19:\penalty0 35--53,
  2001.

\bibitem[Basu et~al.(2016)Basu, Holroyd, Martin, W{\"a}stlund,
  et~al.]{trapping_games}
Riddhipratim Basu, Alexander~E Holroyd, James~B Martin, Johan W{\"a}stlund,
  et~al.
\newblock Trapping games on random boards.
\newblock \emph{The Annals of Applied Probability}, 26\penalty0 (6):\penalty0
  3727--3753, 2016.

\bibitem[Beveridge et~al.(2014)Beveridge, Dudek, Frieze, M{\"u}ller, and
  Stojakovi{\'c}]{maker_breaker_geometric}
Andrew Beveridge, Andrzej Dudek, Alan Frieze, Tobias M{\"u}ller, and
  Milo{\v{s}} Stojakovi{\'c}.
\newblock Maker-breaker games on random geometric graphs.
\newblock \emph{Random structures \& algorithms}, 45\penalty0 (4):\penalty0
  553--607, 2014.

\bibitem[Bienaym{\'e}(1845)]{bienayme}
Ir{\'e}n{\'e}e-Jules Bienaym{\'e}.
\newblock De la loi de multiplication et de la dur{\'e}e des familles.
\newblock \emph{Soc. Philomat. Paris Extraits, S{\'e}r}, 5\penalty0
  (37-39):\penalty0 4, 1845.

\bibitem[Bohman et~al.(2007)Bohman, Frieze, {\L}uczak, Pikhurko, Smyth,
  Spencer, and Verbitsky]{bohman}
Tom Bohman, Alan Frieze, Tomasz {\L}uczak, Oleg Pikhurko, Clifford Smyth, Joel
  Spencer, and Oleg Verbitsky.
\newblock First-order definability of trees and sparse random graphs.
\newblock \emph{Combinatorics, Probability and Computing}, 16\penalty0
  (3):\penalty0 375--400, 2007.

\bibitem[Chv{\'a}tal and Erd{\"o}s(1978)]{biased_positional}
Va{\v{s}}ek Chv{\'a}tal and Paul Erd{\"o}s.
\newblock Biased positional games.
\newblock In \emph{Annals of Discrete Mathematics}, volume~2, pages 221--229.
  Elsevier, 1978.

\bibitem[Ferber et~al.(2015)Ferber, Glebov, Krivelevich, and
  Naor]{biased_random_boards}
Asaf Ferber, Roman Glebov, Michael Krivelevich, and Alon Naor.
\newblock Biased games on random boards.
\newblock \emph{Random Structures \& Algorithms}, 46\penalty0 (4):\penalty0
  651--676, 2015.

\bibitem[Fraenkel(2012)]{survey_games}
Aviezri Fraenkel.
\newblock Combinatorial games: selected bibliography with a succinct gourmet
  introduction.
\newblock \emph{The Electronic Journal of Combinatorics}, pages DS2--Aug, 2012.

\bibitem[Fraenkel(2004)]{complexity_appeal}
Aviezri~S Fraenkel.
\newblock Complexity, appeal and challenges of combinatorial games.
\newblock \emph{Theoretical Computer Science}, 313\penalty0 (3):\penalty0
  393--415, 2004.

\bibitem[Hefetz et~al.(2014)Hefetz, Krivelevich, Stojakovi{\'c}, and
  Szab{\'o}]{positional_games_book}
Dan Hefetz, Michael Krivelevich, Milo{\v{s}} Stojakovi{\'c}, and Tibor
  Szab{\'o}.
\newblock \emph{Positional games}.
\newblock Springer, 2014.

\bibitem[Holroyd and Martin(2021)]{holroyd_martin}
Alexander~E Holroyd and James~B Martin.
\newblock Galton-watson games.
\newblock \emph{Random Structures \& Algorithms}, 59\penalty0 (4):\penalty0
  495--521, 2021.

\bibitem[Holroyd et~al.(2017)Holroyd, Levy, Podder, and Spencer]{podder_3}
Alexander~E Holroyd, Avi Levy, Moumanti Podder, and Joel Spencer.
\newblock Second order logic on random rooted trees.
\newblock \emph{Discrete Mathematics}, 342\penalty0 (1):\penalty0 152--167,
  2017.

\bibitem[Holroyd et~al.(2019)Holroyd, Marcovici, and Martin]{percolation_games}
Alexander~E Holroyd, Ir{\`e}ne Marcovici, and James~B Martin.
\newblock Percolation games, probabilistic cellular automata, and the hard-core
  model.
\newblock \emph{Probability Theory and Related Fields}, 174\penalty0
  (3):\penalty0 1187--1217, 2019.

\bibitem[Johnson et~al.(2020)Johnson, Podder, and Skerman]{podder_4}
Tobias Johnson, Moumanti Podder, and Fiona Skerman.
\newblock Random tree recursions: Which fixed points correspond to tangible
  sets of trees?
\newblock \emph{Random Structures \& Algorithms}, 56\penalty0 (3):\penalty0
  796--837, 2020.

\bibitem[Kim et~al.(2005)Kim, Pikhurko, Spencer, and Verbitsky]{kim}
Jeong~Han Kim, Oleg Pikhurko, Joel~H Spencer, and Oleg Verbitsky.
\newblock How complex are random graphs in first order logic?
\newblock \emph{Random Structures \& Algorithms}, 26\penalty0 (1-2):\penalty0
  119--145, 2005.

\bibitem[Kupavskii and Zhukovskii(2018)]{maksim_3}
Andrey Kupavskii and Maksim Zhukovskii.
\newblock Short monadic second order sentences about sparse random graphs.
\newblock \emph{SIAM Journal on Discrete Mathematics}, 32\penalty0
  (4):\penalty0 2916--2940, 2018.

\bibitem[Mach et~al.(2018)Mach, Sturm, and Swart]{mach2018new}
Tibor Mach, Anja Sturm, and Jan~M Swart.
\newblock A new characterization of endogeny.
\newblock \emph{Mathematical Physics, Analysis and Geometry}, 21\penalty0
  (4):\penalty0 30, 2018.

\bibitem[Mach et~al.(2020)Mach, Sturm, and Swart]{mach2020recursive}
Tibor Mach, Anja Sturm, and Jan~M Swart.
\newblock Recursive tree processes and the mean-field limit of stochastic
  flows.
\newblock \emph{Electronic Journal of Probability}, 25:\penalty0 1--63, 2020.

\bibitem[Martin and Stasi{\'n}ski(2020)]{martin2020minimax}
James~B Martin and Roman Stasi{\'n}ski.
\newblock Minimax functions on galton--watson trees.
\newblock \emph{Combinatorics, Probability and Computing}, 29\penalty0
  (3):\penalty0 455--484, 2020.

\bibitem[Matushkin and Zhukovskii(2018)]{maksim_4}
AD~Matushkin and ME~Zhukovskii.
\newblock First order sentences about random graphs: small number of
  alternations.
\newblock \emph{Discrete Applied Mathematics}, 236:\penalty0 329--346, 2018.

\bibitem[Ostrovsky and Zhukovskii(2017)]{maksim_1}
LB~Ostrovsky and ME~Zhukovskii.
\newblock Monadic second-order properties of very sparse random graphs.
\newblock \emph{Annals of pure and applied logic}, 168\penalty0 (11):\penalty0
  2087--2101, 2017.

\bibitem[Pikhurko et~al.(2006)Pikhurko, Veith, and Verbitsky]{pikhurko}
Oleg Pikhurko, Helmut Veith, and Oleg Verbitsky.
\newblock The first order definability of graphs: upper bounds for quantifier
  depth.
\newblock \emph{Discrete applied mathematics}, 154\penalty0 (17):\penalty0
  2511--2529, 2006.

\bibitem[Podder(2019)]{podder_5}
Moumanti Podder.
\newblock The first order theory of ${G}(n, c/n)$.
\newblock \emph{European Journal of Combinatorics}, 78:\penalty0 214--235,
  2019.

\bibitem[Podder and Spencer(2017{\natexlab{a}})]{podder_1}
Moumanti Podder and Joel Spencer.
\newblock First order probabilities for galton--watson trees.
\newblock In \emph{A Journey Through Discrete Mathematics}, pages 711--734.
  Springer, 2017{\natexlab{a}}.

\bibitem[Podder and Spencer(2017{\natexlab{b}})]{podder_2}
Moumanti Podder and Joel Spencer.
\newblock Galton-watson probability contraction.
\newblock \emph{Electronic Communications in Probability}, 22:\penalty0 Paper
  no.\ 20, 2017{\natexlab{b}}.

\bibitem[R{\'a}th et~al.(2021)R{\'a}th, Swart, and Terpai]{rath2021frozen}
Bal{\'a}zs R{\'a}th, Jan~M Swart, and Tam{\'a}s Terpai.
\newblock Frozen percolation on the binary tree is nonendogenous.
\newblock \emph{The Annals of Probability}, 49\penalty0 (5):\penalty0
  2272--2316, 2021.

\bibitem[R{\'a}th et~al.(2022)R{\'a}th, Swart, and Sz{\H{o}}ke]{rath2022phase}
Bal{\'a}zs R{\'a}th, Jan~M Swart, and M{\'a}rton Sz{\H{o}}ke.
\newblock A phase transition between endogeny and nonendogeny.
\newblock \emph{Electronic Journal of Probability}, 27:\penalty0 1--43, 2022.

\bibitem[Razafimahatratra and Zhukovskii(2020)]{maksim_6}
AS~Razafimahatratra and M~Zhukovskii.
\newblock Zero--one laws for k-variable first-order logic of sparse random
  graphs.
\newblock \emph{Discrete Applied Mathematics}, 276:\penalty0 121--128, 2020.

\bibitem[Simon et~al.(1994)Simon, Blume, et~al.]{simon_blume}
Carl~P Simon, Lawrence Blume, et~al.
\newblock \emph{Mathematics for economists}, volume~7.
\newblock Norton New York, 1994.

\bibitem[Spencer(1991)]{spencer_threshold}
Joel Spencer.
\newblock Threshold spectra via the ehrenfeucht game.
\newblock \emph{Discrete Applied Mathematics}, 30\penalty0 (2-3):\penalty0
  235--252, 1991.

\bibitem[Spencer and Thoma(1997)]{spencer_thoma}
Joel Spencer and Lubos Thoma.
\newblock On the limit values of probabilities for the first order properties
  of graphs.
\newblock \emph{Contemporary trends in discrete mathematics}, 49:\penalty0
  317--336, 1997.

\bibitem[Spencer and St~John(1998)]{spencer_stjohn}
Joel~H Spencer and Katherine St~John.
\newblock Random unary predicates: Almost sure theories and countable models.
\newblock \emph{Random Structures \& Algorithms}, 13\penalty0 (3-4):\penalty0
  229--248, 1998.

\bibitem[Stewart(2012)]{stewart2012essential}
James Stewart.
\newblock \emph{Essential calculus: Early transcendentals}.
\newblock Cengage Learning, 2012.

\bibitem[Stojakovi{\'c}(2014)]{milos_thesis}
Milo{\v{s}} Stojakovi{\'c}.
\newblock Games on graphs.
\newblock In \emph{International Conference on Conceptual Structures}, pages
  31--36. Springer, 2014.

\bibitem[Stojakovi{\'c} and Szab{\'o}(2005)]{milos_tibor}
Milo{\v{s}} Stojakovi{\'c} and Tibor Szab{\'o}.
\newblock Positional games on random graphs.
\newblock \emph{Random Structures \& Algorithms}, 26\penalty0 (1-2):\penalty0
  204--223, 2005.

\bibitem[Stojakovi{\'c} and Trkulja(2021)]{hamiltonian_maker_breaker}
Milo{\v{s}} Stojakovi{\'c} and Nikola Trkulja.
\newblock Hamiltonian maker--breaker games on small graphs.
\newblock \emph{Experimental Mathematics}, 30\penalty0 (1):\penalty0 595--604,
  2021.
\newblock \doi{10.1080/10586458.2019.1586599}.

\bibitem[Verbitsky(2005)]{verbitsky}
Oleg Verbitsky.
\newblock The first order definability of graphs with separators via the
  ehrenfeucht game.
\newblock \emph{Theoretical computer science}, 343\penalty0 (1-2):\penalty0
  158--176, 2005.

\bibitem[W{\"a}stlund(2012)]{wastlund}
Johan W{\"a}stlund.
\newblock Replica symmetry of the minimum matching.
\newblock \emph{Annals of Mathematics}, pages 1061--1091, 2012.

\bibitem[Watson and Galton(1875)]{galton_watson}
Henry~William Watson and Francis Galton.
\newblock On the probability of the extinction of families.
\newblock \emph{The Journal of the Anthropological Institute of Great Britain
  and Ireland}, 4:\penalty0 138--144, 1875.

\bibitem[Zhukovskii(2016)]{maksim_2}
ME~Zhukovskii.
\newblock On infinite spectra of first order properties of random graphs.
\newblock \emph{Moscow Journal of Combinatorics and Number Theory}, 6\penalty0
  (4):\penalty0 73--102, 2016.

\bibitem[Zhukovskii(2020)]{maksim_5}
ME~Zhukovskii.
\newblock Logical laws for short existential monadic second-order sentences
  about graphs.
\newblock \emph{Journal of Mathematical Logic}, 20\penalty0 (02):\penalty0
  2050007, 2020.

\end{thebibliography}
\end{document}